\documentclass[
a4paper,
12pt,
DIV=14,`1234%twoside,
toc=bibliography,
headsepline,
parskip=half-,
abstract=true,
]{scrartcl}
\usepackage{amsmath,amsfonts,amsthm,amssymb,nicefrac}
\usepackage{graphicx,color,url}
\usepackage{xcolor}
\usepackage{accents}
\usepackage{url,hyperref}
\usepackage{enumerate}
\usepackage{comment}
\usepackage{scrlayer-scrpage}

\newtheorem{theorem}{Theorem}
\newtheorem{lemma}[theorem]{Lemma}
\newtheorem{corollary}[theorem]{Corollary}
\newtheorem{proposition}[theorem]{Proposition}
\theoremstyle{definition}
\newtheorem{remark}[theorem]{Remark}

\newtheorem{innercustomthm}{Theorem}
\newenvironment{customthm}[1]
 {\renewcommand\theinnercustomthm{#1}\innercustomthm}
 {\endinnercustomthm}

\numberwithin{equation}{section}\numberwithin{theorem}{section}
\allowdisplaybreaks
\newcounter{stepctr}
{\end{list}}

\def\XXint#1#2#3{{\setbox0=\hbox{$#1{#2#3}{\int}$}
 \vcenter{\hbox{$#2#3$}}\kern-.5\wd0}}

\newcommand{\circo}{\accentset{\circ}}
\newcommand{\ra}{\rangle}
\newcommand{\la}{\langle}
\newcommand{\e}{\varepsilon}

\DeclareMathOperator{\inj}{inj}
\DeclareMathOperator{\tr}{tr}
\providecommand{\titlemacro}{{Singularity Models for High Codimension MCF in Riemannian Manifolds}}
\title{\titlemacro}
\author{Huy T. Nguyen and Artemis A. Vogiatzi}
\ohead{\titlemacro}
\ihead{H.\;T.\;Nguyen, A.A.Vogiatzi}
\setkomafont{pageheadfoot}{\footnotesize}

\date{}
\begin{document}
\title{Singularity Models for High Codimension Mean Curvature Flow in Riemannian Manifolds}
\maketitle
%\tableofcontents
\begin{abstract}
We study the mean curvature flow of smooth $n$-dimensional compact submanifolds with quadratic pinching in a Riemannian manifold $\mathcal{N}^{n+m}$. Our main focus is on the case of high codimension, $m\geq 2$. We establish a codimension estimate that shows in regions of high curvature, the submanifold becomes approximately codimension one in a quantifiable way. This estimate enables us to prove at a singular time of the flow, there exists a rescaling that converges to a smooth codimension-one limiting flow in Euclidean space. Under a cylindrical type pinching, this limiting flow is weakly convex and moves by translation. Our approach relies on the preservation of the quadratic pinching condition along the flow and a gradient estimate that controls the mean curvature in regions of high curvature. These estimates allow us to analyse the behaviour of the flow near singularities and establish the existence of the limiting flow.
\end{abstract}
\section{Introduction}\label{sec_introduction}
Let $F_0\colon\mathcal{M}^n \rightarrow \mathcal{N}^{n+m}$ be a smooth immersion of a compact manifold $\mathcal{M}^n$. The mean curvature flow starting from $F_0$ is the following family of submanifolds
	\begin{align*}
	F\colon \mathcal{M}^n \times[0, T) \rightarrow \mathcal{N}^{n+m}
	\end{align*}
such that
	\begin{align}\label{mean curvature flow}
	\begin{split}
\left\{
	\begin{array}{rl}
		\partial_t F(p, t) &=H(p, t), \ \ \text{for} \ \ p \in \mathcal{M}, t \in[0, T)   \\
		F(p, 0) &=F_0(p)
	\end{array}
	\right.
\end{split}
	\end{align}
where $H(p, t)$ denotes the mean curvature vector of $\mathcal{M}_t=F_t(p)=F(p, t)$ at $p$. It is well known this is a system of quasilinear weakly parabolic partial differential equations for $F$. Geometrically, the mean curvature flow is the steepest descent flow for the area functional of a submanifold and hence it is a natural curvature flow.\\
In the case of codimension one, a crucial step in the study of singularity formation in the mean convex mean curvature flow is the convexity estimate. This states that in regions of large mean curvature, the second fundamental form is almost positive definite.\\
In the paper Huisken \cite{Hu84} proved that closed convex hypersurfaces under the mean curvature flow evolve into spherical singularities, using Stampacchia iteration, the Michael--Simons--Sobolev inequality together with recursion formulae for symmetric polynomials. In \cite{Hu86}, Huisken then generalises this theorem to Riemannian background curvature spaces with strict convexity depending on the background curvature. \\
In contrast, White \cite{Wh03,white2005local} uses compactness theorems of geometric measure theory together with the rigidity of strong maximum principle for the second fundamental form. Haslhofer-Kleiner \cite{HK2} developed an alternative approach to White's results based on Andrews's non-collapsing \cite{An12} result for the mean curvature flow.\\
The case of mean curvature flow of mean convex hypersurfaces in Euclidean space has been investigated by White \cite{Wh03} and Huisken-Sinestrari \cite{HuSi09}, who have developed a deep and far reaching analysis of the formation of singularities. Recently there has been a number of works generalising these results to high codimension mean curvature flow \cite{Nguyen2020},\cite{LyNgConvexity}, \cite{Naff}. The purpose of this paper is to obtain a suitable generalisation of these results for high codimension mean curvature flow in Riemannian submanifolds.\\
Most of the work done on mean curvature flow in higher codimension uses assumptions on the image of the Gauss map. They have either considered graphical submanifolds, \cite{Chen2002},\cite{Li2003}, \cite{Wang2002},\cite{Wang2004}, submanifolds with additional symplectic or Lagrangian structure \cite{Smoczyk2002},\cite{Chen2001},\cite{Wang2001},\cite{Smoczyk2004}, \cite{Neves2007} or using convex subsets of the Grassmannian are preserved by the mean curvature flow, \cite{Tsui2004},\cite{Wang2003},\cite{Wang2005}.
Therefore, we will focus on conditions on the norm of the second fundamental form. In high codimension, the mean curvature flow is more complex than in the hypersurface case, where there is only one normal direction. In the hypersurface setting, the second fundamental form is a symmetric real-valued two-tensor, and the mean curvature is a real-valued function, which simplifies the analysis of the flow. However, the presence of normal curvature complicates reaction terms in the evolution equations for the second fundamental form, making the analysis of high codimension mean curvature flow more challenging.\\
An alternative condition was introduced by Andrews--Baker in \cite{AnBa10}. On a compact submanifold, if $|H|>0$, there exists a $c>0$, such that
	\begin{equation}\label{this}
	|A|^2\le c|H|^2,
	\end{equation}
which is preserved by codimension one mean curvature flow. Also, this condition makes sense for all codimensions. In fact, Andrews--Baker showed that for $c\le\frac{4}{3n}(<\frac{1}{n-1}$  for $n<4$), then \eqref{this} is preserved along the mean curvature flow. For $c = \min\{\frac{4}{3n},\frac{1}{n-1}\}$, remarkably they were able to prove convergence to a round sphere. We note the condition $|A|^2 - \frac{1}{n-1}|H|^2 < 0, H >0$ implies convexity in codimension one. This lead Andrews--Baker to consider the pinching condition:
	\begin{equation}\label{pinchingcondition}
	|A|^2-c_n|H|^2+d_n\le 0,
	\end{equation}
which, is preserved by mean curvature flow, for $c_n\le\frac{4}{3n}$ and $d_n>0$. In the paper \cite{Nguyen2020}, a surgery construction was developed allowing high codimension mean curvature flow with cylindrical pinching to pass through singularities. This generalised the codimension one result of \cite{HuSi09} (see also \cite{HK2}) to high codimension. A key aspect of this surgery procedure is the codimension estimate presented in \cite{Naff}, which shows that near regions of high curvature, singularities become approximately codimension one. Another crucial component is the  cylindrical estimate, which shows that nears regions of high curvature, the submanifold becomes approximately cylindrical of the form $ \mathbb S^{n-1}\times \mathbb R$.  These estimates are essential for the surgery to work and allow us to control the geometry of the submanifold in regions of high curvature.\\
In this paper, we study singularity formation in high codimension mean curvature flow in Riemannian manifolds and will consider the following curvature pinching condition of the length of the second fundamental form 
	\begin{align*}
	|A|^2-c_n|H|^2 \leq -d_n(K_1,K_2,L)
	\end{align*}
for some positive constant $d_n$ depending on the background curvature, where
	\begin{align*}
	\begin{split} 
	-K_1\leq K_{\mathcal N} \leq K_2, \quad |\bar{\nabla} \bar{R} | \leq L,\quad \inj(\mathcal N) \geq \imath_{\mathcal N}.
	\end{split}
	\end{align*}
and 
	\begin{align*}
	c_n:=\min \left\{\frac{4}{3 n}, \frac{1}{n-2}\right\}, \quad \text { if } n\geq 5.
	\end{align*}
This was shown to be preserved in the paper of \cite{Liu} and represents a natural generalisation of Huisken's condition in \cite{Hu86} to high codimension background Riemannian manifolds. We will show in regions of high curvature where the mean curvature is large, the submanifold becomes approximately codimension one in a quantifiable sense. In particular, we will prove a theorem that extends of the main theorem of \cite{Naff} to Riemannian background spaces.
\begin{customthm}{5.1}
Let $F: \mathcal{M}^n\times[0, T) \rightarrow \mathcal{N}^{n+m}$ be a smooth solution to mean curvature flow \eqref{mean curvature flow} so that
$F_0(p)=F(p, 0)$ is compact and quadratically pinched.
Then $\forall \e>0, \exists H_0 >0$, such that if $f \geq H_0$, then
	\begin{align*}
	\left|A^-\right|^2 \leq \e f+C_{\e}
	\end{align*}
$\forall t \in[0, T)$ where $C_\e=C_{\e}(n, m)$.
\end{customthm}
Assuming the quadratic pinching condition, we prove singularity models for the pinched flow must always have codimension one, regardless of the original flow's codimension.\\
The outline of the paper is as follows. In section 2, we give all the technical tools needed for our work and set up our notation. In section 3, we give the proof for the preservation of the quadratic pinching condition along the mean curvature flow. In section 4, we prove the gradient estimate. The importance of the gradient estimate is that it allows us to control the mean curvature and hence the full second fundamental form on a neighbourhood of fixed size. In section 5, we prove the codimension estimate, which is the main theorem of this paper. This means that in regions of high curvature, the submanifold becomes codimension one quantitatively. In section 6, we show how the codimension estimate in Riemannian manifolds actually falls into the Euclidean case. Finally, in section 7, we prove the codimension estimate in the case of constant negative curvature. 

\textbf{Acknowledgements.}  
The first author would like to acknowledge the support of the EPSRC through the grant EP/S012907/1.

\section{Preliminaries}\label{sec_Preliminaries}
This chapter presents the necessary preliminary results and establishes our notation. We derive evolution equations for the length and squared length of the second fundamental form, as well as for the mean curvature vectors, in an arbitrary Riemannian background space of any codimension. Additionally, we provide a proof for a Kato-type inequality we will utilise throughout this paper. Let $F\colon \mathcal M^n\times [0,T)\to\mathcal N^{n+m}$ be an $n$-dimensional smooth, closed and connected submanifold in an $(n+m)$-dimensional smooth complete Riemannian manifold. We adopt the following convention for indices:
	\begin{align*}
	1\le i,j,k,\dots\le n, \ 1\le a,b,c,\dots\le n+m \ \ \text{ and} \ \ 1\le \alpha,\beta,\gamma,\dots \le m.
	\end{align*}
We denote by $A$ to be the normal vector valued second fundamental form tensor and denote by $H$ the mean curvature vector which is the trace of the second fundamental form given by $H^\alpha=\sum_i A^\alpha_{ii}$. The tracefree second fundamental form $\mathring{A}$  is defined by $\mathring{A}=A-\frac{1}{n}Hg$, whose components are given by $\mathring{A}^\alpha_{ij}=A^\alpha_{ij}-\frac{1}{n}H^\alpha g_{ij}$. Obviously, we have $\sum_i \mathring{A}^\alpha_{ii}=0$. \\
We define the principal normal direction to be given by $\nu_1=\frac{H}{|H|}$. This is well defined since in our setting $|H| \neq 0$. We denote by $A^-$ the second fundamental form tensor orthogonal to the principal direction and $h$ to be the tensor valued second fundamental form in the principal direction, that is $h_{ij}=\frac{\langle A_{ij},H\rangle}{|H|}$. Therefore, we have $A=A^-+h\nu_1$. Also, $A^+_{ij}=\langle A_{ij},\nu_1\rangle\nu_1$. From the definition of $A^-$, it is natural to define the connection $\hat{\nabla}^\bot$ acting on $A^-$, by
	\begin{align*}
	\hat{\nabla}^\bot_i A^-_{jk}:=\nabla^\bot_i A^-_{jk}-\langle \nabla^\bot_i A^-_{jk},\nu_1\rangle\nu_1.
	\end{align*}
We denote $\mathring{h}$ to be the traceless part of the second fundamental form in the principal direction.  For the choice of $\nu_1$, we have for $\alpha\ge 2$ and $H^\alpha=\tr A^\alpha=0$. The traceless second fundamental form can be rewritten as $\mathring{A}=\sum_{\alpha} \mathring{A}^\alpha \nu_{\alpha}$, where
	\begin{align*}
	\mathring{h}=h-\frac{|H|}{n} Id, \ \ \text{ for} \ \ \alpha=1
	\end{align*}
and
	\begin{align*}
	\mathring{A}^-=A^\alpha, \ \ \text{ for} \ \ \alpha\ge 2.
	\end{align*}
We set
	\begin{align*}
	|A|^2=|h|^2+|A^-|^2 \ \ \text{ and} \ \ |\mathring{A}|^2=|\mathring{h}|^2+|\mathring{A}^-|^2.
	\end{align*}
Let
	\begin{align*}
	R_{ijkl}=g\big(R(e_i,e_j)e_k,e_l\big),  \ \ \bar{R}_{abcd}=\langle\bar{R}(e_a,e_b)e_c,e_d\rangle \ \ \text{and} \ \ R_{ij\alpha\beta}^\bot =\langle R^\bot(e_i,e_j)e_\alpha,e_\beta\rangle.
	\end{align*}
\begin{proposition}[\cite{AnBa10}, Section 3]
With the summation convention, the evolution equations of $A_{ij}$ and $H$ are
	\begin{align}\label{eqn_A}
	\Big(\partial_t-\Delta\Big)A_{ij}&=\sum_{p,q}\langle A_{ij},A_{pq}\rangle A_{pq}+\sum_{p,q}\langle A_{iq},A_{pq}\rangle A_{pj}+\sum_{p,q}\langle A_{jq},A_{pq}\rangle A_{pi}-2\sum_{p,q}\langle A_{ip},A_{jq}\rangle A_{pq}\nonumber\\
	&+2\sum_{p,q}\bar{R}_{ipjq}A_{pq}-\sum_{k,p}\bar{R}_{kjkp} A_{pi}-\sum_{k,p}\bar{R}_{kikp}A_{pj}+\sum_{k,\alpha,\beta}A^\alpha_{ij}\bar{R}_{k\alpha k\beta} \nu_\beta\nonumber\\
	&-2\sum_{p,\alpha,\beta}A^\alpha_{jp}\bar{R}_{ip\alpha\beta}\nu_\beta-2\sum_{p,\alpha,\beta}A^\alpha_{ip}\bar{R}_{jp\alpha\beta}\nu_\beta\nonumber\\
	&+\sum_{k,\beta}\bar{\nabla}_k \bar{R}_{kij\beta}\nu_\beta-\sum_{k,\beta}\bar{\nabla}_i \bar{R}_{jkk\beta}\nu_\beta,
	\end{align}
	\begin{align}\label{eqn_H}
	\Big(\partial_t-\Delta\Big)H=\sum_{p,q}\langle H,A_{pq}\rangle A_{pq}+\sum_{k,\alpha,\beta}H^\alpha\bar{R}_{k\alpha k\beta}\nu_\beta.
	\end{align}
\end{proposition}
\begin{lemma}[\cite{BakerThesis}, Section 5.1]\label{evoleqs}
Let us consider a family of immersions $F\colon \mathcal{M}^n\times [0,T) \to \mathcal{N}^{n+m}$ moving by mean curvature flow. Then, we have the following evolution equations
	\begin{align}
	\partial_t d\mu_t=-|H|^2d\mu_t,
	\end{align}
	\begin{align}\label{eqn_|A|^2}
	\partial_t |A|^2&=\Delta |A|^2-2|\nabla A|^2+2\sum_{\alpha,\beta}\big(\sum_{i,j} A^\alpha_{ij} A_{ij}^\beta \big)^2+2\sum_{i,j,\alpha,\beta}\Big(\sum_p\big(A^\alpha_{ip}A_{jp}^\beta -A^\alpha_{jp}A_{ip}^\beta \big)\Big)^2\nonumber\\
	&+4\sum_{i,j,p,q}\bar{R}_{ipjq}\big(\sum_{\alpha} A^\alpha_{pq}A^\alpha_{ij}\big)-4\sum_{j,k,p}\bar{R}_{kjkp}\big(\sum_{i,\alpha} A^\alpha_{pi}A^\alpha_{ij}\big)+2\sum_{k,\alpha,\beta}\bar{R}_{k\alpha k\beta}\big(\sum_{i,j} A^\alpha_{ij}A_{ij}^\beta \big)\nonumber\\
	&-8\sum_{j,p,\alpha,\beta}\bar{R}_{jp\alpha\beta}\big(\sum_iA^\alpha_{ip}A_{ij}^\beta \big)+2\sum_{i,j,k,\beta}\bar{\nabla}_k\bar{R}_{kij\beta}A_{ij}^\beta -2\sum_{i,j,k,\beta}\bar{\nabla}_i\bar{R}_{jkk\beta}A_{ij}^\beta,
	\end{align}
	\begin{align}\label{eqn_|H|^2}
	\partial_t |H|^2=\Delta |H|^2-2|\nabla H|^2+2\sum_{i,j}\big(\sum_{\alpha} H^\alpha A^\alpha_{ij}\big)^2+2\sum_{k,\alpha,\beta}\bar{R}_{k\alpha k\beta} H^\alpha H^\beta.
	\end{align}
\end{lemma}
By Berger's inequality,
	\begin{align}\label{Berger}
	&|\bar{R}_{acbc}|\le\frac{1}{2}(K_1+K_2), \ \ \text{ for} \ a\neq b\nonumber\\
	&\\
	&|\bar{R}_{abcd}|\le\frac{2}{3}(K_1+K_2), \ \ \text{for all distinct indices} \ a,b,c,d.\nonumber
	\end{align}
\begin{lemma}[\cite{Liu}, Lemma 3.1]\label{katoinequality} For any $\eta>0$ we have the following inequalities
	\begin{align}\label{part3.1}
	|\nabla^\perp A|^2 \geq\left(\frac{3}{n+2}-\eta\right)&|\nabla^\perp H|^2-\frac{2}{n+2}\left(\frac{2}{n+2} \eta^{-1}-\frac{n}{n-1}\right)|w|^2
	\end{align}
and
	\begin{align}\label{part3.2}
	|\nabla^\perp A|^2-\frac{1}{n}|\nabla^\perp H|^2&\geq \frac{n-1}{2 n+1}|\nabla^\perp A|^2-\frac{2 n}{(n-1)(2 n+1)}|w|^2 \nonumber\\
	&\geq \frac{n-1}{2 n+1}|\nabla^\perp A|^2-C(n, d)\left(K_1+K_2\right)^2.
	\end{align}
Here $w=\sum_{i, j, \alpha} \bar{R}_{\alpha j i j} e_i \otimes \omega_\alpha$ and $C(n, d)=\frac{n^4 d}{2(n-1)(2 n+1)}$.
	\end{lemma}
	\begin{proof}
 Inequality \eqref{part3.2} follows from \eqref{part3.1} with $\eta=\frac{n-1}{n(n+2)}$. To prove \eqref{part3.1}, we set
	\begin{align}\label{E_ijk}
	E_{i j k}&=\frac{1}{n+2}\left(\nabla_i ^\perp H g_{j k}+\nabla_j ^\perp H  g_{i k}+\nabla_k ^\perp H g_{i j}\right)\nonumber\\
	&-\frac{2}{(n+2)(n-1)} w_i g_{j k}+\frac{n}{(n+2)(n-1)}\left(w_j g_{i k}+w_k g_{i j}\right).
	\end{align}
Let $F_{i j k}=\nabla_i ^\perp h_{j k}-E_{i j k}$. By the Codazzi equation we have $\left\langle E_{i j k}, F_{i j k}\right\rangle=0$. Hence, $|\nabla^\perp A|^2 \geq|E|^2$. By a direct computation, we have
	\begin{align*}
	|E|^2=\frac{3}{n+2}|\nabla^\perp H|^2+\frac{2 n}{(n+2)(n-1)}|w|^2+\frac{4}{n+2}\langle\nabla^\perp H, w\rangle.
	\end{align*}
But from Cauchy-Schwarz inequality and Young's inequality for products, we have
	\begin{align*}
	\frac{4}{n+2}\langle\nabla^\perp H,w\rangle\ge-\eta|\nabla^\perp H|^2-\frac{4}{(n+2)^2}\eta^{-1}|w|^2.
	\end{align*}
Plugging the above inequality into \eqref{E_ijk}, we get \eqref{part3.1}.
\end{proof}
\begin{proposition}[\cite{Naff}, Proposition 2.2]
	\begin{align}\label{2.22naff}
|\nabla^{\perp} A|^2 & =\sum_{i,j,k}|\hat{\nabla}_i^{\perp} A^-_{j k}+h_{j k} \nabla_i^{\perp} \nu_1|^2+\sum_{i,j,k}|\langle\nabla_i^{\perp} A^-_{j k}, \nu_1\rangle+\nabla_i h_{j k}|^2. \\
|\nabla^{\perp} H|^2 & =|H|^2|\nabla^{\perp} \nu_1|^2+|\nabla| H||^2. \nonumber\\
|\nabla^{\perp} A^-|^2 & =|\hat{\nabla}^{\perp} A^-|^2+|\langle\nabla^{\perp} A^-, \nu_1\rangle|^2\nonumber.
	\end{align}
\end{proposition}
We will use these identities in Sections 5 and 8. It is very useful to consider the implications of the Codazzi equation for the decomposition of $\nabla_i^{\perp} A_{j k}$ above. Projecting the Codazzi equation onto $E_1$ and $\hat{E}$ implies the both of the tensors
	\begin{align*}
\nabla_i h_{j k}+\langle\nabla_i^{\perp} A^-_{j k}, \nu_1\rangle \quad \text { and } \quad \hat{\nabla}_i^{\perp} A^-_{j k}+h_{j k} \nabla_i^{\perp} \nu_1
	\end{align*}
are symmetric in $i, j, k$. Consequently, it is equivalent to trace over $j, k$ or trace over $i, k$, and this implies
	\begin{align}\label{2.25naff}
& \sum_{k=1}^n (\nabla_k h_{i k}+\langle\nabla_k^{\perp} A^-_{i k}, \nu_1\rangle)=\nabla_i|H|, 
	\end{align}
	\begin{align}\label{2.26naff}
& \sum_{k=1}^n( \hat{\nabla}_k^{\perp} A^-_{i k}+h_{i k} \nabla_k^{\perp} \nu_1)=|H| \nabla_i^{\perp} \nu_1.
	\end{align}
\section{Preservation of the Quadratic Pinching}
This section demonstrates the quadratic pinching condition \eqref{pinching_condition} is preserved throughout the mean curvature flow, for a suitable positive constant $d_n$ that depends on the background curvature. The proof, presented in \cite{Liu}, generalises Huisken's pinching condition \cite{Hu86} to high codimension. As we require a slight refinement of this pinching, we provide the proof for completeness.
\begin{theorem}[\cite{Liu}, Section 3]\label{thm_pinching}
Let $ F\colon\mathcal M^n \rightarrow \mathcal N^{n+m}$ be an $n$-dimensional, smooth, closed and connected submanifold in an $n+m$-dimensional smooth complete Riemannian manifold, such that
	\begin{align}\label{eqn_1}
	\begin{split} 
	 -K_1\leq K_{\mathcal N} \leq K_2, \quad |\bar{\nabla} \bar{R} | \leq L,\quad \inj(\mathcal N) \geq \imath_{\mathcal N}.
	\end{split}
	\end{align}
Then, there is a constant $d_n =d_n(K_1,K_2,L)$ depending only on the dimension $n$, the bounds for the sectional curvature $K_1,K_2$ and the bound for the derivative of the curvature $L$, such that for $ c_n \leq \frac{4}{3n}$
\begin{align} \label{pinching_condition}
|A|^2\leq c_n|H|^2 -d_n
\end{align}
is preserved by the mean curvature flow.
\end{theorem}
\begin{proof}
Set $g=|A|^2-c_n|H|^2+d_n$, where $c_n\le\frac{4}{3n}, d_n>d$ where $d$ is a positive constant to be determined. We compute the evolution equation for $g$ along the mean curvature flow and show if $g=0$ at a point in the space-time, then $(\partial_t-\Delta)g$ is negative at this point. By the maximum principle, the theorem follows. More precisely, by Lemma \ref{evoleqs}, we have
	\begin{align}\label{gafterkato}
	\left(\partial_t-\Delta\right) g&=-2(|\nabla A|^2-c_n|\nabla H|^2)+2R_1-2c_n R_2+P_\alpha,
	\end{align}
 where
	\begin{align*}
	R_1=\sum_{\alpha,\beta}\big(\sum_{i,j} A^\alpha_{ij} A_{ij}^\beta \big)^2+\sum_{i,j,\alpha,\beta}\Big(\sum_p\big(A^\alpha_{ip}A_{jp}^\beta -A^\alpha_{jp}A_{ip}^\beta \big)\Big)^2,
	\end{align*}
	\begin{align*}
	R_2=\sum_{i,j}\big(\sum_{\alpha} H^\alpha A^\alpha_{ij}\big)^2
	\end{align*}
and
	\begin{align}\label{eqn_Pa}
	P_\alpha=I+II+III+IV,
	\end{align}
with
	\begin{align*}
	I=4\sum_{i,j,p,q}\bar{R}_{ipjq}\big(\sum_{\alpha} A^\alpha_{pq}A^\alpha_{ij}\big)-4\sum_{j,k,p}\bar{R}_{kjkp}\big(\sum_{i,\alpha} A^\alpha_{pi}A^\alpha_{ij}\big),
	\end{align*}
	\begin{align*}
	II=2\sum_{k,\alpha,\beta}\bar{R}_{k\alpha k\beta}\big(\sum_{i,j} A^\alpha_{ij}A_{ij}^\beta \big)-2c_n\sum_{k,\alpha,\beta}\bar{R}_{k\alpha k\beta} H^\alpha H^\beta,
	\end{align*}
	\begin{align*}
	III=-8\sum_{j,p,\alpha,\beta}\bar{R}_{jp\alpha\beta}\big(\sum_iA^\alpha_{ip}A_{ij}^\beta \big),
	\end{align*}
	\begin{align*}
	IV=2\sum_{i,j,k,\beta}\bar{\nabla}_k\bar{R}_{kij\beta}A_{ij}^\beta -2\sum_{i,j,k,\beta}\bar{\nabla}_i\bar{R}_{jkk\beta}A_{ij}^\beta.
	\end{align*}
To estimate the reaction terms, it is convenient to work with the traceless part of the second fundamental form $\mathring{A}=A-\frac{1}{n}H Id$. The lengths of $A$ and $\mathring{A}$ are related by
	\begin{align*}
	|\mathring{A}|^2=|A|^2-\frac{1}{n}|H|^2.
	\end{align*}
At the point where $g=0$, that is $|A|^2=c_n|H|^2-d_n$, the mean curvature vector is not zero. We choose a local orthonormal frame $\{\nu_\alpha, 1\le\alpha\le m\}$ for the normal bundle, such that $\nu_1=\frac{H}{|H|}$, the principal normal direction.  For the choice of $\nu_1$, we have $H^+=\tr A^+=|H|$ and for $\alpha\ge 2$ and $H^\alpha=\tr A^\alpha=0$. The traceless second fundamental form can be rewritten as $\mathring{A}=\sum_{\alpha} \mathring{A}^\alpha \nu_{\alpha}$, where
	\begin{align*}
	\mathring{h}=h-\frac{|H|}{n} Id, \ \ \text{ for} \ \ \alpha=1
	\end{align*}
and
	\begin{align*}
	\mathring{A}^-=A^\alpha, \ \ \text{ for} \ \ \alpha\ge 2.
	\end{align*}
We set
	\begin{align*}
	|A|^2=|h|^2+|A^-|^2 \ \ \text{ and} \ \ |\mathring{A}|^2=|\mathring{h}|^2+|\mathring{A}^-|^2.
	\end{align*}
Since $|A|^2=c_n|H|^2-d_n$ at this point, we have $|H|^2=\frac{|\mathring{A}|^2+d_n}{c_n -\frac{1}{n}}$ and from \cite{AnBa10} we see that
	\begin{align}
	\begin{split}
2R_1-2c_n R_2&\le\left(6-\frac{2}{n(c_n-\frac{1}{n})}\right)|\mathring{h}|^2|\mathring{A}^-|^2+\left(3-\frac{2}{n(c_n-\frac{1}{n})}\right)|\mathring{A}^-|^4\\
	&-\frac{2c_n d_n}{c_n-\frac{1}{n}}|\mathring{h}|^2-\frac{4d_n}{n(c_n-\frac{1}{n})}|\mathring{A}^-|^2-\frac{2d_n^2}{n(c_n-\frac{1}{n})}.
\end{split}
	\end{align}
To estimate $I$, for a fixed $\alpha$ we choose a basis for the tangent space $e_i$'s, such that $h$ is diagonal. Denote by $\lambda_i$ and $\mathring{\lambda}_i$ the diagonal entries of $h$ and $\mathring{h}$, respectively. Therefore, $A^\alpha_{ij}=\lambda^\alpha_i\delta_{ij}$.
	\begin{align*}
	I&=4\sum_{i,j,p,q}\bar{R}_{ipjq} A^\alpha_{pq}A^\alpha_{ij}-4\sum_{j,k,p}\bar{R}_{kjkp}\big(\sum_{i,\alpha} A^\alpha_{pi}A^\alpha_{ij}\big)\\
	&=4\sum_{i,p}\bar{R}_{ipip}\big(\lambda^\alpha_i\lambda^\alpha_p-(\lambda^\alpha_i)^2\big)\\
	&=-2\sum_{i,p}\bar{R}_{ipip}\big(\lambda^\alpha_i-\lambda^\alpha_p\big)^2\\
	&\le4nK_1|\mathring{A}^\alpha|^2.
	\end{align*}
Hence, we get
	\begin{align}\label{eqn_I}
	I\le4nK_1(|\mathring{h}|^2+|\mathring{A}^-|^2).
	\end{align}
By the choice of $\nu_1$, we have $II=II_1+II_2+II_3,$ where
	\begin{align*}
	II_1=2\sum_{i,j,k}\bar{R}_{k1 k1}(A_{ij}^+)^2-2c_n\sum_k\bar{R}_{k1 k1}(H^+)^2,
	\end{align*}
	\begin{align*}
	II_2=4\sum_{k,\alpha\ge 2} \bar{R}_{k\alpha k1}\big(\sum_{i,j} A^\alpha_{ij} A_{ij}^+\big)-4c_n\sum_{k,\alpha\ge 2} \bar{R}_{k\alpha k1}H^+H^\alpha,
	\end{align*}
	\begin{align*}
	II_3=2\sum_{k,\alpha,\beta\ge 2}\bar{R}_{k\alpha k\beta}\big(\sum_{i,j}A^\alpha_{ij} A_{ij}^\beta \big)-2c_n\sum_{k,\alpha,\beta\ge 2}\bar{R}_{k\alpha k\beta}H^\alpha H^\beta.
	\end{align*}
Since, $|H|^2=\frac{|\mathring{A}|^2+d_n}{c_n -\frac{1}{n}}$ at that point, we have
	\begin{align}\label{eqn_II1}
	II_1&\le2nK_2|h|^2+2nc_n K_1|H|^2\nonumber\\
	&=2nK_2 |\mathring{h}|^2+2(nc_n K_1+K_2)\frac{|\mathring{A}|^2+d_n}{c_n-\frac{1}{n}}\nonumber\\
	&=\left(2nK_2+\frac{2(nc_n K_1+K_2)}{c_n-\frac{1}{n}}\right)|\mathring{h}|^2+\frac{2(nc_n K_1+K_2)}{c_n-\frac{1}{n}}|\mathring{A}^-|^2+\frac{2(nc_n K_1+K_2)}{c_n-\frac{1}{n}}d_n.
	\end{align}
Since $H^\alpha=0$, for $\alpha\ge 2$, we have the following estimates for $II_2,II_3$.
	\begin{align}\label{eqn_II2}
	II_2&=4\sum_{k,\alpha\ge2} \bar{R}_{k\alpha k1}\big(\sum_{i,j} A^\alpha_{ij} A_{ij}^+\big)\nonumber\\
	&=4\sum_{k,\alpha\ge2}\bar{R}_{k\alpha k1}\big(\sum_{i,j} \mathring{A}^\alpha_{ij}\mathring{A}_{ij}^+\big)\nonumber\\
	&\le(K_1+K_2)\sum_{k,\alpha\ge2}\big(\frac{1}{\rho}\sum_{i,j} (\mathring{A}^\alpha_{ij})^2+\rho\sum_{i,j}(\mathring{A}_{ij}^+)^2\big)\nonumber\\
	&=\rho n(m-1)(K_1+K_2)|\mathring{h}|^2+\frac{n}{\rho}(K_1+K_2)|\mathring{A}^-|^2,
	\end{align}
for any positive constant $\rho$.
	\begin{align}\label{eqn_II3}
	II_3&=2\sum_{k,\alpha,\beta\ge2}\bar{R}_{k\alpha k\beta}\big(\sum_{i,j}A^\alpha_{ij}A_{ij}^\beta \big)\nonumber\\
	&=2\sum_{k,\alpha\ge2}\bar{R}_{k\alpha k\alpha}\big(\sum_{i,j}(A^\alpha_{ij})^2\big)+2\sum_{k,\alpha,\beta\ge2,\alpha\neq\beta}\bar{R}_{k\alpha k\beta}\big(\sum_{i,j} A^\alpha_{ij} A_{ij}^\beta \big)\nonumber\\
	&\le2nK_2|\mathring{A}^-|^2+2\sum_{k,\alpha,\beta\ge2,\alpha\neq\beta}\bar{R}_{k\alpha k\beta}\big(\sum_{i,j}A^\alpha_{ij} A_{ij}^\beta \big)\nonumber\\
	&\le2nK_2|\mathring{A}^-|^2+\sum_{k,\alpha,\beta\ge2,\alpha\neq\beta}|\bar{R}_{k\alpha k\beta}|\sum_{i,j} \big((A^\alpha_{ij})^2+(A_{ij}^\beta)^2\big)\nonumber\\
	&\le2nK_2|\mathring{A}^-|^2+(K_1+K_2)\sum_{i,j,k,\alpha,\beta\ge2,\alpha\neq\beta} \big((A^\alpha_{ij})^2+(A^\beta_{ij})^2\big)\nonumber\\
	&=2nK_2|\mathring{A}^-|^2+n(m-2)(K_1+K_2)|\mathring{h}|^2.
	\end{align}
From \eqref{eqn_II1},\eqref{eqn_II2} and \eqref{eqn_II3}, we get the following estimate for $II$ :
	\begin{align}\label{eqn_II}
	\begin{split}
II&\le\left(2nK_2+\frac{2(nc_n K_1+K_2)}{c_n-\frac{1}{n}}+\big(\rho n(m-1)+n(m-2)\big)(K_1+K_2)\right)|\mathring{h}|^2\\
	&+\left(\frac{2(nc_nK_1+K_2)}{c_n-\frac{1}{n}}+\frac{n}{\rho}(K_1+K_2)+2nK_2\right)|\mathring{A}^-|^2+\frac{2(nc_nK_1+K_2)}{c_n-\frac{1}{n}}d_n.
\end{split}
	\end{align}
For $III$ we have $III=III_1+III_2$, where
	\begin{align*}
	III_1=-16\sum_{j,p,\alpha\ge2}\bar{R}_{jp\alpha 1}\big(\sum_i A^\alpha_{ip}A_{ij}^+\big),
	\end{align*}
	\begin{align*}
	III_2=-8\sum_{j,p,\alpha,\beta\ge2,\alpha\neq\beta}\bar{R}_{jp\alpha\beta}\big(\sum_i A^\alpha_{ip} A_{ij}^\beta \big).
	\end{align*}
We have the following estimates for arbitrary positive constant $\rho$ :
	\begin{align}\label{eqn_III1}
	III_1&=-16\sum_{j,p,\alpha\ge2}\bar{R}_{jp\alpha 1}\sum_i \mathring{A}^\alpha_{ip}\big(\mathring{A}_{ij}^++\frac{|H|}{n}\delta_{ij}\big)\nonumber\\
	&=-16\sum_{j\neq p,\alpha\ge2}\bar{R}_{jp\alpha 1}\sum_i\mathring{A}^\alpha_{ip}\mathring{A}_{ij}^+\nonumber\\
	&\le\frac{16}{3}(K_1+K_2)\sum_{i,j\neq p,\alpha\ge2}\big(\frac{1}{\rho}(\mathring{A}^\alpha_{ip})^2+\rho(\mathring{A}_{ij}^+)^2\big)\nonumber\\
	&=\frac{16}{3}\rho(n-1)(m-1)(K_1+K_2)|\mathring{h}|^2+\frac{16}{3\rho}(n-1)(K_1+K_2)|\mathring{A}^-|^2.
	\end{align}
For the second inequality, we use $\sum_{j,p}\bar{R}_{jp\alpha1}\mathring{A}^\alpha_{jp}=0$, since $\bar{R}_{jp\alpha 1}$ is anti-symmetric for $j,p$ and $\mathring{A}^\alpha_{jp}$ is symmetric for $j,p$.
For any fixed $\beta\ge 2$, we choose $\nu_i$'s, such that $\mathring{A}_{ij}^\beta =\mathring{\lambda}_i^\beta \delta_{ij}$. Then,
	\begin{align}\label{eqn_III2}
	III_2&=-8\sum_{j\neq p,\beta\ge2}\sum_{\alpha\ge2,\alpha\neq\beta} \bar{R}_{jp\alpha\beta}\mathring{A}^\alpha_{jp}\mathring{\lambda}_j^\beta \nonumber\\
	&\le\frac{8}{3}(K_1+K_2)\sum_{\beta\ge2}\Big((n-1)^{\frac{1}{2}}\sum_{j\neq p,\alpha \ge2,\alpha\neq\beta} (\mathring{A}^\alpha_{jp})^2+\frac{1}{(n-1)^{\frac{1}{2}}}\sum_{j\neq p,\alpha\ge2,\alpha\neq\beta}(\mathring{\lambda}_j^\beta)^2\Big)\nonumber\\
	&\le\frac{8}{3}(K_1+K_2)\Big((n-1)^{\frac{1}{2}}(m-2)|\mathring{A}^-|^2+\sum_{\beta\ge2}(n-1)^{\frac{1}{2}}(m-2)|\mathring{A}^\beta|^2\Big)\nonumber\\
	&=\frac{8}{3}(n-1)^{\frac{1}{2}}(m-2)(K_1+K_2)|\mathring{A}^-|^2.
	\end{align}
From \eqref{eqn_III1} and \eqref{eqn_III2}, we have
	\begin{align}\label{eqn_III}
	III&\le\frac{16}{3}\rho(n-1)(m-1)(K_1+K_2)|\mathring{h}|^2\nonumber\\
	&+\left(\frac{16}{3\rho}(n-1)+\frac{8}{3}(n-1)^{\frac{1}{2}}(m-2)\right)(K_1+K_2)|\mathring{A}^-|^2.
	\end{align}
For $IV$, we choose $\nu_i$'s, such that $A_{ij}^+ =\lambda_i \delta_{ij}$. If $K_1+K_2\neq 0$, we have
	\begin{align}\label{eqn_IV}
	IV&=2\sum_{i,k}\bar{\nabla}_k\bar{R}_{kii1}(\lambda_i-\lambda_k)-2\sum_{i,j,k,\beta\ge2}(\bar{\nabla}_k\bar{R}_{kij\beta}-\bar{\nabla}_i\bar{R}_{jkk\beta})\mathring{A}_{ij}^\beta\nonumber\\
	&\le\sum_{i,k}\left(\frac{1}{\theta}(\bar{\nabla}_k\bar{R}_{kii1})^2+\theta(\lambda_i-\lambda_k)^2\right)+\sum_{i,j,k,\beta\ge2}\left(\frac{2}{\vartheta}\big((\bar{\nabla}_k \bar{R}_{kij\beta})^2+(\bar{\nabla}_i\bar{R}_{jkk\beta})^2\big)+\vartheta(\mathring{A}_{ij}^\beta)^2\right)\nonumber\\
	&\le\frac{L^2}{\theta}+\theta|\mathring{h}|^2+\frac{4L^2}{\vartheta}+n\vartheta|\mathring{A}^-|^2,
	\end{align}
for positive constants $\theta,\vartheta$. If $K_1+K_2=0$, then $L=0$ and we may choose $\theta,\vartheta=0$. Combining \eqref{eqn_I},\eqref{eqn_II},\eqref{eqn_III} and \eqref{eqn_IV}, we have
	\begin{align}\label{3.9}
	\left(\partial_t-\Delta\right)g&\le-2(|\nabla A|^2-c_n|\nabla H|^2)+\left(6-\frac{2}{n(c_n-\frac{1}{n})}\right)|\mathring{h}|^2|\mathring{A}^-|^2+\left(3-\frac{2}{n(c_n-\frac{1}{n})}\right)|\mathring{A}^-|^4\nonumber\\
	&-\frac{2c_n d_n}{c_n-\frac{1}{n}}|\mathring{h}|^2-\frac{4d_n}{n(c_n-\frac{1}{n})}|\mathring{A}^-|^2-\frac{2d_n^2}{n(c_n-\frac{1}{n})}\nonumber\\
	&+C_1|\mathring{h}|^2+C_2|\mathring{A}^-|^2+C_3d_n+C_4.
	\end{align}
Here,
	\begin{align*}
	C_1&=4nK_1+2nK_2+\frac{2(nc_nK_1+K_2)}{c_n-\frac{1}{n}}\\
	&+\Big(\rho n(m-1)+n(m-2)+\frac{16}{3}\rho(n-1)(m-1)\Big)(K_1+K_2)+\theta,
	\end{align*}
	\begin{align*}
	C_2&=4nK_1+2nK_2+\frac{2(nc_nK_1+K_2)}{c_n-\frac{1}{n}}\\
	&+\Big(\frac{n}{\rho}+\frac{16}{3\rho}(n-1)+\frac{8}{3}(n-1)^{\frac{1}{2}}(m-2)\Big)(K_1+K_2)+n\vartheta,
	\end{align*}
	\begin{align*}
	C_3=\frac{2(nc_nK_1+K_2)}{c_n-\frac{1}{n}},
	\end{align*}
	\begin{align*}
	C_4=\frac{L^2}{\theta}+\frac{4L^2}{\vartheta}, \ \ \text{ for} \ \ K_1+K_2\neq0 \ \ \text{ and} \ \ C_4=0, \ \ \text{ for} \ \ K_1+K_2=0,
	\end{align*}
   From the Kato type inequality in \eqref{katoinequality}, we have that
	\begin{align}\label{kato_negative}
	-2(|\nabla ^\bot A|^2-c_n|\nabla^\bot H|^2)&\le -2\Big(\frac{3}{n+2}-\eta-c_n \Big)|\nabla ^\bot H|^2\nonumber\\
	&+\Big(\frac{8}{\eta(n+2)^2}-\frac{4n}{(n+2)(n-1)} \Big)|w|^2\nonumber\\
	&\le 0,
	\end{align}
for a suitable positive constant $\eta$. If $K_1+K_2\neq 0$, set
	\begin{align*}
 	d=\max\left\{\frac{C_1}{2c_n}(c_n-\frac{1}{n}),\frac{C_2}{4}n(c_n-\frac{1}{n}),\frac{1}{4}n(c_n-\frac{1}{n})\left(C_3+\sqrt{C_3^2+\frac{8C_4}{n(c_n-\frac{1}{n})}}\right)\right\},
	\end{align*}
with $\rho=\rho=\theta=\vartheta=1$. If $K_1+K_2=0$, set $d=0$. So, if $d_n>d$, we have
	\begin{align*}
	\left(\partial_t-\Delta\right)g<0.
	\end{align*}
Then, by the maximum principle, $|A|^2\le c_n|H|^2-d_n$ is preserved along the mean curvature flow.
\end{proof}

\begin{remark} 
We see that as $K_1,K_2,L \rightarrow 0$ that $d_n \rightarrow 0$. In particular, since any sufficiently small region of a smooth Riemannian manifold is locally Euclidean we see that perturbations of manifolds satisfying $
|A|^2 - c_n|H|^2<0$ in an exponential neighbourhood of any point satisfy this inequality hence there are many submanifolds to which this inequality applies. 
\end{remark}
\section{Gradient Estimate}
This section presents a proof of the gradient estimate for the mean curvature flow. We establish this estimate directly from the quadratic curvature bound $|A|^2 < c_n |H|^2 - d_n$, where $c \leq \frac{4}{3n}$, without relying on the asymptotic cylindrical estimates. In fact, we demonstrate the cylindrical estimates follow as a consequence of the gradient estimates we derive here. These estimates are pointwise gradient estimates that rely solely on the mean curvature (or, equivalently, the second fundamental form) at a point and not on the maximum of curvature, as is the case with more general parabolic-type derivative estimates. Specifically, we obtain
\begin{align*}
\frac{3}{n+2}-c>0.
\end{align*}
This inequality enables us to combine the derivative terms in the evolution equation of $|A|^2-c_n |H|^2+d_n$ with the Kato-type inequality from Lemma \ref{katoinequality}.

\begin{theorem}[cf.\cite{HuSi09}, Section 6]\label{thm_gradient}
Let $ \mathcal{M}_t , t \in [0,T)$ be a closed $n$-dimensional quadratically bounded solution to the mean curvature flow in the Riemannian manifold $\mathcal{N}^{n+m}$, with $n \geq 8$, that is
	\begin{align*}
	|A|^2 - c|H|^2 +d<0, |H| >0
	\end{align*}
with $ c=\frac{1}{n-2}$.
 Then, there exists a constant $ \gamma_1= \gamma_1 (n, \mathcal M_0)$ and a constant $ \gamma_2 = \gamma_2 ( n , \mathcal{M}_0)$, such that the flow satisfies the uniform estimate
	\begin{align*}
	|\nabla A|^2 \leq \gamma_1 |A|^4+\gamma_2,
	\end{align*}
 for every $t\in [0, T)$.
\end{theorem}
\begin{proof}
We choose here $ \kappa_n = \left( \frac{3}{n+2}-c\right)>0$. Since $n\ge 8$, $\kappa_n$ is strictly positive. We will consider here the evolution equation for
	\begin{align*}
	\frac{|\nabla A|^2}{g^2},
	\end{align*}
where $ g = c|H|^2-|A|^2-d>0$. Since $ |A|^2-c|H|^2 < 0, |H|>0$ and $\mathcal{M}_0$ is compact, there exists an $ \eta(\mathcal{M}_0) >0, C_\eta(\mathcal{M}_0)>0$, so that
	\begin{align}\label{eqn_eta}
	\left( c-\eta \right)|H|^2-|A|^2 \geq C_\eta>0.
	\end{align}
Hence, we set
 	\begin{align*}
	g=c|H|^2-|A|^2\ge\eta|H|^2>\frac{\eta}{c}|A|^2>\e_1|A|^2+\e_2,
	\end{align*}
where $ \e_1 = \frac{\eta}{c}$ and $\e_2>0$. From \eqref{gafterkato} and \eqref{katoinequality} in Theorem \ref{thm_pinching} and a suitable constant $d$, we get
	\begin{align*}
	\partial_t g&= \Delta g-2 \left( c|\nabla H|^2-|\nabla A|^2 \right)+2 \left( c R_2-R_1 \right)+P_{\alpha}\\
	&\geq \Delta g-2 \left(\Big(\frac{3}{n+2}-\eta\Big)^{-1}c-1\right) |\nabla A|^2\\
	 	&\ge\Delta g-2\Big(\frac{n+2}{3}c-1\Big)|\nabla A|^2\\
	&= \Delta g+2\kappa_n \frac{n+2}{3}| \nabla A|^2,
	\end{align*}
for a suitable positive constant $\eta$. The evolution equation for $ |\nabla A|^2 $ is given by  	\begin{align*}
	\Big(\partial_t-\Delta\Big) |\nabla A|^2&\leq-2 |\nabla^2 A|^2+c |A|^2 |\nabla A|^2+d|\nabla A|^2.
	\end{align*}
Let $w,z$ satisfy the evolution equations
	\begin{align*}
	\partial_tw = \Delta w+W , \quad \partial_tz = \Delta z+Z
	\end{align*}
 then, we find
 	\begin{align*}
	\Big(\partial_t-\Delta\Big)\frac{w}{z}&=\frac{2}{z}\left\la \nabla \left( \frac{w}{z}\right) , \nabla z \right\ra+\frac{W}{z}-\frac{w}{z^2} Z\\
	&=2\frac{\la \nabla w , \nabla z \ra}{z^2}-2 \frac{w|\nabla z |^2}{z^3}+\frac{W}{z}-\frac{w}{z^2} Z.
	\end{align*}
Furthermore, for any function $g$, we have by Kato's inequality
	\begin{align*}
	\la \nabla g , \nabla |\nabla A|^2 \ra&\leq 2 |\nabla g| |\nabla^2 A| |\nabla A| \leq \frac{1}{g}|\nabla g |^2 | \nabla A|^2+g |\nabla^2 A|^2.
	\end{align*}
We then get
	\begin{align*}
	-\frac{2}{g}| \nabla^2 A|^2+\frac{2}{g}\left\la \nabla g ,\nabla \left( \frac{|\nabla A|^2}{g}\right) \right\ra \leq-\frac{2}{g}| \nabla^2 A|^2-\frac{2}{g^3}|\nabla g|^2 |\nabla A|^2+\frac{2}{g^2}\la \nabla g ,\nabla |\nabla A|^2 \ra \leq 0.
	\end{align*}
Then, if we let $ w = |\nabla A|^2 $ and $ z = g$, with $W \leq-2 |\nabla^2 A|^2+c |A|^2 |\nabla A|^2+d|\nabla A|^2$ and $Z\geq 2\kappa_n \frac{n+2}{3}| \nabla A|^2 $, we get
	\begin{align*}
	\Big(\partial_t-\Delta\Big) \frac{|\nabla A|^2}{g}&\leq \frac{2}{g}\left\la \nabla g ,\nabla \left( \frac{|\nabla A|^2}{g}\right) \right\ra+\frac{1}{g}(-2 |\nabla^2 A|^2+c |A|^2 |\nabla A|^2 \\
	&+d|\nabla A|^2)-2 \kappa_n \frac{n+2}{3}\frac{|\nabla A|^4}{g^2} \\
	&\leq c |A|^2 \frac{|\nabla A|^2}{g}+d\frac{|\nabla A|^2}{g}-2 \kappa_n \frac{n+2}{3}\frac{|\nabla A|^4}{g^2}.
	\end{align*}
We repeat the above computation with $w = \frac{|\nabla A|^2}{g}, z = g,$
	\begin{align*}
	W\leq c |A|^2 \frac{|\nabla A|^2}{g}+d\frac{|\nabla A|^2}{g}-2 \kappa_n \frac{n+2}{3}\frac{|\nabla A|^4}{g^2}\end{align*}
and $ Z\geq 0$, to get
	\begin{align*}
	\Big(\partial_t-\Delta\Big)\frac{|\nabla A|^2}{g^2}&\leq \frac{2}{g}\left\la \nabla g ,\nabla \left( \frac{|\nabla A|^2}{g^2}\right) \right\ra \\
	&+\frac{1}{g}\left( c |A|^2 \frac{|\nabla A|^2}{g}+d\frac{|\nabla A|^2}{g}-2 \kappa_n \frac{n+2}{3}\frac{|\nabla A|^4}{g^2}\right).
 	\end{align*}
The nonlinearity then is
	\begin{align*}
	\frac{|\nabla A|^2}{g^2} \left( c|A|^2+d-\frac{2 \kappa_n(n+2)}{3}\frac{|\nabla A|^2}{g} \right).
	\end{align*}
Since
	\begin{align*}
	g>\e_1|A|^2+\e_2,
	\end{align*}
there exists a constant $N$, such that
	\begin{align*}
	Ng\ge c|A|^2+d.
	\end{align*}
Hence, by the maximum principle, there exists a constant (with $\eta,\e_1,\e_2$ chosen sufficiently small so that N is sufficiently large, this estimate holds at the initial time), such that
	\begin{align*}
	\frac{|\nabla A|^2}{g^2}\leq \frac{3N}{2 \kappa_n (n+2)}.
	\end{align*}
Therefore, we see there exists a constant $\mathcal{C} = \frac{3N}{2 \kappa_n (n+2)}= \mathcal{C}(n, \mathcal{M}_0) $, such that
	\begin{align*}
	\frac{|\nabla A|^2}{g^2}\leq \mathcal{C}
	\end{align*}
and from the definition of $g$, we get the result of the lemma.
\end{proof}
\begin{theorem} Let $\mathcal{M}_t,t\in[0,T)$ be a solution of the mean curvature flow with surgery and normalised initial data. Then there exists constants $\gamma_3, \gamma_4$ depending only on the dimension, so that
	\begin{align}\label{eqn_HigherOrderGradEstimateA}
	|\nabla^2 A|^2 \leq \gamma_3|A|^6+\gamma_4,
	\end{align}
for any $t\in[0,T)$.
\end{theorem}
\begin{proof} We have the following evolution equation
		\begin{align*}
\Big(\partial_t-\Delta\Big)|\nabla^2 A|^2 & \leq-2|\nabla^3 A|^2+k_1|A|^2|\nabla^2 A|^2+k_2|A||\nabla A|^2|\nabla^2 A| \\
	& \leq-2|\nabla^3 A|^2+\left(k_1+\frac{k_2}{2}\right)|A|^2|\nabla^2 A|^2+\frac{k_2}{2}|\nabla A|^4 .
	\end{align*}
We now consider the evolution equation of the term $\frac{|\nabla^2 A|^2}{|H|^5}$. Firstly we see
	\begin{align*}
	\partial_t|H|^\alpha &=\Delta|H|^\alpha+\alpha|H|^{\alpha-1}(\partial_t-\Delta)|H| -\alpha(\alpha-1)|H|^{\alpha-2}|\nabla H|^2\\
	&\ge \Delta|H|^\alpha-\alpha(\alpha-1)|H|^{\alpha-2}|\nabla H|^2,
	\end{align*}
since $|\nabla|H||^2\ge|\nabla H|^2$ and $|H|>0$. Therefore, we get
	\begin{align*}
	\Big(\partial_t -\Delta\Big) \frac{|\nabla^2 A|^2}{|H|^5} & \leq \frac{1}{|H|^5}\left(-2|\nabla^3 A|^2+\left(k_1+\frac{k_2}{2}\right)|A|^2|\nabla^2 A|^2+\frac{k_2}{2}|\nabla A|^4\right) \\
	& +\frac{20|H|^3|\nabla^2 A|^2|\nabla H|^2}{|H|^{10}}+\frac{2}{|H|^{10}}\left\langle\nabla|H|^5, \nabla|\nabla^2 A|^2\right\rangle\\
	&-\frac{2|\nabla^2 A|^2}{|H|^{15}}|\nabla| H|^5|^2 .
	\end{align*}
We have the terms
	\begin{align*}
	\frac{20|H|^3|\nabla^2 A|^2|\nabla H|^2}{|H|^{10}}-\frac{2|\nabla^2 A|^2}{|H|^{15}}|\nabla| H|^5|^2 & \leq \frac{20|\nabla^2 A|^2|\nabla H|^2}{|H|^7}-\frac{50|\nabla^2 A|^2|\nabla| H||^2}{|H|^7} \\
	& \leq \frac{20|\nabla^2 A|^2|\nabla H|^2}{|H|^7} .
	\end{align*}
and
	\begin{align*}
	\frac{2}{|H|^{10}}\left\langle\nabla|H|^5, \nabla|\nabla^2 A|^2\right\rangle & =\frac{10\left\langle\nabla|H|, \nabla|\nabla^2 A|^2\right\rangle}{|H|^6} \\
	&\le\frac{20\langle|\nabla H|,|\nabla^2 A||\nabla^3 A|\rangle}{|H|^6}\\
	& \leq \frac{1}{|H|^5}|\nabla^3 A|^2+\frac{100|\nabla H|^2|\nabla^2 A|^2}{|H|^7}.
	\end{align*}
Together with the gradient estimate, Theorem \ref{thm_gradient} this gives the following evolution equation
	\begin{align*}
	\Big(\partial_t-\Delta\Big) \frac{|\nabla^2 A|^2}{|H|^5}& \leq-\frac{|\nabla^3 A|^2}{|H|^5}+k_3 \frac{|\nabla^2 A|^2}{|H|^3}+\frac{120|\nabla H|^2|\nabla^2 A|^2}{|H|^7}+\frac{k_2}{2}\frac{|\nabla A|^4}{|H|^5} \\
	& \leq-\frac{|\nabla^3 A|^2}{|H|^5}+k_4 \frac{|\nabla^2 A|^2}{|H|^3}+C_1 \frac{|\nabla^2 A|^2}{|H|^7}+\frac{k_5|H|^8+C_2}{|H|^5}.
	\end{align*}
Similar computations give us
	\begin{align*}
	& \Big(\partial_t-\Delta\Big) \frac{|\nabla A|^2}{|H|^3} \leq-\frac{|\nabla^2 A|^2}{|H|^3}+\frac{k_6|H|^8+C_3 }{|H|^5}, \\
	& \Big(\partial_t-\Delta\Big) \frac{|\nabla A|^2}{|H|^7} \leq-\frac{|\nabla^2 A|^2}{|H|^7}+\frac{k_7|H|^8+C_4 }{|H|^9} .
	\end{align*}
We now set
	\begin{align*}
	f=\frac{|\nabla^2 A|^2}{|H|^5}+N \frac{|\nabla A|^2}{|H|^3}+M \frac{|\nabla A|^2}{|H|^7}-\kappa \sqrt{c|H|-|A|^2},
	\end{align*}
and so we have
	\begin{align*}
	\Big(\partial_t -\Delta\Big) f & \leq k_4 \frac{|\nabla^2 A|^2}{|H|^3}+k_5|H|^3+C_1 \frac{|\nabla^2 A|^2}{|H|^7}+\frac{C_2}{|H|^5} \\
	& -N \frac{|\nabla^2 A|^2}{|H|^3}+N k_6|H|^3+\frac{N C_3}{|H|^5} \\
	& -\frac{M|\nabla^2 A|^2}{|H|^7}+\frac{k_7 M}{|H|}+\frac{C_4 M}{|H|}-\kappa \varepsilon_0|H|^3 .
	\end{align*}
Therefore, we choose
	\begin{align*}
	N>k_4, \quad \varepsilon_0 \kappa>N k_6+k_5, \quad M>C_1 .
	\end{align*}
Since $H>0$, there exists a constant $\alpha_1$, such that $|H| \geq \alpha_1$ and we find
	\begin{align*}
	\Big(\partial_t-\Delta\Big) f \leq C_5
	\end{align*}
which implies
	\begin{align*}
	\max _{\mathcal{M}_t} f \leq \max _{\mathcal{M}_{t_0}} f+C_5\left(t-t_0\right) .
	\end{align*}
Given the bound on the maximal time of existence, we have
	\begin{align*}
	f \leq C(n),
	\end{align*}
which implies
	\begin{align*}
	|\nabla^2 A|^2 \leq c(n)|H|^6+C(n)|H|^5.
	\end{align*}
Applying the quadratic pinching, we get \eqref{eqn_HigherOrderGradEstimateA}.	
\end{proof}
Higher order estimates on $\left|\nabla^m A\right|$ for all $m$ follow by an analogous method. Furthermore, we derive estimates on the time derivative of the second fundamental form since we have the evolution equation
	\begin{align*}
	\left|\partial_t A\right|=|\Delta A+A * A * A| \leq C|\nabla^2 A|^2+C|A|^3 \leq c_1|A|^3+c_2.
	\end{align*}
\section{Codimension Estimates}
In this section, we want to show in regions of high curvature, the submanifold becomes approximately codimension one in a quantifiable sense. Our goal is to separate the second fundamental form in the principal direction and the second fundamental form in the other directions and compute their evolution equations separately. Later, we find estimates for the reaction and gradient terms as well as for the lower order terms, which appear due to the Riemannian ambient space. Then, we start by computing the evolution equation of the quantity $\frac{|A^-|^2}{f}$, which since in the limit the background space is Euclidean, the result will follow from the maximum principle. The theorem we will prove is the following.
\begin{theorem}\label{Th1}
Let $F: \mathcal{M}^n\times[0, T) \rightarrow \mathcal{N}^{n+m}$ be a smooth solution to mean curvature flow so that
$F_0(p)=F(p, 0)$ is compact and quadratically pinched.
Then $\forall \e>0, \exists H_0 >0$, such that if $f \geq H_0$, then
	\begin{align*}
	\left|A^-\right|^2 \leq \e f+C_{\e}
	\end{align*}
$\forall t \in[0, T)$ where $C_\e=C_{\e}(n, m)$.
\end{theorem}
\subsection{The Evolution Equation of $|A^-|^2$}
We start by computing the evolution equation of $|A^-|^2$. We define the tensor $A^-$ by
	\begin{align*}
	A^-(X,Y)=A(X,Y)-\frac{\langle A(X,Y),H\rangle}{|H|^2}H,
	\end{align*}
for vector fields $X,Y$ tangent to $\mathcal{M}_t$. The tensor $A^-$ is well defined, since $|H|>0$. Therefore, we will need to compute the evolution equations of $|A|^2$ and $\frac{|\langle A,H\rangle|^2}{|H|^2}.$ Using \eqref{eqn_|H|^2} and the quotient rule, we have
	\begin{align*}
	\Big(\partial_t&-\Delta\Big)\frac{\sum_{i,j}|\langle A_{ij},H\rangle|^2}{|H|^2}=|H|^{-2}\Big(\partial_t-\Delta\Big)\sum_{i,j}|\langle A_{ij},H\rangle|^2\\
	&+2|H|^{-2}\sum_{k}\Big\langle \nabla_k |H|^2,\nabla_k \frac{\sum_{i,j}|\langle A_{ij},H\rangle|^2}{|H|^2}\Big\rangle\\
	&-|H|^{-4}\sum_{i,j}|\langle A_{ij},H\rangle|^2\Big(-2|\nabla^\bot H|^2+2\sum_{i,j}|\langle A_{ij},H\rangle|^2+2\sum_{k,\alpha,\beta}\bar{R}_{k\alpha k\beta} H^\alpha H^\beta\Big).
	\end{align*}
Before computing the evolution equation of $\sum_{i,j}|\langle A_{ij},H\rangle|^2$, we simplify the other terms. In particular, using $\sum_{i,j}|\langle A_{ij},H\rangle|^2=|H|^2|h|^2$ and
	\begin{align*}
	|\nabla^\bot H|^2=|H|^2|\nabla^\bot \nu_1|^2+|\nabla |H||^2,
	\end{align*}
we write
	\begin{align*}
	2|H|^{-4}\sum_{i,j}|\langle A_{ij},H\rangle|^2|\nabla^\bot H|^2=2|h|^2|\nabla^\bot \nu_1|^2+2|H|^{-2}|h|^2|\nabla |H||^2,
	\end{align*}
	\begin{align*}
	-2|H|^{-4}\sum_{i,j}|\langle A_{ij},H\rangle|^4=-2|h|^4,
	\end{align*}
	\begin{align*}
	2|H|^{-4}\sum_{i,j}|\langle A_{ij},H\rangle|^2\sum_{k,\alpha,\beta}\bar{R}_{k\alpha k\beta} H^\alpha H^\beta=2|h|^2|H|^{-2}\sum_{k,\alpha,\beta}\bar{R}_{k\alpha k\beta}H^\alpha H^\beta.
	\end{align*}
As for the remaining gradient terms, we have
	\begin{align*}
	\nabla_k |H|^2=2\langle \nabla_k^\bot H,H\rangle
	\end{align*}
and
	\begin{align*}
	\nabla_k (|H|^{-2}\sum_{i,j}|\langle A_{ij},H\rangle|^2)=\nabla_k |h|^2=2\sum_{i,j}h_{ij} \nabla_k h_{ij}.
	\end{align*}
Therefore, since $H=|H|\nu_1$ and $\langle\nabla^\bot_k \nu_1,\nu_1\rangle=0$, we have
	\begin{align*}
	2|H|^{-2}\sum_k\Big\langle\nabla_k |H|^2,\nabla_k\frac{\sum_{i,j}|\langle A_{ij},H\rangle|^2}{|H|^2}\Big\rangle&=8|H|^{-2}\sum_{i,j,k}\langle\nabla^\bot_k H,H\rangle h_{ij}\nabla_k h_{ij}\\
	&=8|H|^{-1}\sum_{i,j,k} \nabla_k |H|h_{ij}\nabla_k h_{ij}.
	\end{align*}
To summarise, we have shown so far that
	\begin{align*}
	\Big(\partial_t-\Delta\Big)\frac{\sum_{i,j}|\langle A_{ij},H\rangle|^2}{|H|^2}&=|H|^{-2}\Big(\partial_t-\Delta\Big)\sum_{i,j}|\langle A_{ij},H\rangle|^2-2|h|^4+2|h|^2|\nabla^\bot_k \nu_1|^2\\
	&+2|H|^{-2}|h|^2|\nabla|H||^2+8|H|^{-1}\sum_{i,j,k} \nabla_k |H|h_{ij}\nabla_k h_{ij}\\
	&-2|h|^2|H|^{-2}\sum_{k,\alpha,\beta}\bar{R}_{k\alpha k\beta} H^\alpha H^\beta.
	\end{align*}
For the evolution equation of $\langle A_{ij},H\rangle$, we have the following lemma.
\begin{lemma}\label{B}
The evolution equation of $|\langle A_{ij},H\rangle|^2$ is
	\begin{align*}
	|H|^{-2}\Big(\partial_t-\Delta\Big)|\langle A_{ij},H\rangle|^2&=4|\mathring{h}_{ij}A_{ij}^- |^2+2|R_{ij}^\bot (\nu_1)|^2+4|h|^4-4|H|^{-1}\mathring{h}_{ij}\nabla_k |H|\langle\nabla^\bot_k A_{ij}^- ,\nu_1\rangle\\
	&-4\mathring{h}_{ij}\langle\nabla^\bot_k A_{ij}^- ,\nabla^\bot_k\nu_1\rangle-4|h|^2 |\nabla^\bot_k\nu_1|^2-2|H|^{-2}|h|^2|\nabla|H||^2\\
	&-8|H|^{-1} \nabla_k |H|h_{ij}\nabla_k h_{ij}-2|\nabla h|^2+2B'\\
	&-2|\bar{R}_{ij}(\nu_1)|^2-4\langle\bar{R}_{ij}(\nu_1),\mathring{h}_{ip}A^-_{jp}-\mathring{h}_{jp}A^-_{ip}\rangle,
	\end{align*}
where
	\begin{align*}
	B'&:=2|H|^{-2}\bar{R}_{ipjq}\langle A_{pq},H\rangle\langle A_{ij},H\rangle-2|H|^{-2}\bar{R}_{kjkp} \langle A_{pi},H\rangle\langle A_{ij},H\rangle\\
	&+|H|^{-2}A^\alpha_{ij}\bar{R}_{k\alpha k\beta}H^\beta\langle A_{ij},H\rangle+|H|^{-2}H^\alpha\bar{R}_{k\alpha k\beta} A_{ij}^\beta \langle A_{ij},H\rangle\\
	&-2|H|^{-2}A^\alpha_{jp} \bar{R}_{ip\alpha\beta} H^\beta\langle A_{ij},H\rangle -2|H|^{-2}A^\alpha_{ip}\bar{R}_{jp\alpha\beta} H^\beta \langle A_{ij},H\rangle\\
	&+|H|^{-2}\bar{\nabla}_k \bar{R}_{kij\beta} H^\beta\langle A_{ij},H\rangle-|H|^{-2}\bar{\nabla}_i \bar{R}_{jkk\beta} H^\beta\langle A_{ij},H\rangle.
	\end{align*}
\end{lemma}
\begin{proof}
Whenever $h$ is traced with $A^-$ or its derivative, we may replace $h$ with $\mathring{h}$, because $A^-$ is traceless. Also, for simplicity, we avoid the summation notation. To begin with, using \eqref{eqn_A}, we substitute formulas
	\begin{align*}
	\Big\langle\Big(\partial_t-\Delta\Big)^\bot A_{ij},H\Big\rangle&=\langle A_{ij},A_{pq}\rangle \langle A_{pq},H\rangle+\langle A_{iq},A_{pq}\rangle \langle A_{pj},H\rangle+\langle A_{jq},A_{pq}\rangle\langle A_{pi},H\rangle\\
	&-2\langle A_{ip},A_{jq}\rangle \langle A_{pq}, H\rangle+2\bar{R}_{ipjq}\langle A_{pq},H\rangle-\bar{R}_{kjkp} \langle A_{pi},H\rangle-\bar{R}_{kikp}\langle A_{pj},H\rangle\\
	&+A^\alpha_{ij}\bar{R}_{k\alpha k\beta} \langle\nu_\beta, H\rangle-2A^\alpha_{jp}\bar{R}_{ip\alpha\beta}\langle \nu_\beta, H\rangle-2A^\alpha_{ip}\bar{R}_{jp\alpha\beta}\langle\nu_\beta, H\rangle\\
	&+\bar{\nabla}_k \bar{R}_{kij\beta}\langle\nu_\beta, H\rangle-\bar{\nabla}_i \bar{R}_{jkk\beta}\langle\nu_\beta, H\rangle,\\
	\Big\langle A_{ij} ,\Big(\partial_t-\Delta\Big)^\bot H\Big\rangle&=\langle A_{pq}, H\rangle\langle A_{pq},A_{ij}\rangle+H^\alpha\bar{R}_{k\alpha k\beta}\langle \nu_\beta, A_{ij}\rangle.
	\end{align*}
Tracing each of the equations with a copy of $\langle A_{ij},H\rangle$, we get
	\begin{align*}
	\Big\langle\Big(\partial_t-\Delta\Big)^\bot A_{ij},H\Big\rangle\langle A_{ij},H\rangle&=\langle A_{ij},A_{pq}\rangle \langle A_{pq},H\rangle\langle A_{ij},H\rangle+2\langle A_{iq},A_{pq}\rangle \langle A_{pj},H\rangle\langle A_{ij},H\rangle\\
	&-2\langle A_{ip},A_{jq}\rangle \langle A_{pq},H\rangle\langle A_{ij},H\rangle+2\bar{R}_{ipjq}\langle A_{pq},H\rangle\langle A_{ij},H\rangle\\
	&-2\bar{R}_{kjkp} \langle A_{pi},H\rangle\langle A_{ij},H\rangle+A^\alpha_{ij}\bar{R}_{k\alpha k\beta}\langle \nu_\beta, H\rangle\langle A_{ij},H\rangle\\
	&-2A^\alpha_{jp} \bar{R}_{ip\alpha\beta}\langle \nu_\beta, H\rangle\langle A_{ij},H\rangle -2A^\alpha_{ip}\bar{R}_{jp\alpha\beta}\langle \nu_\beta, H\rangle \langle A_{ij},H\rangle\\
	&+\bar{\nabla}_k \bar{R}_{kij\beta}\langle\nu_\beta, H\rangle\langle A_{ij},H\rangle-\bar{\nabla}_i \bar{R}_{jkk\beta}\langle\nu_\beta, H\rangle\langle A_{ij},H\rangle,\\
	\Big\langle A_{ij} ,\Big(\partial_t-\Delta\Big)^\bot H\Big\rangle\langle A_{ij},H\rangle&=\langle A_{pq}, H\rangle\langle A_{pq},A_{ij}\rangle\langle A_{ij},H\rangle+H^\alpha\bar{R}_{k\alpha k\beta}\langle\nu_\beta, A_{ij}\rangle\langle A_{ij},H\rangle.
	\end{align*}
Putting the above equations together and keeping in mind that $\langle \nu_\beta,H\rangle=H^\beta$ we have,
	\begin{align*}
	\Big(\Big(\partial_t-\Delta\Big)\langle A_{ij},H\rangle\Big)\langle A_{ij},H\rangle&=2\langle A_{ij},A_{pq}\rangle \langle A_{pq},H\rangle\langle A_{ij},H\rangle+2\langle A_{iq},A_{pq}\rangle \langle A_{pj},H\rangle\langle A_{ij},H\rangle\\
	&-2\langle A_{ip},A_{jq}\rangle \langle A_{pq},H\rangle\langle A_{ij},H\rangle +2\bar{R}_{ipjq}\langle A_{pq},H\rangle\langle A_{ij},H\rangle\\
	&-2\bar{R}_{kjkp} \langle A_{pi},H\rangle\langle A_{ij},H\rangle+A^\alpha_{ij}\bar{R}_{k\alpha k\beta} H^\beta\langle A_{ij},H\rangle\\
	&-2A^\alpha_{jp} \bar{R}_{ip\alpha\beta} H^\beta\langle A_{ij},H\rangle -2A^\alpha_{ip}\bar{R}_{jp\alpha\beta} H^\beta \langle A_{ij},H\rangle\\
	&+\bar{\nabla}_k \bar{R}_{kij\beta} H^\beta\langle A_{ij},H\rangle-\bar{\nabla}_i \bar{R}_{jkk\beta}H^\beta\langle A_{ij},H\rangle\\
	&+H^\alpha\bar{R}_{k\alpha k\beta} A_{ij}^\beta \langle A_{ij},H\rangle-2\langle \nabla^\bot_k A_{ij},\nabla^\bot_k H\rangle\langle A_{ij},H\rangle.
	\end{align*}
Define
	\begin{align*}
	B&:=2\bar{R}_{ipjq}\langle A_{pq},H\rangle\langle A_{ij},H\rangle-2\bar{R}_{kjkp} \langle A_{pi},H\rangle\langle A_{ij},H\rangle\\
	&+A^\alpha_{ij}\bar{R}_{k\alpha k\beta} H^\beta\langle A_{ij},H\rangle+H^\alpha\bar{R}_{k\alpha k\beta} A_{ij}^\beta \langle A_{ij},H\rangle\\
	&-2A^\alpha_{jp} \bar{R}_{ip\alpha\beta} H^\beta\langle A_{ij},H\rangle -2A^\alpha_{ip}\bar{R}_{jp\alpha\beta} H^\beta \langle A_{ij},H\rangle\\
	&+\bar{\nabla}_k \bar{R}_{kij\beta} H^\beta \langle A_{ij},H\rangle-\bar{\nabla}_i \bar{R}_{jkk\beta}H^\beta\langle A_{ij},H\rangle.
	\end{align*}
We use the Uhlenbeck's trick to suppose that we are in an orthogonal frame. That is, suppose $g^{ij}=\delta_{ij}$ remains orthogonal along the flow. More precisely, for any $e_i,e_j$ orthonormal, we have
	\begin{align*}
	\partial_t g^{ij}=\partial_t \langle e_i,e_j\rangle=0.
	\end{align*}
Therefore, excluding the time derivative of the inverse of the metric, which is the term
	\begin{align*}
	2\big(\partial_tg^{ij}\big)g^{pq}\langle A_{ip},H\rangle\langle A_{jq},H\rangle,
	\end{align*}
we have
	\begin{align}\label{eqn_AH}
	\Big(\partial_t-\Delta\Big)|\langle A_{ij},H\rangle|^2&=2\Big(\Big(\partial_t-\Delta\Big)\langle A_{ij},H\rangle\Big)\langle A_{ij},H\rangle-2|\nabla\langle A_{ij},H\rangle|^2\nonumber\\
	&=4\langle A_{ij},A_{pq}\rangle \langle A_{pq},H\rangle\langle A_{ij},H\rangle+4\langle A_{iq},A_{pq}\rangle \langle A_{pj},H\rangle\langle A_{ij},H\rangle\nonumber\\
	&-4\langle A_{ip},A_{jq}\rangle \langle A_{pq},H\rangle\langle A_{ij},H\rangle-4\langle\nabla^\bot_k A_{ij},\nabla_k^\bot H\rangle\langle A_{ij},H\rangle\nonumber\\
	&-2|\nabla\langle A_{ij},H\rangle|^2+2B.
	\end{align}
To finish the proof, we multiply $|H|^{-2}$ and then rewrite each of the remaining terms using $A=A^-+h\nu_1$. For the first term on the first line of \eqref{eqn_AH}, we have
	\begin{align}\label{eqn2}
	4|H|^{-2}\langle A_{ij},A_{pq}\rangle \langle A_{pq},H\rangle\langle A_{ij},H\rangle&=4|H|^{-2}|H|^2h_{ij} h_{pq}\langle A_{ij},A_{pq}\rangle\nonumber\\
	&=4|h|^4+4h_{ij} h_{pq}\langle A_{ij}^- ,A_{pq}^- \rangle\nonumber\\
	&=4|h|^4+4\mathring{h}_{ij}\mathring{h}_{pq}\langle A_{ij}^- ,A_{pq}^- \rangle\nonumber\\
	&=4|h|^4+4|\mathring{h}_{ij}A_{ij}^- |^2.
	\end{align}
Also, B can be rewritten as
	\begin{align*}
	B'&:=2|H|^{-2}\bar{R}_{ipjq}\langle A_{pq},H\rangle\langle A_{ij},H\rangle-2|H|^{-2}\bar{R}_{kjkp} \langle A_{pi},H\rangle\langle A_{ij},H\rangle\\
	&+|H|^{-2}A^\alpha_{ij}\bar{R}_{k\alpha k\beta}H^\beta\langle A_{ij},H\rangle+|H|^{-2}H^\alpha\bar{R}_{k\alpha k\beta} A_{ij}^\beta \langle A_{ij},H\rangle\\
	&-2|H|^{-2}A^\alpha_{jp} \bar{R}_{ip\alpha\beta} H^\beta\langle A_{ij},H\rangle -2|H|^{-2}A^\alpha_{ip}\bar{R}_{jp\alpha\beta} H^\beta \langle A_{ij},H\rangle\\
	&+|H|^{-2}\bar{\nabla}_k \bar{R}_{kij\beta} H^\beta\langle A_{ij},H\rangle-|H|^{-2}\bar{\nabla}_i \bar{R}_{jkk\beta} H^\beta\langle A_{ij},H\rangle.
	\end{align*}
In higher codimension, the fundamental Gauss, Codazzi and Ricci equations on Riemannian manifold in local frame take the form
	\begin{align*}
	R_{ijpq}=\bar{R}_{ijpq}+A^\alpha_{ip}A^\alpha_{jq}-A^\alpha_{jp}A^\alpha_{iq},
	\end{align*}
	\begin{align*}
	(\nabla^\bot_i A)^\alpha_{jp}-(\nabla^\bot_j A)^\alpha_{ip}=-\bar{R}_{ijp\alpha},
	\end{align*}
and
	\begin{align*}
	R^\bot_{ij\alpha\beta}=\bar{R}_{ij\alpha\beta}+A^\alpha_{ip}A^\beta_{jp}-A^\beta_{ip}A^\alpha_{jp}.
	\end{align*}
Define a vector-valued version of the normal curvature by
	\begin{align}
	R^\bot_{ij}(\nu_\alpha)=R^\bot_{ij\alpha\beta}\nu_\beta= \bar{R}_{ij\alpha\beta}\nu_\beta+A^\alpha_{ip}A^\beta_{jp}-A^\beta_{ip}A^\alpha_{jp}\nu_\beta.
	\end{align}
In particular, we note that $R^\bot_{ij}(\nu_1)=\bar{R}_{ij}(\nu_1)+h_{ip}A^\beta_{jp}-{h}_{jp}A^\beta_{ip}$, which in view of
	\begin{align*}
	A_{ij}=A^-_{ij}+h_{ij}\nu_1=A^-_{ij}+\mathring{h}_{ij}\nu_1+\frac{1}{n}|H|g_{ij}\nu_1,
	\end{align*}
gives
	\begin{align}\label{eqn_Rbot}
	R^\bot_{ij}(\nu_1)=\bar{R}_{ij}(\nu_1)+\mathring{h}_{ip}A^-_{jp}-\mathring{h}_{jp}A^-_{ip}.
	\end{align}
For the difference of second and third term of \eqref{eqn_AH}, we notice the resemblance to $|R_{ij}^\bot (\nu_1)|^2$ in \eqref{eqn_Rbot}. We compute
	\begin{align}\label{eqn3}
	|\mathring{h}_{ip}A_{jp}^- -\mathring{h}_{jp}A_{ip}^- |^2&=|h_{ip}A_{jp} -h_{jp}A_{ip}|^2=\langle h_{ip}A_{jp} -h_{jp}A_{ip} ,h_{iq}A_{jq} -h_{jq}A_{iq}\rangle\nonumber\\
	&=2h_{ip}h_{iq}\langle A_{jp},A_{jq}\rangle-2h_{ip}h_{jq}\langle A_{jp},A_{iq}\rangle\nonumber\\
	&=2|H|^{-2}\big(\langle A_{jp},A_{jq}\rangle\langle A_{ip},H\rangle\langle A_{iq},H\rangle-\langle A_{jp},A_{iq}\rangle\langle A_{ip},H\rangle\langle A_{jq},H\rangle\big).
	\end{align}
Therefore,
	\begin{align*}
	|R_{ij}^\bot (\nu_1)|^2&= |\bar{R}_{ij}(\nu_1)|^2+2|H|^{-2}\big(\langle A_{jp},A_{jq}\rangle\langle A_{ip},H\rangle\langle A_{iq},H\rangle-\langle A_{jp},A_{iq}\rangle\langle A_{ip},H\rangle\langle A_{jq},H\rangle\big)\\
	&+2\langle \bar{R}_{ij}(\nu_1),\mathring{h}_{ip}A^-_{jp}-\mathring{h}_{jp}A^-_{ip}\rangle.
	\end{align*}
After reindexing (e.g. $j \to p\to q\to i\to j$ on the second term and $j\to i\to q\to j, p\to p$ on the third term), this gives
	\begin{align*}
	2|R_{ij}^\bot (\nu_1)|^2&=2|\bar{R}_{ij}(\nu_1)|^2+4|H|^{-2}\big(\langle A_{ip},A_{pq}\rangle\langle A_{jq},H\rangle\langle A_{ij},H\rangle-\langle A_{ip},A_{jq}\rangle\langle A_{pq},H\rangle\langle A_{ij},H\rangle\big)\\
	&+4\langle\bar{R}_{ij}(\nu_1),\mathring{h}_{ip}A^-_{jp}-\mathring{h}_{jp}A^-_{ip}\rangle.
	\end{align*}
Thus, we have shown the reaction terms of our lemma statement are correct. For the gradient terms, it follows from the identities
	\begin{align*}
	\langle\nabla^\bot_k A_{ij},\nu_1\rangle=\langle\nabla^\bot_k A_{ij}^- ,\nu_1\rangle+\nabla_k h_{ij},
	\end{align*}
	\begin{align}\label{eqn4}
	\langle\nabla^\bot_k A_{ij},\nabla^\bot_k \nu_1\rangle=\langle\nabla^\bot_k A_{ij}^- ,\nabla^\bot_k \nu_1\rangle+h_{ij}|\nabla_k^\bot \nu_1|^2,
	\end{align}
	\begin{align*}
	\nabla^\bot_k H=\nabla_k |H|\nu_1+|H|\nabla^\bot_k \nu_1.
	\end{align*}
Therefore, we have
	\begin{align}\label{eqn5}
	-4|H|^{-2}\langle\nabla^\bot_k A_{ij},\nabla^\bot_k H\rangle\langle A_{ij},H\rangle&=-4|H|^{-1}h_{ij} \nabla_k |H|\langle\nabla^\bot_k A_{ij},\nu_1\rangle-4|H|^{-1}h_{ij}\langle\nabla^\bot_k A_{ij},\nabla^\bot_k \nu_1\rangle\nonumber\\
	&=-4|H|^{-1}\mathring{h}_{ij}\nabla_k |H|\langle\nabla^\bot_k A_{ij}^- ,\nu_1\rangle-4|H|^{-1}h_{ij} \nabla_k |H|\nabla_k h_{ij}\nonumber\\
	&-4\mathring{h}_{ij}\langle\nabla^\bot_k A_{ij}^- ,\nabla^\bot_k \nu_1\rangle-4|h|^2|\nabla_k^\bot\nu_1|^2,\nonumber\\
	&\\
	-2|H|^{-2}|\nabla\langle A_{ij},H\rangle|^2&=-2|H|^{-2}|\nabla(|H|h_{ij})|^2\nonumber\\
	&=-2|H|^{-2}|h|^2|\nabla|H||^2-2|\nabla h|^2-4|H|^{-1}h_{ij} \nabla_k |H|\nabla_k h_{ij}\nonumber,
	\end{align}
since $A^-_{ii}=0$, meaning that it's trace free. Combining \eqref{eqn2}-\eqref{eqn5}, we get the desired result.
\end{proof}
Substituting the result of the above lemma into our equation for the evolution of $|H|^{-2}\sum_{i,j}|\langle A_{ij},H\rangle|^2$ and combining like terms, we have
	\begin{align*}
	\Big(\partial_t-\Delta\Big)\frac{\sum_{i,j}|\langle A_{ij},H\rangle|^2}{|H|^2}&=4\sum_{i,j}|\mathring{h}_{ij}A_{ij}^- |^2+2\sum_{i,j}|R_{ij}^\bot (\nu_1)|^2+2|h|^4\\
	&-4|H|^{-1}\sum_{i,j,k}\mathring{h}_{ij}\nabla_k |H|\langle\nabla^\bot_k A_{ij}^- ,\nu_1\rangle-4\sum_{i,j,k}\mathring{h}_{ij}\langle\nabla^\bot_k A_{ij}^- ,\nabla^\bot_k\nu_1\rangle\\
	&-2|h|^2\sum_k  |\nabla^\bot_k\nu_1|^2-2|\nabla h|^2+2B'-2|h|^2|H|^{-2}\sum_{k,\alpha,\beta}\bar{R}_{k\alpha k\beta} H^\alpha H^\beta\\
	&-2\sum_{i,j}|\bar{R}_{ij}(\nu_1)|^2-4\sum_{i,j,p}\langle \bar{R}_{ij}(\nu_1),\mathring{h}_{ip}A^-_{jp}-\mathring{h}_{jp}A^-_{ip}\rangle.
	\end{align*}
We negate the expression above, add in the evolution equation of $|A|^2$ and use \eqref{eqn_Pa} to get
	\begin{align*}
	\Big(\partial_t-\Delta\Big)|A^-|^2&=-2|\nabla^\bot A|^2+2\sum_{i,j,p,q}|\langle A_{ij},A_{pq}\rangle|^2+2\sum_{i,j}|R_{ij}^\bot |^2+\Big(P_\alpha-2B'\Big)\\
	&-4\sum_{i,j}|\mathring{h}_{ij}A_{ij}^- |^2-2\sum_{i,j}|R_{ij}^\bot (\nu_1)|^2-2|h|^4+4|H|^{-1}\sum_{i,j,k}\mathring{h}_{ij}\nabla_k |H|\langle\nabla^\bot_k A_{ij}^- ,\nu_1\rangle\\
	&+4\sum_{i,j,k}\mathring{h}_{ij}\langle\nabla^\bot_k A_{ij}^- ,\nabla^\bot_k \nu_1\rangle+2|\nabla h|^2+2|h|^2\sum_k |\nabla^\bot_k\nu_1|^2\\
	&+2|h|^2|H|^{-2}\sum_{k,\alpha,\beta}\bar{R}_{k\alpha k\beta} H^\alpha H^\beta+2\sum_{i,j}|\bar{R}_{ij}(\nu_1)|^2\\
	&+4\sum_{i,j,p}\langle\bar{R}_{ij}(\nu_1),\mathring{h}_{ip}A^-_{jp}-\mathring{h}_{jp}A^-_{ip}\rangle.
	\end{align*}
Taking the term $2|H|^{-2}H^\alpha\sum_{k,\alpha,\beta}\bar{R}_{k\alpha k\beta} A_{ij}^\beta \langle A_{ij},H\rangle$ out of $2B'$ and the last term of the evolution equation of $\frac{\sum_{i,j}|\langle A_{ij},H\rangle|^2}{|H|^2}$, we have
	\begin{align*}
	&2|H|^{-2}\sum_{i,j,k,\alpha,\beta}\bar{R}_{k\alpha k\beta}H^\alpha A_{ij}^\beta \langle A_{ij},H\rangle-2|h|^2|H|^{-2}\sum_{k,\alpha,\beta} \bar{R}_{k\alpha k\beta} H^\alpha H^\beta\\
	&=2|H|^{-2}\sum_{i,j,k,\alpha,\beta}\bar{R}_{k\alpha k\beta} H^\alpha\big(A^{-,\beta}_{ij}+\frac{|\langle A_{ij},H\rangle|}{|H|^2}H^\beta\big) \langle A_{ij},H\rangle-2|h|^2|H|^{-2}\sum_{k,\alpha,\beta} \bar{R}_{k\alpha k\beta} H^\alpha H^\beta\\
	&=2|H|^{-2}\sum_{i,j,k,\alpha,\beta\ge 2}\bar{R}_{k\alpha k\beta} H^\alpha A^{\beta}_{ij} \langle A_{ij},H\rangle.
	\end{align*}
The reaction terms satisfy
	\begin{align*}
	2\sum_{i,j,p,q}|\langle A_{ij},A_{pq}\rangle|^2-4\sum_{i,j}|\mathring{h}_{ij}A_{ij}^- |^2-2|h|^4=2\sum_{i,j,p,q}|\langle A_{ij}^- ,A_{pq}^- \rangle|^2,
	\end{align*}
	\begin{align}\label{R-R}
	2\sum_{i,j}|R_{ij}^\bot |^2-2\sum_{i,j}|R_{ij}^\bot (\nu_1)|^2&=2|\hat{R}^\bot|^2+2\sum_{i,j}|R_{ij}^\bot (\nu_1)|^2.
	\end{align}
where
	\begin{align}\label{hatR}
	|\hat{R}^\bot|^2=\sum_{i,j,\alpha,\beta\ge 2}\Big(\sum_{p}|A_{ip}^\alpha A_{jp}^\beta -A_{jp}^\alpha A_{ip}^\beta |^2+|\bar{R}_{ij\alpha\beta}|^2+2\sum_{p}\langle \bar{R}_{ij\alpha\beta},A_{ip}^{\alpha}A_{jp}^{\beta}-A_{jp}^{\alpha}A_{ip}^{\beta}\rangle\Big)
	\end{align}
As for the gradient terms, taking the form of $\nabla^\bot_i A_{jp}=\nabla^\bot_i A_{jp}^- +\nabla_i h_{jp}\nu_1+h_{jp}\nabla^\bot_i \nu_1$, we see
	\begin{align*}
	|\nabla^\bot A|^2=|\nabla^\bot A^-|^2+|\nabla h|^2+|h|^2|\nabla^\bot \nu_1|^2+2\sum_{i,j,k}\mathring{h}_{ij}\langle\nabla^\bot A_{ij}^- ,\nabla^\bot_k \nu_1\rangle+2\sum_{i,j,k}\nabla_k \mathring{h}_{ij}\langle\nabla^\bot_k A_{ij}^- ,\nu_1\rangle.
	\end{align*}
Thus,
	\begin{align*}
	-2|\nabla^\bot A|^2+2|\nabla h|^2+2|h|^2|\nabla^\bot \nu_1|^2+4\sum_{i,j,k}\mathring{h}_{ij}\langle\nabla^\bot_k A_{ij}^- ,\nabla^\bot_k \nu_1\rangle&=-2|\nabla^\bot A^-|^2\\
	&-4\sum_{i,j,k}\nabla_k \mathring{h}_{ij}\langle\nabla^\bot_k A_{ij}^- ,\nu_1\rangle.
	\end{align*}
Putting this all together gives
	\begin{align*}
	\Big(\partial_t-\Delta\Big)|A^-|^2&=2\sum_{i,j,p,q}|\langle A_{ij}^- ,A_{pq}^- \rangle|^2+2|\hat{R}^\bot|^2+2\sum_{i,j}|R_{ij}^\bot (\nu_1)|^2\\
	&-2|\nabla^\bot A^-|^2+4|H|^{-1}\sum_{i,j,k}\mathring{h}_{ij}\nabla_k |H|\langle\nabla^\bot_k A_{ij}^- ,\nu_1\rangle-4\sum_{i,j,k}\nabla_k\mathring{h}_{ij}\langle\nabla^\bot_k A_{ij}^- ,\nu_1\rangle\\
	&+2|H|^{-2}\sum_{i,j,k,\alpha,\beta\ge 2}\bar{R}_{k\alpha k\beta} H^\alpha A^{\beta}_{ij}\langle A_{ij},H\rangle+\Big(P_\alpha-2B''\Big)\\
	&+2\sum_{i,j}|\bar{R}_{ij}(\nu_1)|^2+4\sum_{i,j,p}\langle\bar{R}_{ij}(\nu_1),\mathring{h}_{ip}A^-_{jp}-\mathring{h}_{jp}A^-_{ip}\rangle,
	\end{align*}
where
	\begin{align*}
	B''&:=2|H|^{-2}\sum_{i,j,p,q}\bar{R}_{ipjq}\langle A_{pq},H\rangle\langle A_{ij},H\rangle-2|H|^{-2}\sum_{i,j,p,k}\bar{R}_{kjkp} \langle A_{pi},H\rangle\langle A_{ij},H\rangle\\
	&+|H|^{-2}\sum_{i,j,k,\alpha,\beta}A^\alpha_{ij}\bar{R}_{k\alpha k\beta}H^\beta\langle A_{ij},H\rangle-4|H|^{-2}\sum_{i,j,p,\alpha,\beta}A^\alpha_{ip}\bar{R}_{jp\alpha\beta} H^\beta \langle A_{ij},H\rangle\\
	&+|H|^{-2}\sum_{i,j,k,\beta}\bar{\nabla}_k \bar{R}_{kij\beta} H^\beta\langle A_{ij},H\rangle-|H|^{-2}\sum_{i,j,k,\beta}\bar{\nabla}_i \bar{R}_{jkk\beta} H^\beta\langle A_{ij},H\rangle.
	\end{align*}
and we let
	\begin{align*}
	P_\alpha&= 4\sum_{i,j,p,q}\bar{R}_{ipjq}\big(\sum_{\alpha} A^\alpha_{pq}A^\alpha_{ij}\big)-4\sum_{j,k,p}\bar{R}_{kjkp}\big(\sum_{i,\alpha} A^\alpha_{pi}A^\alpha_{ij}\big)+2\sum_{k,\alpha,\beta}\bar{R}_{k\alpha k\beta}\big(\sum_{i,j} A^\alpha_{ij}A_{ij}^\beta \big)\nonumber\\
	&-8\sum_{j,p,\alpha,\beta}\bar{R}_{jp\alpha\beta}\big(\sum_iA^\alpha_{ip}A_{ij}^\beta \big)+2\sum_{i,j,k,\beta}\bar{\nabla}_k\bar{R}_{kij\beta}A_{ij}^\beta -2\sum_{i,j,k,\beta}\bar{\nabla}_i\bar{R}_{jkk\beta}A_{ij}^\beta,
	\end{align*}
to be the lower order terms appearing in \eqref{eqn_|A|^2}.
Because $\langle A_{ij}^-, \nu_1 \rangle = 0 $, differentiating with respect to $\nabla_k$ gives
	\begin{align*}
	\langle\nabla^\bot_k A_{ij}^- ,\nu_1\rangle&=-\langle A_{ij}^- ,\nabla^\bot_k \nu_1\rangle=-\langle\mathring{A}_{ij},\nabla^\bot_k \nu_1\rangle.
	\end{align*}
Since $\mathring{h}_{ij} = \langle \mathring{A}_{ij}, \nu_1 \rangle$, from the equation above, we get
	\begin{align*}
	\nabla_k\mathring{h}_{ij}&= \langle\nabla^\bot_k \mathring{A}_{ij},\nu_1\rangle + \langle\mathring{A}_{ij},\nabla^\bot_k \nu_1\rangle =\langle\nabla^\bot_k \mathring{A}_{ij},\nu_1\rangle-\langle\nabla^\bot_kA_{ij}^- ,\nu_1\rangle.
	\end{align*}
To simplify our final expression, let us define the tensor
	\begin{align*}
	Q_{ijk}:=\langle\nabla^\bot_k\mathring{A}_{ij},\nu_1\rangle-\langle\nabla^\bot_k A_{ij}^- ,\nu_1\rangle-|H|^{-1}\mathring{h}_{ij}\nabla_k |H|.
	\end{align*}
Here we have the lower order terms in the evolution equation for the evolution of $|A^-|^2$. We match them to the evolution of the pinching quantity $f>0$. For the term $P_\alpha-2B''$, we have
	\begin{align*}
	P_\alpha-2B''&=4\sum_{i,j,p,q}\bar{R}_{ipjq}\big(\sum_{\alpha\ge 2} A^\alpha_{pq}A^\alpha_{ij}\big)-4\sum_{j,k,p}\bar{R}_{kjkp}\big(\sum_{i,\alpha\ge 2} A^\alpha_{pi}A^\alpha_{ij}\big)\\
	&+2\sum_{k,\alpha,\beta\ge 2}\bar{R}_{k\alpha k\beta}\big(\sum_{i,j} A^\alpha_{ij}A_{ij}^\beta \big)-8\sum_{j,p,\alpha,\beta\ge 2}\bar{R}_{jp\alpha\beta}\big(\sum_iA^\alpha_{ip}A_{ij}^\beta \big)\\
	&+2\sum_{i,j,k,\beta\ge 2}\bar{\nabla}_k\bar{R}_{kij\beta}A_{ij}^\beta -2\sum_{i,j,k,\beta\ge 2}\bar{\nabla}_i\bar{R}_{jkk\beta}A_{ij}^\beta.
	\end{align*}
In conclusion, according to Theorem \ref{B} and \eqref{eqn_Pa}, we get the following proposition.
\begin{proposition}\label{eqnof|A^-|^2}
The evolution equation of $|A^-|^2$ is
	\begin{align*}
	\Big(\partial_t-\Delta\Big)|A^-|^2&=2\sum_{i,j,p,q}|\langle A_{ij}^- ,A_{pq}^- \rangle|^2+2|\hat{R}^\bot|^2+2\sum_{i,j}|R_{ij}^\bot (\nu_1)|^2\\
	&-2|\nabla^\bot A^-|^2+4\sum_{i,j,k}Q_{ijk}\langle A_{ij}^- ,\nabla^\bot_k \nu_1\rangle\\
	&+2|H|^{-2}\sum_{i,j,k,\alpha,\beta\ge 2}\bar{R}_{k\alpha k\beta} H^\alpha A^{\beta}_{ij}\langle A_{ij},H\rangle+\Big(P_\alpha-2B''\Big)\\
	&+2\sum_{i,j}|\bar{R}_{ij}(\nu_1)|^2+4\sum_{i,j,p}\langle\bar{R}_{ij}(\nu_1),\mathring{h}_{ip}A^-_{jp}-\mathring{h}_{jp}A^-_{ip}\rangle,
	\end{align*}
where
	\begin{align*}
	Q_{ijk}:=\langle\nabla^\bot_k\mathring{A}_{ij},\nu_1\rangle-\langle\nabla^\bot_k A_{ij}^- ,\nu_1\rangle-|H|^{-1}\mathring{h}_{ij}\nabla_k |H|.
	\end{align*}
\end{proposition}
We consider the function $f=-d_n+c_n |H|^2-|A|^2$. The assumption of the theorem is $f>0$ (and consequently $|H|>0$) everywhere on $\mathcal{M}_0$. As $\mathcal{M}_0$ is compact, there exist constants $\e_0,\e_1>0$ depending on $\mathcal{M}_0$, such that $f\ge\e_1|H|^2+\e_0$, on $\mathcal{M}_0$. By Theorem 2 in \cite{AnBa10}, $f\ge\e_1|H|^2+\e_0$, on $\mathcal{M}_t$, for every $t\in[0,T)$ and consequently $|H|>0$ is preserved as well. Recall
	\begin{align*}
	c_n=\frac{4}{3n}, \ \ \text{ if} \ \ n\ge 8 \ \ \ \text{ and} \ \ \ c_n=\frac{3(n+1)}{2n(n+2)}, \ \ \text{ if} \ \ n=5,6 \ \text{ or} \ 7.
	\end{align*}
We will require additional pinching for our estimates when $n=5,6$ or $7$. Since $|A|^2+\e_0\le(c_n-\e_1)|H|^2$, for every $t\in[0,T)$, without loss of generality, we may replace $c_n$ by $c_n-\e_1$ and assume throughout the proof that
	\begin{align*}
	c_n\le\frac{4}{3n}, \ \ \text{ if} \ \ n\ge 8 \ \ \ \text{ and} \ \ \ c_n\leq\frac{3(n+1)}{2n(n+2)}, \ \ \text{ if} \ \ n=5,6 \ \text{ or} \ 7.
	\end{align*}
The strictness of the latter inequality depends on initial data through $\e_1$. We still have $f\ge\e_0>0$ and $|H|>0$, for every $t$.
Let $\delta>0$ be a small constant to be determined later in the proof. By previous work, the evolution equation for $f$ is
	\begin{align}\label{evoloff}
	\begin{split}
	\Big(\partial_t-\Delta\Big)f&=2(|\nabla^\bot A|^2-c_n|\nabla^\bot H|^2)+2\Big(c_n\sum_{i,j}|\langle A_{ij},H\rangle|^2-\sum_{i,j,p,q}|\langle A_{ij},A_{pq}\rangle|^2-\sum_{i,j}|R_{ij}^\bot |^2\Big)\\
	&+2c_n\sum_{k,\alpha,\beta} \bar{R}_{k\alpha k\beta} H^\alpha H^\beta-P_{\alpha}.
	\end{split}
	\end{align}
The pinching condition implies both terms on the right hand side of the equation for $f$ are non negative at each point in space-time. The first step of the proof and the main effort is to analyse the evolution equation $\frac{|A^-|^2}{f}$. We will show this ratio satisfies a favourable evolution equation with a right hand side has a nonpositive term. Specifically, we will show that
	\begin{align}\label{initialclaim}
	\begin{split}
	\Big(\partial_t-\Delta\Big)\frac{|A^-|^2}{f}&\le2\Big\langle\nabla\frac{|A^-|^2}{f},\nabla\log f\Big\rangle-\delta\frac{|A^-|^2}{f^2}\Big(\partial_t-\Delta\Big)f+C'\frac{|A^-|^2}{f}+C''\frac{|A^-|}{\sqrt{f}}\\
	&+\frac{1}{f}\Big(2\sum_{i,j}|\bar{R}_{ij}(\nu_1)|^2+4\sum_{i,j,p}\langle\bar{R}_{ij}(\nu_1),\mathring{h}_{ip}A^-_{jp}-\mathring{h}_{jp}A^-_{ip}\rangle \Big),
	\end{split}
	\end{align}
for $C',C''$ constants, that depend on $n,K_1,K_2$ and $d_n$. Then, since at the limit the background space is Euclidean, the result will follow from the maximum principle. By what we have shown this far, the evolution equation of $\frac{|A^-|^2}{f}$ is
	\begin{align*}
	\Big(\partial_t&-\Delta\Big)\frac{|A^-|^2}{f}=\frac{1}{f}\Big(\partial_t-\Delta\Big)|A^-|^2-|A^-|^2\frac{1}{f^2}\Big(\partial_t-\Delta\Big)f+2\Big\langle\nabla\frac{|A^-|^2}{f},\nabla \log f\Big\rangle\\
	&=\frac{1}{f}\Big(2\sum_{i,j,p,q}|\langle A_{ij}^- ,A_{pq}^- \rangle|^2+2|\hat{R}^\bot|^2+2\sum_{i,j}|R_{ij}^\bot (\nu_1)|^2\Big)\\
	&+\frac{1}{f}\Big(-2|\nabla^\bot A^-|^2+4\sum_{i,j,k}Q_{ijk} \langle A_{ij}^- ,\nabla^\bot_k \nu_1\rangle+2|H|^{-2}\sum_{i,j,k,\alpha,\beta\ge 2}\bar{R}_{k\alpha k\beta} H^\alpha A^{\beta}_{ij}\langle A_{ij},H\rangle\Big)\\
	&+\frac{1}{f}\Big(2\sum_{i,j}|\bar{R}_{ij}(\nu_1)|^2+4\sum_{i,j,p}\langle\bar{R}_{ij}(\nu_1),\mathring{h}_{ip}A^-_{jp}-\mathring{h}_{jp}A^-_{ip}\rangle \Big)\\
	&-|A^-|^2\frac{1}{f^2}\Big(2(|\nabla^\bot A|^2-c_n|\nabla^\bot H|^2)\Big)\\
	&-|A^-|^2\frac{1}{f^2}\Big(2\Big(c_n\sum_{i,j}|\langle A_{ij},H\rangle|^2-\sum_{i,j,p,q}|\langle A_{ij},A_{pq}\rangle|^2-\sum_{i,j}|R_{ij}^\bot |^2\Big)\Big)\\
	&+2\Big\langle\nabla\frac{|A^-|^2}{f},\nabla\log f\Big\rangle\\
	&+\frac{1}{f}\Big(P_\alpha-2B''\Big)-|A^-|^2\frac{1}{f^2}\Big(2c_n\sum_{k,\alpha,\beta}\bar{R}_{k\alpha k\beta} H^\alpha H^\beta-P_\alpha\Big).
	\end{align*}
Rearranging these terms, we have
	\begin{align*}
	&\Big(\partial_t-\Delta\Big)\frac{|A^-|^2}{f}=\frac{1}{f}\Big(2\sum_{i,j,p,q}|\langle A_{ij}^- ,A_{pq}^- \rangle|^2+2|\hat{R}^\bot|^2+2\sum_{i,j}|R_{ij}^\bot (\nu_1)|^2\Big)\\
	&+\frac{1}{f}\Big(-2\frac{|A^-|^2}{f}\Big(c_n\sum_{i,j}|\langle A_{ij},H\rangle|^2-\sum_{i,j,p,q}|\langle A_{ij},A_{pq}\rangle|^2-\sum_{i,j}|R_{ij}^\bot |^2\Big)\Big)\\
	&+\frac{1}{f}\Big(2\sum_{i,j}|\bar{R}_{ij}(\nu_1)|^2+4\sum_{i,j,p}\langle\bar{R}_{ij}(\nu_1),\mathring{h}_{ip}A^-_{jp}-\mathring{h}_{jp}A^-_{ip}\rangle \Big)\\
	&+\frac{1}{f}\Big(4\sum_{i,j,p,q}\bar{R}_{ipjq}\big(\sum_{\alpha\ge 2} A^\alpha_{pq}A^\alpha_{ij}\big)-4\sum_{j,k,p}\bar{R}_{kjkp}\big(\sum_{i,\alpha\ge 2} A^\alpha_{pi}A^\alpha_{ij}\big)\Big)\\
	&+\frac{1}{f}\Big(2\sum_{k,\alpha,\beta\ge 2}\bar{R}_{k\alpha k\beta}\big(\sum_{i,j} A^\alpha_{ij}A_{ij}^\beta \big)-8\sum_{j,p,\alpha,\beta\ge 2}\bar{R}_{jp\alpha \beta}\big(\sum_iA^\alpha_{ip}A_{ij}^\beta \big)\Big)\\
	&+\frac{1}{f}\Big(2|H|^{-2}\sum_{i,j,k,\alpha,\beta\ge 2}\bar{R}_{k\alpha k\beta} H^\alpha A^{\beta}_{ij}\langle A_{ij},H\rangle+2\sum_{i,j,k,\beta\ge 2}\bar{\nabla}_k\bar{R}_{kij\beta}A_{ij}^\beta-2\sum_{i,j,k,\beta\ge 2}\bar{\nabla}_i\bar{R}_{jkk\beta}A_{ij}^\beta\Big)\\
	&+\frac{1}{f}\Big(4\sum_{i,j,k}Q_{ijk}\langle A_{ij}^- ,\nabla^\bot_k \nu_1\rangle-2|\nabla^\bot A^-|^2-2\frac{|A^-|^2}{f}(|\nabla^\bot A|^2-c_n|\nabla^\bot H|^2)\Big)\\
	&+\frac{1}{f}\Big(\frac{|A^-|^2}{f}\big(4\sum_{i,j,p,q}\bar{R}_{ipjq}\big(\sum_{\alpha} A^\alpha_{pq}A^\alpha_{ij}\big)-4\sum_{j,k,p}\bar{R}_{kjkp}\big(\sum_{i,\alpha} A^\alpha_{pi}A^\alpha_{ij}\big)\big)-2\sum_k \bar{R}_{k1k1}\sum_{i,j} (A^1_{ij})^2\Big)\\
	&+\frac{1}{f}\Big(\frac{|A^-|^2}{f}\big(2\sum_{k,\alpha,\beta}\bar{R}_{k\alpha k\beta}\big(\sum_{i,j} A^\alpha_{ij}A_{ij}^\beta \big)+4c_n\sum_{k,\alpha,\beta}\bar{R}_{k\alpha k\beta} H^\alpha H^\beta -8\sum_{j,p,\alpha,\beta}\bar{R}_{jp\alpha\beta}\big(\sum_iA^\alpha_{ip}A_{ij}^\beta \big)\big)\Big)\\
	&+\frac{1}{f}\Big(\frac{|A^-|^2}{f}\big(2\sum_{i,j,k,\beta}\bar{\nabla}_k\bar{R}_{kij\beta}A_{ij}^\beta -2\sum_{i,j,k,\beta}\bar{\nabla}_i\bar{R}_{jkk\beta}A_{ij}^\beta \big)\Big)\\
	&+2\Big\langle \nabla\frac{|A^-|^2}{f},\nabla\log f\Big\rangle.
	\end{align*}
Let us provide a brief explanation of the above evolution equation. The first two lines on the right hand side are the higher order terms and the terms in the third line are Euclidean terms.  The terms from the third line to the sixth line are lower order terms, that are orthogonal to the principal direction. The terms on the seventh and eleventh line are gradient terms and the terms from the eighth to the tenth line are lower order terms, both in the principal direction and orthogonal to the principal direction. \\
We begin by estimating the reaction terms. We will make use of two estimates. The first estimate is proven on page 372 in \cite{AnBa10}, Section 3. The second estimate is a matrix inequality, which is Lemma 3.3 in \cite{Li1992}.
\begin{lemma}\label{4.1}
	\begin{align}\label{eq4.5}
	\sum_{i, j}\big|\mathring{h}_{ij} A^-_{ij}\big|^2+\sum_{i,j}|R_{ij}^{\perp}(\nu_1)|^2 \leq 2|\mathring{h}|^2|{A^-}|^2+\sum_{i,j}|\bar{R}_{ij}(\nu_1)|^2+4|\bar{R}_{ij}(\nu_1)||\mathring{h}||A^-|,
	\end{align}
	\begin{align}\label{eq4.6}
	\sum_{i,j,p,q}|\langle A^-_{ij}, A^-_{pq}\rangle|^2+|\hat{R}^\perp|^2 \leq\frac{3}{2}|A^-|^4+\sum_{\alpha, \beta\ge 2}\Big(\sum_{i,j}|\bar{R}_{ij\alpha\beta}|^2+4|\bar{R}_{ij\alpha\beta}||A^-|^2\Big).
	\end{align}
\end{lemma}
\begin{proof} The arguments given in \cite{AnBa10} to prove inequality \eqref{eq4.5} are simple and short, so we will repeat them in our notation here. We will express inequality \eqref{eq4.6} so that it is an immediate consequence of Lemma 3.3 in \cite{Li1992}.
Fix any point $p \in \mathcal{M}$ and time $t \in[0, T)$. Let $e_1, \ldots, e_n$ be an orthonormal basis which identifies $T_p \mathcal{M} \cong \mathbb{R}^n$ at time $t$ and then choose $\nu_2, \ldots, \nu_{m}$ to be a basis of the orthogonal complement of principal normal $\nu_1$ in $N_p \mathcal{M}$ at time $t$. For each $\beta \in\{2, \ldots, m\}$, define a matrix $A_{\beta}=\left\langle A, \nu_{\beta}\right\rangle$ whose components are given by $(A_{\beta})_{ij}=A^\beta_{ij}$.

Then $A^-=\sum_{\beta\ge 2} A_{\beta} \nu_{\beta}$. We also have $\mathring{h}=\langle\mathring{A}, \nu_1\rangle$.
To prove \eqref{eq4.5}, let $\lambda_1, \ldots, \lambda_n$ denote the eigenvalues of $\mathring{h}$. Assume the orthonormal basis is an eigenbasis of $\mathring{h}$. Now
	\begin{align*}
	\sum_{i, j}| \mathring{h}_{i j} A^-_{i j}|^2=\sum_{\beta\ge 2} \sum_{i, j,p,q} \mathring{h}_{i j} \mathring{h}_{pq} A^\beta_{i j } A^\beta_{pq} =\sum_{\beta\ge 2}\big(\sum_{i, j} \mathring{h}_{i j} A^\beta_{i j }\big)^2 =\sum_{\beta\ge 2}\big(\sum_{i} \lambda_i A^\beta_{i i }\big)^2.
	\end{align*}
By Cauchy-Schwarz,
	\begin{align}\label{eq4.7}
	\sum_{i, j}|\mathring{h}_{i j} A^-_{i j}|^2 \leq \sum_{\beta\ge 2}\big(\sum_{i} \lambda_j^2\big)\big(\sum_{i} (A^\beta_{i i})^2\big)=|\mathring{h}|^2 \sum_{\beta\ge 2} \sum_{i}(A^\beta_{i i})^2.
	\end{align}
Now, using
	\begin{align}\label{eq2.40}
	\sum_{i,j}|R^\bot_{ij}(\nu_1)|^2=\sum_{i,j}|\bar{R}_{ij}(\nu_1)|^2+\sum_{i,j,k}|\mathring{h}_{ik}A^-_{jk}-\mathring{h}_{jk}A^-_{ik}|^2+2\sum_{i,j,p}\langle \bar{R}_{ij}(\nu_1),\mathring{h}_{ip}A^-_{jp}-\mathring{h}_{jp}A^-_{ip}\rangle,
	\end{align}
and \eqref{Berger} we have
	\begin{align*}
	\sum_{i,j}|R_{i j}^{\perp}(\nu_1)|^2&=\sum_{\beta\ge 2} \sum_{i, j,k}\big(\mathring{h}_{i k} A^\beta_{j k }-\mathring{h}_{j k} A^\beta_{ik}\big)^2+\sum_{i,j}|\bar{R}_{ij}(\nu_1)|^2+2\sum_{i,j,p}\langle \bar{R}_{ij}(\nu_1),\mathring{h}_{ip}A^-_{jp}-\mathring{h}_{jp}A^-_{ip}\rangle \\
	&=\sum_{\beta\ge 2} \sum_{i, j}\big(\lambda_i-\lambda_j\big)^2 (A^\beta_{i j})^2 +\sum_{i,j}|\bar{R}_{ij}(\nu_1)|^2+2\sum_{i,j,p}\langle \bar{R}_{ij}(\nu_1),\mathring{h}_{ip}A^-_{jp}-\mathring{h}_{jp}A^-_{ip}\rangle \\
	&=\sum_{\beta\ge 2} \sum_{i \neq j}\big(\lambda_i-\lambda_j\big)^2 (A^\beta_{i j })^2+\sum_{i,j}|\bar{R}_{ij}(\nu_1)|^2+2\sum_{i,j,p}\langle \bar{R}_{ij}(\nu_1),\mathring{h}_{ip}A^-_{jp}-\mathring{h}_{jp}A^-_{ip}\rangle.
	\end{align*}
Since $\left(\lambda_i-\lambda_j\right)^2 \leq 2\left(\lambda_i^2+\lambda_j^2\right) \leq 2|\mathring{h}|^2$, we have
	\begin{align}\label{eq4.8}
	\sum_{i,j}|R_{i j}^{\perp}(\nu_1)|^2 \leq 2|\mathring{h}|^2 \sum_{\beta\ge 2} \sum_{i \neq j} (A^\beta_{i j})^2+\sum_{i,j}|\bar{R}_{ij}(\nu_1)|^2+4|\bar{R}_{ij}(\nu_1)||\mathring{h}||A^-|.
	\end{align}
Summing \eqref{eq4.7} and \eqref{eq4.8}, we obtain
	\begin{align*}
	\sum_{i, j} |\mathring{h}_{i j} A^-_{i j}|^2+\sum_{i,j}|R_{i j}^{\perp}(\nu_1)|^2&\leq|\mathring{h}|^2 \sum_{\beta\ge 2} \sum_{i}(A^\beta_{i i})^2+2|\mathring{h}|^2 \sum_{\beta\ge 2} \sum_{i \neq j} (A^\beta_{i j})^2+\sum_{i,j}|\bar{R}_{ij}(\nu_1)|^2\\
	&+4|\bar{R}_{ij}(\nu_1)||\mathring{h}||A^-|\\
	&\leq 2|\mathring{h}|^2|A^-|^2+\sum_{i,j}|\bar{R}_{ij}(\nu_1)|^2+4|\bar{R}_{ij}(\nu_1)||\mathring{h}||A^-|,
	\end{align*}
which is \eqref{eq4.5}.
To establish \eqref{4.6}, for $\alpha, \beta \in\{2, \ldots, m\}$ define
	\begin{align*}
	S_{\alpha \beta}:=\operatorname{tr}\left(A_{\alpha} A_{\beta}\right)=\sum_{i, j,\alpha} A^\alpha_{i j } A^\beta_{i j} \quad \text{ and } \quad S_{\alpha}:=\left|A_{\alpha}\right|^2=\sum_{i, j,\alpha} A^\alpha_{i j} A^\alpha_{i j}
	\end{align*}
Let $S:=S_2+\cdots+S_m=|A^-|^2$. Now
	\begin{align*}
	\sum_{i,j,p,q}|\langle A^-_{i j}, A^-_{pq}\rangle|^2&=\sum_{i, j, p,q} \sum_{\alpha, \beta\ge 2}A^\alpha_{i j} A^\alpha_{pq} A^\beta_{i j } A^\beta_{pq} \\
	&=\sum_{\alpha, \beta\ge 2}\big(\sum_{i, j} A^\alpha_{i j} A^\beta_{i j}\big)(\sum_{p,q} A^\alpha_{pq } A^\beta_{pq}) \\
	&=\sum_{\alpha, \beta\ge 2} S_{\alpha \beta}^2.
	\end{align*}
In addition, recalling \eqref{hatR}, we may write
	\begin{align*}
	|\hat{R}^\perp|^2=\sum_{\alpha, \beta\ge 2}\Big(|A_{\alpha} A_{\beta}-A_{\beta} A_{\alpha}|^2+\sum_{i,j}|\bar{R}_{ij\alpha\beta}|^2+2\sum_{i,j,p}\langle \bar{R}_{ij\alpha\beta},A_{ip}^{\alpha}A_{jp}^{\beta}-A_{jp}^{\alpha}A_{ip}^{\beta}\rangle\Big)
	\end{align*}
where $\left(A_{\alpha} A_{\beta}\right)_{i j}=\left(A_{\alpha}\right)_{i k}\left(A_{\beta}\right)_{k j}=\left(A_{\alpha}\right)_{i k}\left(A_{\beta}\right)_{j k}$ denotes standard matrix multiplication and $|\cdot|$ is the usual square norm of the matrix. We see that inequality \eqref{eq4.6} is equivalent to
	\begin{align}\label{eq4.9}
	\sum_{\alpha, \beta\ge 2}|A_{\alpha} A_{\beta}-A_{\beta} A_{\alpha}|^2+\sum_{\alpha, \beta\ge 2} S_{\alpha \beta}^2 \leq \frac{3}{2} S^2.
	\end{align}
Therefore, we have
	\begin{align*}
	\sum_{i,j,p,q}|\langle A^-_{ij},A^-_{pq}\rangle|^2+|\hat{R}^\bot|^2&\le\frac{3}{2}|A^-|^4+\sum_{i,j,\alpha, \beta\ge 2}\Big(|\bar{R}_{ij\alpha\beta}|^2+2\sum_{p}\langle \bar{R}_{ij\alpha\beta},A_{ip}^{\alpha}A_{jp}^{\beta}-A_{jp}^{\alpha}A_{ip}^{\beta}\rangle\Big)\\
	&\le\frac{3}{2}|A^-|^4+\sum_{\alpha, \beta\ge 2}\Big(\sum_{i,j}|\bar{R}_{ij\alpha\beta}|^2+4|\bar{R}_{ij\alpha\beta}||A^-|^2\Big).
	\end{align*}
Now if $m=2$, inequality \eqref{eq4.6} is trivial since $|\hat{R}^{\perp}|^2=0$ and $\sum_{i,j,p,q}|\langle A^-_{i j}, A^-_{pq}\rangle|^2=|A^-|^4$. Otherwise, if $m \geq 3$, inequality \eqref{eq4.9} follows Lemma 3.3 in \cite{Li1992}. This completes the proof.
\end{proof}
As an immediate consequence of the previous lemma, we have the following estimate for the reaction terms coming from the evolution of $|A^-|^2$.
\begin{lemma} [Upper bound for the reaction terms of $(\partial_t-\Delta)|A^-|^2$]\label{4.2}
	\begin{align}\label{eq4.10}
	\sum_{i,j,p,q}|\langle A^-_{i j}, A^-_{pq}\rangle|^2+|\hat{R}^{\perp}|^2+\sum_{i,j,}|R_{i j}^{\perp}(\nu_1)|^2&\leq\frac{3}{2}|A^-|^4+\sum_{\alpha, \beta\ge 2}\Big(\sum_{i,j}|\bar{R}_{ij\alpha\beta}|^2+4|\bar{R}_{ij\alpha\beta}||A^-|^2\Big)\nonumber\\
	&+2|\mathring{h}|^2|A^-|^2+\sum_{i,j}|\bar{R}_{ij}(\nu_1)|^2+4|\bar{R}_{ij}(\nu_1)||\mathring{h}||A^-|.
	\end{align}
\end{lemma}
\begin{proof} The proof follows from Lemma \ref{4.1}.
\end{proof}
Next we express the reaction term in the evolution of $f$ in terms of $A^-, \mathring{h}$, and $|H|$. In view of the definition of $f$, observe that
	\begin{align}\label{eq4.11}
	\frac{n c_n-1}{n}|H|^2=|A^-|^2+|\mathring{h}|^2+f+d_n.
	\end{align}
In the following lemma, we get a lower bound for the reaction terms in the evolution of $f$.
\begin{lemma} [Lower bound for the reaction terms of $\left(\partial_t-\Delta\right) f$]\label{lemma4.3} \ \newline
If $\frac{1}{n}<c_n \leq \frac{4}{3 n}$, then
	\begin{align}\label{eq4.12}
	\frac{|A^-|^2}{f}&\Big(c_n\sum_{i,j}\left|\left\langle A_{i j}, H\right\rangle\right|^2-\sum_{i,j,p,q}\left|\left\langle A_{i j}, A_{pq}\right\rangle\right|^2-\sum_{i,j}|R_{i j}^{\perp}|^2\Big)\geq \frac{2}{n c_n-1}|A^-|^4+\frac{n c_n}{n c_n-1}|\mathring{h}|^2|A^-|^2\nonumber\\
	&-\frac{|A^-|^2}{f}\Big(\sum_{\alpha,\beta\ge 2}\Big(\sum_{i,j}|\bar{R}_{ij\alpha\beta}|^2+4|\bar{R}_{ij\alpha\beta}||A^-|^2\Big)+2\sum_{i,j}|\bar{R}_{ij}(\nu_1)|^2+8|\bar{R}_{ij}(\nu_1)||\mathring{h}||A^-|\Big).
	\end{align}
\end{lemma}
\begin{proof} We do a computation that is similar to a computation in \cite{AnBa10}, except we do not throw away the pinching term $f$. By the following equations
	\begin{align*}
	|h|^2=|\mathring{h}|^2+\frac{1}{n}|H|^2,
	\end{align*}
	\begin{align*}
	\sum_{i,j}|\langle A_{ij},H\rangle|^2=|H|^2|h|^2,
	\end{align*}
	\begin{align*}
	\sum_{i,j,p,q}|\langle A_{ij},A_{pq}\rangle|^2=|h|^4+2\sum_{i,j}\big| \mathring{h}_{ij}A^-_{ij}\big|^2+\sum_{i,j,p,q}|\langle A^-_{ij},A^-_{pq}\rangle|^2
	\end{align*}
and
	\begin{align}
	2\sum_{i,j}|R^\bot_{ij}|^2-2\sum_{i,j}|R^\bot_{ij} (\nu_1)|^2=|\hat{R}^\bot|^2+2\sum_{i,j}|R^\bot_{ij}(\nu_1)|^2=|R^\bot|^2,
	\end{align}
we have
	\begin{align*}
	c_n\sum_{i,j}\left|\left\langle A_{i j}, H\right\rangle\right|^2&-\sum_{i,j,p,q}\left|\left\langle A_{i j}, A_{pq}\right\rangle\right|^2-\sum_{i,j}|R_{i j}^{\perp}|^2=\frac{1}{n} c_n|H|^4+c_n|\mathring{h}|^2|H|^2-|\mathring{h}|^4\\
	&-\frac{2}{n}|\mathring{h}|^2|H|^2-\frac{1}{n^2}|H|^4-2\sum_{i,j}|\mathring{h}_{i j} A^-_{i j}|^2-\sum_{i,j,p,q}|\langle A^-_{i j}, A^-_{pq}\rangle|^2-|\hat{R}^{\perp}|^2\\
	&-2\sum_{i,j}|R_{i j}^{\perp}(\nu_1)|^2\\
	&=\frac{1}{n}\left(c_n-\frac{1}{n}\right)|H|^4+\left(c_n-\frac{1}{n}\right)|\mathring{h}|^2|H|^2-\frac{1}{n}|\mathring{h}|^2|H|^2-|\mathring{h}|^4\\
	&-2\sum_{i,j}|\mathring{h}_{i j} A^-_{i j}|^2-2\sum_{i,j}|R_{i j}^{\perp}(\nu_1)|^2-\sum_{i,j,p,q}|\langle A^-_{i j}, A^-_{pq}\rangle|^2-|\hat{R}^{\perp}|^2.
	\end{align*}
Use \eqref{eq4.11} and cancel terms to get
	\begin{align*}
	c_n&\sum_{i,j}\left|\left\langle A_{i j}, H\right\rangle\right|^2-\sum_{i,j,p,q}\left|\left\langle A_{i j}, A_{pq}\right\rangle\right|^2-\sum_{i,j}|R_{i j}^{\perp}|^2\\
	&=\frac{1}{n}\big(|A^-|^2+|\mathring{h}|^2+f+d_n \big)|H|^2+|\mathring{h}|^2\big( |A^-|^2+|\mathring{h}|^2+f+d_n\big)\\
	&-\frac{1}{n}|\mathring{h}|^2|H|^2-|\mathring{h}|^4-2\sum_{i,j}|\mathring{h}_{i j} A^-_{i j}|^2-2\sum_{i,j}|R_{i j}^{\perp}\left(\nu_1\right)|^2-\sum_{i,j,p,q}|\langle A^-_{i j}, A^-_{pq}\rangle|^2-|\hat{R}^{\perp}|^2\\
	&=\frac{1}{n}\left(f+|A^-|^2+d_n\right)|H|^2+\left(f+|A^-|^2+d_n\right) |\mathring{h}|^2\\
	&-2\sum_{i,j}|\mathring{h}_{i j} A^-_{i j}|^2-2\sum_{i,j}|R_{i j}^{\perp}(\nu_1)|^2-\sum_{i,j,p,q}|\langle A^-_{i j}, A^-_{pq}\rangle|^2-|\hat{R}^{\perp}|^2.
	\end{align*}
Using \eqref{eq4.11} once more for the remaining factor of $|H|^2$ gives
	\begin{align*}
	&c_n\sum_{i,j}\left|\left\langle A_{i j}, H\right\rangle\right|^2-\sum_{i,j,p,q}\left|\left\langle A_{i j}, A_{pq}\right\rangle\right|^2-\sum_{i,j}|R_{i j}^{\perp}|^2 \\
	&= \frac{1}{n}\left(f+|A^-|^2+d_n\right)\left(c_n-\frac{1}{n}\right)^{-1}(f+|A^-|^2+|\mathring{h}|^2+d_n)+\left(f+|A^-|^2+d_n\right)|\mathring{h}|^2 \\
	&-2\sum_{i,j}|\mathring{h}_{i j} A^-_{i j}|^2-2\sum_{i,j}|R_{i j}^{\perp}(\nu_1)|^2-\sum_{i,j,p,q}|\langle A^-_{i j}, A^-_{pq}\rangle|^2-|\hat{R}^{\perp}|^2 \\
	&= \frac{1}{n c_n-1} f(f+2|A^-|^2+|\mathring{h}|^2+2d_n)+f|\mathring{h}|^2+\frac{1}{n c_n-1}| A^-|^4+\frac{n c_n}{n c_n-1}|A^-|^2|\mathring{h}|^2 \\
	&+\frac{nc_n}{nc_n-1}d_n |\mathring{h}|^2+\frac{1}{nc_n-1}d_n|A^-|^2-2\sum_{i,j}|\mathring{h}_{i j} A^-_{i j}|^2-2\sum_{i,j}|R_{i j}^{\perp}(\nu_1)|^2-\sum_{i,j,p,q}|\langle A^-_{i j}, A^-_{pq}\rangle|^2\\
	&-|\hat{R}^{\perp}|^2.
	\end{align*}
Now by the two estimates in Lemma \ref{4.1}
	\begin{align*}
	2\sum_{i,j}|\mathring{h}_{ij}A^-_{ij}|^2&+2\sum_{i,j}|R^\bot_{ij}(\nu_1)|^2+\sum_{i,j,p,q}|\langle A^-_{ij},A^-_{pq}\rangle|^2+|\hat{R}^\bot|^2\le4|\mathring{h}|^2|A^-|^2+2\sum_{i,j}|\bar{R}_{ij}(\nu_1)|^2\\
	&+8|\bar{R}_{ij}(\nu_1)||\mathring{h}||A^-|+\frac{3}{2}|A^-|^4+\sum_{\alpha, \beta\ge 2}\Big(\sum_{i,j}|\bar{R}_{ij\alpha\beta}|^2+4|\bar{R}_{ij\alpha\beta}||A^-|^2\Big).
	\end{align*}
Therefore,
	\begin{align*}
	\frac{1}{n c_n-1}&|A^-|^4+\frac{n c_n}{n c_n-1}|A^-|^2|\mathring{h}|^2-2\sum_{i,j}|\mathring{h}_{i j} A^-_{i j}|^2-2\sum_{i,j}|R_{i j}^{\perp}(\nu_1)|^2\\
	&-\sum_{i,j,p,q}|\langle A^-_{i j}, A^-_{pq}\rangle|^2-|\hat{R}^{\perp}|^2\\
	&\geq\left(\frac{1}{n c_n-1}-\frac{3}{2}\right)|A^-|^4+\left(\frac{n c_n}{n c_n-1}-4\right)|\mathring{h}|^2|A^-|^2-2\sum_{i,j}|\bar{R}_{ij}(\nu_1)|^2\\
	&-8|\bar{R}_{ij}(\nu_1)||\mathring{h}||A^-|-\sum_{\alpha, \beta\ge 2}\Big(\sum_{i,j}|\bar{R}_{ij\alpha\beta}|^2+4|\bar{R}_{ij\alpha\beta}||A^-|^2\Big).
	\end{align*}
Since $c_n \leq \frac{4}{3 n}$, we have
	\begin{align*}
	\frac{1}{n c_n-1}-\frac{3}{2} \geq \frac{3}{2}, \quad \frac{n c_n}{n c_n-1}-4 \geq 0.
	\end{align*}
Consequently, we have
	\begin{align}\label{eq4.13}
	c_n\sum_{i,j}\left|\left\langle A_{i j}, H\right\rangle\right|^2-&\sum_{i,j,p,q}\left|\left\langle A_{i j}, A_{pq}\right\rangle\right|^2-\sum_{i,j}|R_{i j}^{\perp}|^2 \geq \frac{2}{n c_n-1} f|A^-|^2+\frac{n c_n}{n c_n-1} f|\mathring{h}|^2\nonumber\\
	&+\frac{1}{n c_n-1} f^2 +\frac{1}{nc_n-1}d_n\big(|A^-|^2+2f\big)+\frac{nc_n}{nc_n-1}|\mathring{h}|^2d_n-2\sum_{i,j}|\bar{R}_{ij}(\nu_1)|^2\nonumber\\
	&-8|\bar{R}_{ij}(\nu_1)||\mathring{h}||A^-|-\sum_{\alpha, \beta\ge 2}\Big(\sum_{i,j}|\bar{R}_{ij\alpha\beta}|^2+4|\bar{R}_{ij\alpha\beta}||A^-|^2\Big)\nonumber\\
	&\geq \frac{2}{n c_n-1} f|A^-|^2+\frac{n c_n}{n c_n-1} f|\mathring{h}|^2-8|\bar{R}_{ij}(\nu_1)||\mathring{h}||A^-|\nonumber\\
	&-\sum_{\alpha, \beta\ge 2}\Big(\sum_{i,j}|\bar{R}_{ij\alpha\beta}|^2+4|\bar{R}_{ij\alpha\beta}||A^-|^2\Big)-2\sum_{i,j}|\bar{R}_{ij}(\nu_1)|^2.
	\end{align}
Multiplying both sides by $\frac{|A^-|^2}{f}$ completes the proof of the lemma.
\end{proof}
Putting Lemmas \ref{4.2} and \ref{lemma4.3} together, we have
\begin{lemma}[Reaction term estimate]\label{lemma4.4} \ \newline
If $0<\delta \leq \frac{1}{2}$ and $\frac{1}{n}<c_n \leq \frac{4}{3 n}$, then
	\begin{align}\label{eq4.14}
	\sum_{i,j,p,q}|\langle A^-_{i j}, A^-_{pq}\rangle|^2+|\hat{R}^{\perp}|^2&+\sum_{i,j}|R_{i j}^{\perp}(\nu_1)|^2\leq(1-\delta) \frac{|A^-|^2}{f}\Big(c_n\sum_{i,j}\left|\left\langle A_{i j}, H\right\rangle\right|^2\nonumber\\
	&-\sum_{i,j,p,q}\left|\left\langle A_{i j}, A_{pq}\right\rangle\right|^2-\sum_{i,j}|R_{i j}^{\perp}|^2\Big)\nonumber\\
	&+(1-\delta)\frac{|A^-|^2}{f}\Big(\sum_{\alpha,\beta\ge 2}\Big(\sum_{i,j}|\bar{R}_{ij\alpha\beta}|^2+4|\bar{R}_{ij\alpha\beta}||A^-|^2\Big)\nonumber\\
	&+2\sum_{i,j}|\bar{R}_{ij}(\nu_1)|^2+8|\bar{R}_{ij}(\nu_1)||\mathring{h}||A^-|\Big).
	\end{align}
\end{lemma}
\begin{proof}
In view of \eqref{eq4.10} and \eqref{eq4.12}, we have
	\begin{align*}
	&\sum_{i,j,p,q}|\langle A^-_{i j}, A^-_{pq}\rangle|^2+|\hat{R}^{\perp}|^2+\sum_{i,j}|R_{i j}^{\perp}(\nu_1)|^2-(1-\delta) \frac{|A^-|^2}{f}\Big(c_n\sum_{i,j}\left|\left\langle A_{i j}, H\right\rangle\right|^2\\
	&-\sum_{i,j,p,q}\left|\left\langle A_{i j}, A_{pq}\right\rangle\right|^2-\sum_{i,j}|R_{i j}^{\perp}|^2\Big) \\
	&\leq \frac{3}{2}|A^-|^4+2|\mathring{h}|^2|A^-|^2-\frac{2(1-\delta)}{n c_n-1}|A^-|^4-\frac{n c_n(1-\delta)}{n c_n-1}|\mathring{h}|^2|A^-|^2 \\
	&+(1-\delta)\frac{|A^-|^2}{f}\Big(\sum_{\alpha,\beta\ge 2}\Big(\sum_{i,j}|\bar{R}_{ij\alpha\beta}|^2+4|\bar{R}_{ij\alpha\beta}||A^-|^2\Big)\\
	&+2\sum_{i,j}|\bar{R}_{ij}(\nu_1)|^2+8|\bar{R}_{ij}(\nu_1)||\mathring{h}||A^-|\Big)\\
	&=\left(\frac{3}{2}-\frac{2(1-\delta)}{n c_n-1}\right)|A^-|^4+\left(2-\frac{n c_n(1-\delta)}{n c_n-1}\right)|\mathring{h}|^2|A^-|^2\\
	&+(1-\delta)\frac{|A^-|^2}{f}\Big(\sum_{\alpha,\beta\ge2}\Big(\sum_{i,j}|\bar{R}_{ij\alpha\beta}|^2+4|\bar{R}_{ij\alpha\beta}||A^-|^2\Big)+2\sum_{i,j}|\bar{R}_{ij}(\nu_1)|^2+8|\bar{R}_{ij}(\nu_1)||\mathring{h}||A^-|\Big).
	\end{align*}
If $c_n \leq \frac{4}{3 n}$, then
	\begin{align*}
	\frac{1}{n c_n-1} \geq 3 \quad \text{ and } \quad \frac{n c_n}{n c_n-1} \geq 4.
	\end{align*}
Therefore, if $\delta\le\frac{1}{2}$
	\begin{align*}
	\frac{3}{2}-\frac{2(1-\delta)}{n c_n-1}\le \frac{3}{2}-6(1-\delta)\le 0,
	\end{align*}
	\begin{align*}
	2-\frac{n c_n(1-\delta)}{n c_n-1}\le 2-4(1-\delta)\le0,
	\end{align*}
which gives \eqref{eq4.14}.
\end{proof}
We are following the arguments of Naff \cite{Naff}, we turn our attention to the gradient terms. For this, we will use \eqref{2.22naff}. Recalling that $A^-_{j k}$ is traceless, it is straightforward to verify that
	\begin{align}\label{eq4.15}
	\sum_{i,j,k}|\nabla_i h_{j k}+\langle\nabla_i^{\perp} A^-_{j k}, \nu_1\rangle|^2 =\sum_{i,j,k}|\nabla_i \mathring{h}_{j k}+\langle\nabla_i^{\perp} A^-_{j k}, \nu_1\rangle|^2+\frac{1}{n}|\nabla| H \|^2
	\end{align}
	\begin{align}\label{eq4.16}
	\sum_{i,j,k}|\hat{\nabla}_i^{\perp} A^-_{j k}+h_{j k} \nabla_i^{\perp} \nu_1|^2 =\sum_{i,j,k}|\hat{\nabla}_i^{\perp} A^-_{j k}+\mathring{h}_{j k} \nabla_i^{\perp} \nu_1|^2+\frac{1}{n}|H|^2|\nabla^{\perp} \nu_1|^2
	\end{align}
Observe that the first term in \eqref{eq4.15} is just
	\begin{align}\label{eq4.17}
	\sum_{i,j,k}|\langle\nabla_i^{\perp} \mathring{A}_{j k}, \nu_1\rangle|^2=\sum_{i,j,k}|\nabla_i \mathring{h}_{j k}+\langle\nabla_i^{\perp} A^-_{j k}, \nu_1\rangle|^2,
	\end{align}
which will be useful later on. 
Now as observed in \cite{Hu84}, using Lemma \ref{katoinequality}, it follows from the Codazzi identity for the second fundamental form that the tensor
	\begin{align*}
	E_{i j k}=& \frac{1}{n+2}\left(\nabla_i ^\perp  H g_{j k}+\nabla_j ^\perp   H g_{i k}+\nabla_k ^\perp   Hg_{i j}\right)\\
&-\frac{2}{(n+2)(n-1)} w_i g_{j k}+\frac{n}{(n+2)(n-1)}\left(w_j g_{i k}+w_k g_{i j}\right)
	\end{align*}
is an irreducible component of $\nabla_i^{\perp} A_{j k}$ consisting of its various traces. In other words, $\langle E_{i j k}, \nabla_i^{\perp} A_{j k}\rangle=|E|^2$. This allows one to get an improved estimate over the trivial one. Namely,
	\begin{align*}
	|E|^2=\frac{3}{n+2}|\nabla^{\perp} H|^2 \leq|\nabla^{\perp} A|^2.
	\end{align*}
The projection of the Codazzi identity onto $\nu_1$ and its orthogonal complement implies the tensors $\nabla_i h_{j k}+\langle\nabla_i^{\perp} A^-_{j k}, \nu_1\rangle$ and $\hat{\nabla}_i^{\perp} A^-_{j k}+h_{j k} \nabla_i^{\perp} \nu_1$ are symmetric in $i, j, k$. Recalling \eqref{2.25naff} and \eqref{2.26naff}, it follows that an irreducible component of each tensor is given by
	\begin{align*}
	& E_{i j k}^{(1)}:=\frac{1}{n+2}(g_{i j} \nabla_k|H|+g_{j k} \nabla_i|H|+g_{k i} \nabla_j|H|), \\
	& E_{i j k}^{(\perp)}:=\frac{1}{n+2}(g_{i j}|H| \nabla_k^{\perp} \nu_1+g_{j k}|H| \nabla_i^{\perp} \nu_1+g_{k i}|H| \nabla_j^{\perp} \nu_1) .
	\end{align*}
You can readily confirm that $E_{i j k}^{(1)}(\nabla_i h_{j k}+\langle\nabla_i^{\perp} A^-_{j k}, \nu_1\rangle)=|E^{(1)}|^2$ and $\sum_{i,j,k}\langle E_{i j k}^{(\perp)}, \hat{\nabla}_i^{\perp} A^-_{j k}+h_{j k} \nabla_i^{\perp} \nu_1\rangle=|E^{(\perp)}|^2$. As in Lemma \ref{katoinequality}, we obtain that
	\begin{align*}
\frac{3}{n+2}|\nabla| H||^2 & \leq\sum_{i,j,k}|\nabla_i h_{j k}+\langle\nabla_i^{\perp} A^-_{j k}, \nu_1\rangle|^2,
	\end{align*}
	\begin{align}\label{4.20naff}
\frac{3}{n+2}|H|^2|\nabla^\bot \nu_1|^2\le\sum_{i,j,k}|\hat{\nabla}^\bot_i A^-_{jk}+h_{jk}\nabla^\bot_i \nu_1|^2.
	\end{align}

            	 	 	 	                 	 	 	       From Theorem \ref{thm_pinching} and \eqref{Berger}, we have that
	\begin{align*}
	 	&\frac{|A^-|^2}{f^2}\Big(4\sum_{i,j,p,q}\bar{R}_{ipjq}\big(\sum_{\alpha} A^\alpha_{pq}A^\alpha_{ij}\big)-4\sum_{j,k,p}\bar{R}_{kjkp}\big(\sum_{i,\alpha} A^\alpha_{pi}A^\alpha_{ij}\big)+2\sum_{k,\alpha,\beta}\bar{R}_{k\alpha k\beta}\big(\sum_{i,j} A^\alpha_{ij}A_{ij}^\beta \big)\Big)\\
	&+\frac{|A^-|^2}{f^2}\Big(-8\sum_{j,p,\alpha,\beta}\bar{R}_{jp\alpha\beta}\big(\sum_i A^\alpha_{ip}A_{ij}^\beta \big)+2\sum_{i,j,k,\beta}\bar{\nabla}_k\bar{R}_{kij\beta}A_{ij}^\beta -2\sum_{i,j,k,\beta}\bar{\nabla}_i\bar{R}_{jkk\beta}A_{ij}^\beta \Big)\\
	&<C_1\frac{|A^-|^2}{f^2}(|A|+|A|^2)\\
	&\le C_2\frac{|A^-|^2}{f},
	\end{align*}
where we used the fact that the quantities in the parenthesis divided by $f$ are bounded and $C_1$ and $C_2$ are constants, which depend on $n,K_1,K_2$ and $d_n$.
  	 	 	 	 	 	 	       Also, from \eqref{Berger}, we have
	\begin{align*}
	 	&\frac{1}{f}\Big(4\sum_{i,j,p,q}\bar{R}_{ipjq}\big(\sum_{\alpha\ge 2} A^\alpha_{pq}A^\alpha_{ij}\big)-4\sum_{j,k,p}\bar{R}_{kjkp}\big(\sum_{i,\alpha\ge 2} A^\alpha_{pi}A^\alpha_{ij}\big)\Big)\\
	&+\frac{1}{f}\Big(2\sum_{k,\alpha,\beta\ge 2}\bar{R}_{k\alpha k\beta}\big(\sum_{i,j} A^\alpha_{ij}A_{ij}^\beta \big)-8\sum_{j,p,\alpha,\beta\ge 2}\bar{R}_{jp\alpha \beta}\big(\sum_iA^\alpha_{ip}A_{ij}^\beta \big)\Big)\\
	&+\frac{1}{f}\Big(2|H|^{-2}\sum_{i,j,k,\alpha,\beta\ge 2}\bar{R}_{k\alpha k\beta} H^\alpha A^{\beta}_{ij}\langle A_{ij},H\rangle\Big)\\
	&+\frac{1}{f}\Big(2\sum_{i,j,k,\beta\ge 2}\bar{\nabla}_k\bar{R}_{kij\beta}A_{ij}^\beta-2\sum_{i,j,k,\beta\ge 2}\bar{\nabla}_i\bar{R}_{jkk\beta}A_{ij}^\beta\Big)\\
	&\le C_3\frac{1}{f}\Big(|A^-|^2 +|A^-||A|+|A^-||h|+|A^-|\Big)\\
	&\le C_4\frac{|A^-|^2}{f}+ C''\frac{|A^-|}{\sqrt{f}},
	\end{align*}
where we used the fact that the quantities in the parenthesis divided by $\sqrt{f}$ are bounded and $C_3,C_4$ and $C''$ are constants, which depend on $n,K_1,K_2$ and $d_n$. By previous calculations we have upper bounds for most of the terms. We will show that he rest of the gradient terms satisfy the following:
	\begin{align*}
	4 \sum_{i,j,k}Q_{i j k}\left\langle A^-_{ij}, \nabla_k^{\perp} \nu_1\right\rangle \leq 2|\nabla^{\perp} {A^-}|^2+2(1-\delta) \frac{|{A^-}|^2}{f}\left(|\nabla^{\perp} A|^2-c_n|\nabla^{\perp} H|^2\right).
	\end{align*}
\begin{lemma}[{Lower bound for Bochner term of $\left(\partial_t-\Delta\right)|A^-|^2$}]\label{4.6} \ \newline
\begin{enumerate}
\item If $\frac{1}{n}<c_n \leq \frac{4}{3 n}$, then
	\begin{align*}
	2|\hat{\nabla}^{\perp} A^-|^2&\geq\frac{4 n-10}{n+2}|\mathring{h}|^2|\nabla^{\perp} \nu_1|^2+\frac{6(n-1)}{n+2}(|A^-|^2+f+d_n)|\nabla^{\perp} \nu_1|^2.
	\end{align*}
\item If $\frac{1}{n}<c_n \leq \frac{3(n+1)}{2 n(n+2)}$, then
	\begin{align*}
	2|\hat{\nabla}^{\perp} A^-|^2 \geq 2|\mathring{h}|^2|\nabla^{\perp} \nu_1|^2+4(|A^-|^2+f+d_n)|\nabla^{\perp} \nu_1|^2.
	\end{align*}
\end{enumerate}
\end{lemma}
\begin{proof}
We begin by applying Young's inequality
	\begin{align*}
	\sum_{i,j,k}|\hat{\nabla}_i^{\perp} A^-_{j k}+\mathring{h}_{j k} \nabla_i^{\perp} \nu_1|^2&=|\hat{\nabla}^{\perp} A^-|^2+2\sum_{i,j,k}\langle\hat{\nabla}_i^{\perp} A^-_{j k}, \mathring{h}_{j k} \nabla_i^{\perp} \nu_1\rangle+|\mathring{h}|^2|\nabla^{\perp} \nu_1|^2 \\
	&\leq 2|\hat{\nabla}^{\perp} A^-|^2+2|\mathring{h}|^2|\nabla^{\perp} \nu_1|^2.
	\end{align*}
Multiplying both sides of \eqref{eq4.11} by $\frac{2(n-1)}{(n+2)(n c_n-1)}$ gives
	\begin{align*}
	\frac{2(n-1)}{n(n+2)}|H|^2=\frac{2(n-1)}{(n+2)\left(n c_n-1\right)}\left(f+|A^-|^2+|\mathring{h}|^2+d_n\right).
	\end{align*}
Since
	\begin{align}\label{4.22}
	\frac{2(n-1)}{n(n+2)}|H|^2|\nabla^\bot \nu_1|^2\le\sum_{i,j,k}|\hat{\nabla}^\bot_i A^-_{jk}+\mathring{h}_{jk}\nabla^\bot_i \nu_1|^2,
	\end{align}
our observations give us that
	\begin{align*}
	\frac{2(n-1)}{(n+2)\left(n c_n-1\right)}\left(f+|A^-|^2+|\mathring{h}|^2+d_n\right)|\nabla^{\perp} \nu_1|^2 \leq 2|\hat{\nabla}^{\perp} A^-|^2+2|\mathring{h}|^2|\nabla^{\perp} \nu_1|^2.
	\end{align*}
Subtracting the $|\mathring{h}|^2|\nabla^{\perp} \nu_1|^2$ term on the right-hand side gives
	\begin{align}\label{eqn_nablaA-}
	\frac{2(n-1)}{(n+2)\left(n c_n-1\right)}\left(f+|A^-|^2+d_n\right)|\nabla^{\perp} \nu_1|^2&+\left(\frac{2(n-1)}{(n+2)\left(n c_n-1\right)}-2\right)|\mathring{h}|^2|\nabla^{\perp} \nu_1|^2 \nonumber\\
	&\leq 2|\hat{\nabla}^{\perp}A^-|^2.
	\end{align}
If $c_n \leq \frac{4}{3 n}$, then $n c_n-1 \leq \frac{1}{3}$ and
	\begin{align*}
	\frac{2(n-1)}{(n+2)\left(n c_n-1\right)} \geq \frac{6(n-1)}{n+2}.
	\end{align*}
Plugging this into \eqref{eqn_nablaA-}, gives the first estimate of the lemma.
If $c_n \leq \frac{3(n+1)}{2 n(n+2)}$, then $n c_n-1 \leq \frac{n-1}{2(n+2)}$ and
	\begin{align*}
	\frac{2(n-1)}{(n+2)\left(n c_n-1\right)} \geq 4.
	\end{align*}
Plugging this into \eqref{eqn_nablaA-}, establishes the second estimate in the lemma.
\end{proof}
\begin{lemma} [{Lower bound for Bochner term of $\left(\partial_t-\Delta\right) f$}]\label{4.7} \ \newline
\begin{enumerate}
\item If $\frac{1}{n}<c_n \leq \frac{4}{3 n}$, then
	\begin{align*}
	2 \frac{|A^-|^2}{f}\left(|\nabla^{\perp} A|^2-c_n|\nabla^{\perp} H|^2\right) \geq \frac{5 n-8}{3(n-1)} \frac{|A^-|^2}{f}|\langle\nabla^{\perp} \mathring{A}, \nu_1\rangle|^2+\frac{10 n-16}{n+2}|A^-|^2|\nabla^{\perp} \nu_1|^2.
	\end{align*}
\item If $\frac{1}{n}<c_n \leq \frac{3(n+1)}{2 n(n+2)}$, then
	\begin{align*}
	2 \frac{|A^-|^2}{f}\left(|\nabla^{\perp} A|^2-c_n|\nabla^{\perp} H|^2\right) \geq \frac{3}{2} \frac{|A^-|^2}{f}|\langle\nabla^{\perp} \mathring{A}, \nu_1\rangle|^2+6|A^-|^2|\nabla^{\perp} \nu_1|^2.
	\end{align*}
\end{enumerate}
\end{lemma}
\begin{proof}
Using
	\begin{align*}
	|\nabla^\bot A|^2=\sum_{i,j,k}|\hat{\nabla}^\bot_i A^-_{jk} + h_{jk}\nabla^\bot_i \nu_1|^2+\sum_{i,j,k}|\langle\nabla^\bot_i A^-_{jk},\nu_1\rangle+\nabla_i h_{jk}|^2
	\end{align*}
and
	\begin{align}\label{2.23}
	|\nabla^\bot H|^2=|H|^2|\nabla^\bot \nu_1|^2+|\nabla |H||^2,
	\end{align}
we have
	\begin{align*}
	|\nabla^{\perp} A|^2-c_n|\nabla^{\perp} H|^2&=\sum_{i,j,k}|\langle\nabla_i^{\perp} A^-_{j k}, \nu_1\rangle+\nabla_i h_{j k}|^2-c_n|\nabla| H \|^2 \\
	&+\sum_{i,j,k}|\hat{\nabla}_i^{\perp} A^-_{j k}+h_{j k} \nabla_i^{\perp} \nu_1|^2-c_n|H|^2|\nabla^{\perp} \nu_1|^2.
	\end{align*}
Note that
	\begin{align*}
	\sum_{i,j,k}|\nabla_i h_{jk}+\langle\nabla^\bot_i A^-_{jk},\nu_1\rangle|^2=\sum_{i,j,k}|\nabla_i \mathring{h}_{jk}+\langle \nabla^\bot_i A^-_{jk},\nu_1\rangle|^2+\frac{1}{n}|\nabla |H||^2
	\end{align*}
and
	\begin{align*}
	\sum_{i,j,k}|\langle\nabla^\bot_i \mathring{A}_{jk},\nu_1\rangle|^2=\sum_{i,j,k}|\nabla_i\mathring{h}_{jk}+\langle\nabla^\bot_i A^-_{jk},\nu_1\rangle|^2.
	\end{align*}
Also, since
	\begin{align*}
	\frac{2(n-1)}{n(n+2)}|\nabla |H||^2\le \sum_{i,j,k}|\langle \nabla^\bot_i \mathring{A}_{jk},\nu_1\rangle|^2,
	\end{align*}
we have
	\begin{align*}
	\sum_{i,j,k}|\langle\nabla_i^{\perp} A^-_{j k}, \nu_1\rangle+\nabla_i h_{j k}|^2-c_n|\nabla| H||^2&=\sum_{i,j,k}|\langle\nabla_i^{\perp} \mathring{A}_{j k}, \nu_1)|^2-\frac{n c_n-1}{n}|\nabla| H||^2\\
	&\geq\left(1-\frac{(n+2)\left(n c_n-1\right)}{2(n-1)}\right)\sum_{i,j,k}|\langle\nabla_i^{\perp} \mathring{A}_{j k}, \nu_1\rangle|^2.
	\end{align*}
In view of \eqref{eq4.11} and \eqref{4.20naff}, we have
	\begin{align*}
	\sum_{i,j,k}|\hat{\nabla}_i^\bot A^-_{jk}+h_{jk} \nabla_i^\bot\nu_1|^2&-c_n|H|^2|\nabla^\bot \nu_1|^2\geq\left(\frac{3}{n+2}-c_n\right)|H|^2|\nabla^\bot\nu_1|^2 \\
	&=\frac{n}{n c_n-1}\left(\frac{3}{n+2}-c_n\right)(f+|A^-|^2+|\mathring{h}|^2+d_n)|\nabla^\bot\nu_1|^2 \\
	&\geq \frac{n}{n c_n-1}\left(\frac{3}{n+2}-c_n\right) f|\nabla^\bot \nu_1|^2.
	\end{align*}
Thus, by the three previous computations, we have
	\begin{align*}
	2 \frac{|A^-|^2}{f}\left(|\nabla^{\perp} A|^2-c_n|\nabla^{\perp} H|^2\right)&\geq\left(2-\frac{(n+2)\left(n c_n-1\right)}{n-1}\right) \frac{|A^-|^2}{f}\sum_{i,j,k}|\langle\nabla_i^{\perp} \mathring{A}_{j k}, \nu_1\rangle|^2 \\
	&+\frac{2 n}{n c_n-1}\left(\frac{3}{n+2}-c_n\right)|A^-|^2|\nabla^{\perp} \nu_1|^2.
	\end{align*}
If $c_n \leq \frac{4}{3 n}$, then $n c_n-1 \leq \frac{1}{3}$ and
	\begin{align*}
	2-\frac{(n+2)\left(n c_n-1\right)}{n-1}&\geq 2-\frac{n+2}{3(n-1)}=\frac{5 n-8}{3(n-1)}, \\
	\frac{2 n}{n c_n-1}\left(\frac{3}{n+2}-c_n\right)&\geq 6 n\left(\frac{9 n-4(n+2)}{3 n(n+2) }\right)=\frac{10 n-16}{n+2}.
	\end{align*}
This establishes the first inequality of the lemma.
 If $c_n \leq \frac{3(n+1)}{2 n(n+2)}$, then $n c_n-1 \leq \frac{n-1}{2(n+2)}$ and
	\begin{align*}
	&2-\frac{(n+2)\left(n c_n-1\right)}{n-1} \geq 2-\frac{1}{2}=\frac{3}{2}, \\
	&\frac{2 n}{n c_n-1}\left(\frac{3}{n+2}-c_n\right) \geq \frac{4 n(n+2)}{n-1}\left(\frac{6 n-3(n+1)}{2 n(n+2)}\right)=6.
	\end{align*}
This establishes the second inequality of the lemma.
\end{proof}
\begin{lemma}[Upper bound for gradient term of $\left(\partial_t-\Delta\right)|A^-|^2$ ]\label{4.8} \ \newline
\begin{enumerate}
\item If $\frac{1}{n}<c_n \leq \frac{4}{3 n}$, then
	\begin{align*}
	4 \sum_{i,j,k}Q_{i j k}\langle A^-_{i j}, \nabla_k^{\perp} \nu_1\rangle&\leq 2|\langle\nabla^{\perp} A^-, \nu_1\rangle|^2+\frac{5 n-9}{3(n-1)} \frac{|A^-|^2}{f}|\langle\nabla^{\perp} \mathring{A}, \nu_1\rangle|^2 \\
	&+2|A^-|^2|\nabla^{\perp} \nu_1|^2+\frac{3(n-1)}{n-3} f|\nabla^{\perp} \nu_1|^2+\frac{2(n+2)}{n+3}|\mathring{h}|^2|\nabla^{\perp} \nu_1|^2.
	\end{align*}
\item If $\frac{1}{n}<c_n \leq \frac{3(n+1)}{2 n(n+2)}-\e_0$ and $\e=\frac{2 n(n+2)}{3(n-1)} \e_0$, then
	\begin{align*}
	4 \sum_{i,j,k}Q_{i j k}\langle A^-_{i j}, \nabla_k^{\perp} \nu_1\rangle&\leq 2|\langle\nabla^{\perp} A^-, \nu_1\rangle|^2+(1-\e) \frac{3}{2} \frac{|A^-|^2}{f}|\langle\nabla^{\perp}\mathring{A}, \nu_1\rangle|^2 \\
	&+2|A^-|^2|\nabla^{\perp} \nu_1|^2+4 f|\nabla^{\perp} \nu_1|^2+2|\mathring{h}|^2|\nabla^{\perp} \nu_1|^2.
	\end{align*}
\end{enumerate}
\end{lemma}
\begin{proof}
Using the definition of $Q_{ijk}$, we get
	\begin{align}\label{4.30}
	|Q| \leq|\langle\nabla^{\perp} \mathring{A}, \nu_1\rangle|+|\langle\nabla^{\perp}A^-, \nu_1\rangle|+|H|^{-1}|\mathring{h}||\nabla| H||
	\end{align}
We will first treat the case $\frac{1}{n}<c_n \leq \frac{4}{3 n}$. It easily follows from the definition of $f$ that
	\begin{align*}
	f \leq\left(c_n-\frac{1}{n}\right)|H|^2 \leq \frac{1}{3 n}|H|^2.
	\end{align*}
Consequently, using the estimate
	\begin{align}\label{4.21}
	\frac{2(n-1)}{n(n+2)}|\nabla |H||^2\le\sum_{i,j,k}|\langle\nabla^\bot_i \mathring{A}_{jk},\nu_1\rangle|^2
	\end{align}
we obtain
	\begin{align}\label{4.31}
	\frac{|A^-|^2}{|H|^2}|\nabla| H||^2\leq \frac{n(n+2)}{2(n-1)} \frac{1}{3 n} \frac{|A^-|^2}{f}|\langle\nabla^{\perp} \mathring{A}, \nu_1\rangle|^2=\frac{n+2}{6(n-1)} \frac{|A^-|^2}{f}|\langle\nabla^{\perp} \mathring{A}, \nu_1\rangle|^2.
	\end{align}
Then
	\begin{align*}
	|\langle A^-,\nabla^\bot \nu_1\rangle|^2=\sum_{i,j} \langle A^-_{ij},\nabla^\bot_i \nu_1\rangle^2 \le\sum_{i,j,k}\sum_{\beta\ge 2} (A^{\beta}_{ij})^2 \langle \nabla^\bot_k \nu_1,\nu_\beta\rangle^2
	\end{align*}
and \eqref{4.30} give
	\begin{align*}
	4\sum_{i,j,k} Q_{i j k}\langle A^-_{i j}, \nabla_k^{\perp} \nu_1\rangle&\leq 4|Q||\langle A^-, \nabla^{\perp} \nu_1\rangle| \\
	&\leq 4\left(|\langle\nabla^{\perp}\mathring{A}, \nu_1\rangle|+|\langle\nabla^{\perp} A^-, \nu_1\rangle|+|H|^{-1}|\mathring{h} | \nabla| H||\right)|A^-||\nabla^{\perp} \nu_1|.
	\end{align*}
Now to each of these three summed terms above we apply Young's inequality with constants $a_1, a_2, a_3>0$. Specifically, we have
	\begin{align*}
	4|\langle\nabla^{\perp} A^-, \nu_1\rangle||A^-||\nabla^{\perp} \nu_1|&\leq 2 a_1|\langle\nabla^{\perp} A^-, \nu_1\rangle|^2+\frac{2}{a_1}|A^-|^2|\nabla^{\perp} \nu_1|^2, \\
	4|\langle\nabla^{\perp} \mathring{A}, \nu_1\rangle||A^-||\nabla^{\perp} \nu_1|&=4|\langle\nabla^{\perp} \mathring{A}, \nu_1\rangle| \frac{|A^-|}{\sqrt{f}} f^{\frac{1}{2}}|\nabla^{\perp} \nu_1| \\
	&\leq 2 a_2 \frac{|A^-|^2}{f}|\langle\nabla^{\perp} \mathring{A}, \nu_1\rangle|^2+\frac{2}{a_2} f|\nabla^{\perp} \nu_1|^2, \\
	4|H|^{-1}|\mathring{h}||\nabla| H|||A^-||\nabla^{\perp} \nu_1|&\leq 2 a_3 \frac{|A^-|^2}{|H|^2}\left|\nabla| H|\right|^2+\frac{2}{a_3}|\mathring{h}|^2 |\nabla^{\perp} \nu_1|^2 \\
	&\leq 2 a_3 \frac{n+2}{6(n-1)} \frac{|A^-|^2}{f}|\langle\nabla^{\perp} \mathring{A}, \nu_1\rangle|^2+\frac{2}{a_3}|\mathring{h}|^2|\nabla^{\perp} \nu_1|^2.
	\end{align*}
Note we used \eqref{4.31} in the last inequality. Hence
	\begin{align}\label{4.32}
	4 \sum_{i,j,k}Q_{i j k}\langle A^-_{i j}, \nabla_k^{\perp} \nu_1\rangle&\leq 2 a_1|\langle\nabla^{\perp} A^-, \nu_1\rangle|^2+\left(2 a_2+2 a_3 \frac{n+2}{6(n-1)}\right) \frac{|A^-|^2}{f}|\langle\nabla^{\perp} \mathring{A}, \nu_1\rangle|^2 \nonumber\\
	&+\frac{2}{a_1}|A^-|^2|\nabla^{\perp} \nu_1|^2+\frac{2}{a_2} f|\nabla^{\perp} \nu_1|^2+\frac{2}{a_3}|\mathring{h}|^2|\nabla^{\perp} \nu_1|^2.
	\end{align}
Now set
	\begin{align*}
	a_1=1, \ \ \ a_2=\frac{2(n-3)}{3(n-1)}, \ \ \ a_3=\frac{n+3}{n+2}.
	\end{align*}
In this case,
	\begin{align*}
	2 a_2+2 a_3 \frac{n+2}{6(n-1)}&=\frac{4(n-3)}{3(n-1)}+\frac{n+3}{n+2} \frac{n+2}{3(n-1)}=\frac{5 n-9}{3(n-1)}, \\
	\frac{2}{a_2}&=\frac{3(n-1)}{n-3}, \\
	\frac{2}{a_3}&=\frac{2(n+2)}{n+3}.
	\end{align*}
Plugging these into \eqref{4.32}, we have the first inequality as claimed.
Now if $\frac{1}{n}<c_n \leq \frac{3(n+1)}{2 n(n+2)}-\e_0$, then $c_n-\frac{1}{n} \leq \frac{n-1}{2 n(n+2)}-\e_0$. Therefore, if we take $\e=\frac{2 n(n+2)}{3(n-1)} \e_0$, then
	\begin{align*}
	c_n-\frac{1}{n} \leq(1-3 \e) \frac{n-1}{2 n(n+2)}.
	\end{align*}
In this case,
	\begin{align*}
	f \leq\left(c_n-\frac{1}{n}\right)|H|^2 \leq(1-3 \e) \frac{n-1}{2 n(n+2)}|H|^2.
	\end{align*}
Again using the definition of $f$, it follows that
	\begin{align*}
	\frac{|A^-|^2}{|H|^2}|\nabla| H||^2&\leq(1-3 \e) \frac{n(n+2)}{2(n-1)} \frac{n-1}{2 n(n+2)} \frac{|A^-|^2}{f}|\langle\nabla^{\perp} \mathring{A}, \nu_1\rangle|^2 \\
	&=\frac{1}{4}(1-3 \e) \frac{|A^-|^2}{f}|\langle\nabla^{\perp} \mathring{A}, \nu_1\rangle|^2.
	\end{align*}
Proceeding as we did before, we obtain the inequality
	\begin{align}\label{4.33}
	4 \sum_{i,j,k}Q_{i j k}\langle A^-_{i j}, \nabla_k^{\perp} \nu_1\rangle&\leq 2 a_1|\langle\nabla^{\perp} A^-, \nu_1\rangle|^2+\left(2 a_2+\frac{1}{2} a_3(1-3 \e)\right) \frac{|A^-|^2}{f}|\langle\nabla^{\perp} \mathring{A}, \nu_1\rangle|^2\nonumber \\
	&+\frac{2}{a_1}|A^-|^2|\nabla^{\perp} \nu_1|^2+\frac{2}{a_2} f|\nabla^{\perp} \nu_1|^2+\frac{2}{a_3}|\mathring{h}|^2|\nabla^{\perp} \nu_1|^2.
	\end{align}
Set
	\begin{align*}
	a_1=1, \quad a_2=\frac{1}{2}, \quad a_3=1.
	\end{align*}
In this case,
	\begin{align*}
	2 a_2+\frac{1}{2} a_3(1-3\e)&=\frac{3}{2}(1-\e), \\
	\frac{2}{a_2}&=4, \\
	\frac{2}{a_3}&=2.
	\end{align*}
Plugging these into \eqref{4.33}, we get the second inequality as claimed.
\end{proof}
Finally, putting the conclusions of Lemma \ref{4.6}, \ref{4.7} and \ref{4.8} together, we get the following result.
\begin{lemma}[Gradient term estimate]\label{4.9} \ \newline
Suppose either $n \geq 8, \frac{1}{n}<c_n \leq \frac{4}{3 n}$ and $0<\delta \leq \frac{1}{5 n-8}$ or $\frac{1}{n}<c_n \leq \frac{3(n+1)}{2 n(n+2)}-\e_0$, and $0<\delta \leq \min \left\{\frac{1}{2}, \frac{2 n(n+2)}{3(n-1)} \e_0\right\}$. Then in either case,
	\begin{align*}
	4\sum_{i,j,k} Q_{i j k}\langle A^-_{i j}, \nabla_k^{\perp} \nu_1\rangle \leq 2|\nabla^{\perp} A^-|^2+2(1-\delta) \frac{|A^-|^2}{f}\left(|\nabla^{\perp} A|^2-c_n|\nabla^{\perp} H|^2\right).
	\end{align*}
\end{lemma}
\begin{proof}
First suppose $n \geq 8, \frac{1}{n}<c_n \leq \frac{4}{3 n}$ and $0<\delta \leq \frac{1}{5 n-8}$. Expanding $|\nabla^{\perp} A^-|^2$ using
	\begin{align}\label{2.24}
	|\nabla^\bot A^-|^2=|\hat{\nabla}^\bot A^-|^2+|\langle\nabla^\bot A^-,\nu_1\rangle|^2
	\end{align}
and using the inequality \eqref{2.23} in Lemma \ref{4.6} gives us
	\begin{align*}
	2|\nabla^{\perp} A^-|^2&=2|\hat{\nabla}^{\perp} A^-|^2+2|\langle\nabla^{\perp}A^-, \nu_1\rangle|^2 \\
	&\geq 2|\langle\nabla^{\perp} A^-, \nu_1\rangle|^2+\frac{4 n-10}{n+2}|\mathring{h}|^2|\nabla^{\perp} \nu_1|^2+\frac{6(n-1)}{n+2}(|A^-|^2+f+d_n)|\nabla^{\perp} \nu_1|^2.
	\end{align*}
Multiplying the first result in Lemma \ref{4.7} by $(1-\delta)$ and using that $1-\delta \geq \frac{1}{2}$ on the coefficient of $|A^-|^2|\nabla^{\perp} \nu_1|^2$ gives
	\begin{align*}
	2(1-\delta) \frac{|A^-|^2}{f}(\left|\nabla^{\perp} A|^2-c_n|\nabla^{\perp} H|^2\right) &\geq(1-\delta) \frac{5 n-8}{3(n-1)} \frac{|A^-|^2}{f}|\langle\nabla^{\perp} \mathring{A}, \nu_1\rangle|^2\\
	&+\frac{5 n-8}{n+2}|A^-|^2|\nabla^{\perp} \nu_1|^2.
	\end{align*}
Putting these together, we get
	\begin{align*}
	2|\nabla^{\perp} A^-|^2&+2(1-\delta) \frac{|A^-|^2}{f}\left(|\nabla^{\perp} A|^2-c_n|\nabla^{\perp} H|^2\right) \geq 2|\langle\nabla^{\perp} A^-, \nu_1\rangle|^2\\
	&+(1-\delta) \frac{5 n-8}{3(n-1)} \frac{|A^-|^2}{f}|\langle\nabla^{\perp} \mathring{A}, \nu_1\rangle|^2 +\frac{11 n-14}{n+2}|A^-|^2|\nabla^{\perp} \nu_1|^2\\
	&+\frac{6(n-1)}{n+2} (f+d_n)|\nabla^{\perp} \nu_1|^2+\frac{4 n-10}{n+2}|\mathring{h}|^2|\nabla^{\perp} \nu_1|^2.
	\end{align*}
On the other hand, the first result of Lemma \ref{4.8} gives us that
	\begin{align*}
	4 \sum_{i,j,k}Q_{i j k}\langle A^-_{i j}, \nabla_k^{\perp} \nu_1\rangle &\leq 2|\langle\nabla^{\perp} A^-, \nu_1\rangle|^2+\frac{5 n-9}{3(n-1)} \frac{|A^-|^2}{f}|\langle\nabla^{\perp} \mathring{A}, \nu_1\rangle|^2 \\
	&+2|A^-|^2|\nabla^{\perp} \nu_1|^2+\frac{3(n-1)}{n-3} f|\nabla^{\perp} \nu_1|^2+\frac{2(n+2)}{n+3}|\mathring{h}|^2|\nabla^{\perp} \nu_1|^2.
	\end{align*}
Therefore, it only remains to compare the coefficients of like terms in the two inequalities above. For the coefficients of $\frac{|A^-|^2}{f}|\langle\nabla^{\perp} \mathring{A}, \nu_1\rangle|^2$, we need
	\begin{align*}
	\frac{5 n-9}{3(n-1)} \leq(1-\delta) \frac{5 n-8}{3(n-1)} \Longleftrightarrow \delta \leq \frac{1}{5 n-8}.
	\end{align*}
Comparing the coefficients of the remaining terms implies we need
	\begin{align*}
	2 n+4 \leq 11 n-14 \Longleftrightarrow 2 \leq n,
	\end{align*}
	\begin{align*}
	n+2 \leq 2(n-3) \Longleftrightarrow 8 \leq n
	\end{align*}
and
	\begin{align*}
	2(n+2)^2 \leq(4 n-10)(n+3) \Longleftrightarrow 19 \leq n(n-3).
	\end{align*}
Each of these inequalities is true if $n \geq 8$ completing the proof for the first case.
Now suppose $\frac{1}{n}<c_n<\frac{3(n+1)}{2 n(n+2)}$ and $0<\delta<\min \left\{\frac{1}{2}, \frac{2 n(n+2)}{3(n-1)} \e_0\right\}$. Arguing as before, this time using the second result in Lemma \ref{4.6} and the second result in Lemma \ref{4.7} yields
	\begin{align*}
	2\left|\nabla^{\perp}A^-\right|^2&+2(1-\delta) \frac{|A^-|^2}{f}\left(|\nabla^{\perp} A|^2-c_n|\nabla^{\perp} H|^2\right) \geq 2|\langle\nabla^{\perp} A^-, \nu_1\rangle|^2\\
	&+(1-\delta) \frac{3}{2} \frac{|A^-|^2}{f}|\langle\nabla^{\perp} \mathring{A}, \nu_1\rangle|^2+7|A^-|^2|\nabla^{\perp} \nu_1|^2+4 (f+d_n)|\nabla^{\perp} \nu_1|^2\\
	&+2|\mathring{h}|^2|\nabla^{\perp} \nu_1|^2.
	\end{align*}
Note we again used $\delta \leq \frac{1}{2}$ to simplify the coefficient of $|A^-|^2|\nabla^{\perp} \nu_1|^2$. On the other hand, by the second result in Lemma \ref{4.8}, we have
	\begin{align*}
	4 \sum_{i,j,k}Q_{i j k}\langle A^-_{i j}, \nabla_k^{\perp} \nu_1\rangle&\leq 2|\langle\nabla^{\perp} A^-,\nu_1\rangle|^2+(1-\delta) \frac{3}{2}\frac{|A^-|^2}{f}|\langle\nabla^{\perp}\mathring{A}, \nu_1\rangle|^2+f|A^-|^2|\nabla^\bot \nu_1|^2\\
	&+4 (f+d_n)|\nabla^{\perp} \nu_1|^2+2|\mathring{h}|^2|\nabla^{\perp} \nu_1|^2\\
	&\le 2|\nabla^\bot A^-|^2+2(1-\delta)\frac{|A^-|^2}{f}\big(|\nabla^\bot A|^2-c_n|\nabla^\bot H|^2\big),
	\end{align*}
where recall $\e=\frac{2 n(n+2)}{3(n-1)} \e_0$. Using the assumption that $\delta \leq \e$, this completes the proof of the lemma for the second case.
\end{proof}
   	       	 	                                         	           	 	 	 	 	     	 	 	 	     	 	 	 	 	         	 	         	         	      	 	 	                 	                            	    	     	                         	              	                   	                	 	     	     	 	     	                    	 	     	    We complete the proof of Theorem \ref{Th1}. Let $\delta$ be sufficiently small so that each of our above calculations hold. We begin by splitting off the desired nonpositive term in the evolution equation.
	\begin{align*}
	\Big(\partial_t-\Delta\Big)\frac{|A^-|^2}{f}&=\frac{1}{f}\Big(\partial_t-\Delta\Big)|A^-|^2-|A^-|^2\frac{1}{f^2}\Big(\partial_t-\Delta\Big)f+2\Big\langle\nabla \frac{|A^-|^2}{f},\nabla\log f\Big\rangle\\
	&=2\Big\langle\nabla\frac{|A^-|^2}{f},\nabla\log f\Big\rangle-\delta\frac{|A^-|^2}{f^2}\Big(\partial_t-\Delta\Big)f\\
	&+\frac{1}{f}\Big(\partial_t-\Delta\Big)|A^-|^2-(1-\delta)\frac{|A^-|^2}{f^2}\Big(\partial_t-\Delta\Big)f.
	\end{align*}
Using the previous calculations, the sum of the terms at the second line are non positive:
	\begin{align*}
	\frac{1}{f}\Big(\partial_t&-\Delta\Big)|A^-|^2-(1-\delta)\frac{|A^-|^2}{f^2}\Big(\partial_t-\Delta\Big)f\\
	&=\frac{1}{f}\Big(2\sum_{i,j,p,q}|\langle A_{ij}^- ,A_{pq}^- \rangle|^2+2|\hat{R}^\bot|^2+2\sum_{i,j}|R_{ij}^\bot (\nu_1)|^2\Big)\\
	&+\frac{1}{f}\Big(4\sum_{i,j,p,q}\bar{R}_{ipjq}\big(\sum_{\alpha\ge 2} A^\alpha_{pq}A^\alpha_{ij}\big)-4\sum_{j,k,p}\bar{R}_{kjkp}\big(\sum_{i,\alpha\ge 2} A^\alpha_{pi}A^\alpha_{ij}\big)\Big)\\
	&+\frac{1}{f}\Big(2\sum_{k,\alpha,\beta\ge 2}\bar{R}_{k\alpha k\beta}\big(\sum_{i,j} A^\alpha_{ij}A_{ij}^\beta \big)-8\sum_{j,p,\alpha,\beta\ge 2}\bar{R}_{jp\alpha\beta}\big(\sum_iA^\alpha_{ip}A_{ij}^\beta \big)\Big)\\
	&+\frac{1}{f}\Big(2|H|^{-2}\sum_{i,j,k,\alpha,\beta\ge 2}\bar{R}_{k\alpha k\beta} H^\alpha A^{\beta}_{ij}\langle A_{ij},H\rangle+2\sum_{i,j,k,\beta\ge 2}\bar{\nabla}_k\bar{R}_{kij\beta}A_{ij}^\beta-2\sum_{i,j,k,\beta\ge 2}\bar{\nabla}_i\bar{R}_{jkk\beta}A_{ij}^\beta\Big)\\
	&+\frac{1}{f}\Big(2\sum_{i,j}|\bar{R}_{ij}(\nu_1)|^2+4\sum_{i,j,p}\langle\bar{R}_{ij}(\nu_1),\mathring{h}_{ip}A^-_{jp}-\mathring{h}_{jp}A^-_{ip}\rangle \Big)\\
	&+\frac{1}{f}\Big(4\sum_{i,j,k}Q_{ijk}\langle A_{ij}^- ,\nabla^\bot_k \nu_1\rangle-2|\nabla^\bot A^-|^2 -2(1-\delta)\frac{|A^-|^2}{f}\big( |\nabla^\bot A|^2-c_n|\nabla^\bot H|^2 \big) \Big)\\
	&-(1-\delta)\Big(-2\frac{|A^-|^2}{f^2}\Big(c_n\sum_{i,j}|\langle A_{ij},H\rangle|^2-\sum_{i,j,p,q}|\langle A_{ij},A_{pq}\rangle|^2-\sum_{i,j}|R_{ij}^\bot |^2\Big)\\
	&+\frac{|A^-|^2}{f^2}\big(4\sum_{i,j,p,q}\bar{R}_{ipjq}\big(\sum_{\alpha} A^\alpha_{pq}A^\alpha_{ij}\big)-4\sum_{j,k,p}\bar{R}_{kjkp}\big(\sum_{i,\alpha} A^\alpha_{pi}A^\alpha_{ij}\big)-2\sum_k \bar{R}_{k1k1}\sum_{i,j} (A^1_{ij})^2\big)\\
	&+\frac{|A^-|^2}{f^2}\big(2\sum_{k,\alpha,\beta}\bar{R}_{k\alpha k\beta}\big(\sum_{i,j} A^\alpha_{ij}A_{ij}^\beta \big)+4c_n\sum_{k,\alpha,\beta}\bar{R}_{k\alpha k\beta} H^\alpha H^\beta -8\sum_{j,p,\alpha,\beta}\bar{R}_{jp\alpha\beta}\big(\sum_iA^\alpha_{ip}A_{ij}^\beta \big)\big)\\
	&+\frac{|A^-|^2}{f^2}\big(2\sum_{i,j,k,\beta}\bar{\nabla}_k\bar{R}_{kij\beta}A_{ij}^\beta -2\sum_{i,j,k,\beta}\bar{\nabla}_i\bar{R}_{jkk\beta}A_{ij}^\beta \big)\Big)\\
	&\le(1-\delta)\frac{|A^-|^2}{f^2}\Big(\sum_{\alpha,\beta\ge 2}\Big(\sum_{i,j}|\bar{R}_{ij\alpha\beta}|^2+4|\bar{R}_{ij\alpha\beta}||A^-|^2\Big)+2\sum_{i,j}|\bar{R}_{ij}(\nu_1)|^2+8|\bar{R}_{ij}(\nu_1)||\mathring{h}||A^-|\Big)\\
	&+C\frac{|A^-|^2}{f}+C''\frac{|A^-|}{\sqrt{f}}+\frac{1}{f}\Big(2\sum_{i,j}|\bar{R}_{ij}(\nu_1)|^2+4\sum_{i,j,p}\langle\bar{R}_{ij}(\nu_1),\mathring{h}_{ip}A^-_{jp}-\mathring{h}_{jp}A^-_{ip}\rangle \Big),
	\end{align*}
for $C,C''$ constants depending on $n,K_1,K_2$ and $d_n$. Repeating the same estimate as before, we can see
	\begin{align*}
	&(1-\delta)\frac{|A^-|^2}{f^2}\Big(\sum_{\alpha,\beta\ge 2}\Big(\sum_{i,j}|\bar{R}_{ij\alpha\beta}|^2+4|\bar{R}_{ij\alpha\beta}||A^-|^2\Big)+2\sum_{i,j}|\bar{R}_{ij}(\nu_1)|^2+8|\bar{R}_{ij}(\nu_1)||\mathring{h}||A^-|\Big)\\
	&+C\frac{|A^-|^2}{f}+C''\frac{|A^-|}{\sqrt{f}}+\frac{1}{f}\Big(2\sum_{i,j}|\bar{R}_{ij}(\nu_1)|^2+4\sum_{i,j,p}\langle\bar{R}_{ij}(\nu_1),\mathring{h}_{ip}A^-_{jp}-\mathring{h}_{jp}A^-_{ip}\rangle \Big)\\
	&=(1-\delta)\frac{|A^-|^2}{f}\Big(\sum_{\alpha,\beta\ge 2}\Big( \frac{\sum_{i,j}|\bar{R}_{ij\alpha\beta}|^2}{f}+4|\bar{R}_{ij\alpha\beta}|\frac{|A^-|^2}{f}\Big)+ 2\frac{\sum_{i,j}|\bar{R}_{ij}(\nu_1)|^2}{f}+8|\bar{R}_{ij}(\nu_1)|\frac{|\mathring{h}|}{\sqrt{f}}\frac{|A^-|}{\sqrt{f}}\Big)
\\
	&+C\frac{|A^-|^2}{f}+C''\frac{|A^-|}{\sqrt{f}}+\frac{1}{f}\Big(2\sum_{i,j}|\bar{R}_{ij}(\nu_1)|^2+4\sum_{i,j,p}\langle\bar{R}_{ij}(\nu_1),\mathring{h}_{ip}A^-_{jp}-\mathring{h}_{jp}A^-_{ip}\rangle \Big)\\
	&\le C'\frac{|A^-|^2}{f}+C''\frac{|A^-|}{\sqrt{f}}+\frac{1}{f}\Big(2\sum_{i,j}|\bar{R}_{ij}(\nu_1)|^2+4\sum_{i,j,p}\langle\bar{R}_{ij}(\nu_1),\mathring{h}_{ip}A^-_{jp}-\mathring{h}_{jp}A^-_{ip}\rangle \Big),
	\end{align*}
were the last term on the last row is bounded from above. Thus, according to our previous calculations we get \eqref{initialclaim}, which was our initial claim:
	\begin{align*}
	\Big(\partial_t-\Delta\Big)\frac{|A^-|^2}{f}&\le2\Big\langle\nabla\frac{|A^-|^2}{f},\nabla\log f\Big\rangle-\delta\frac{|A^-|^2}{f^2}\Big(\partial_t-\Delta\Big)f+C'\frac{|A^-|^2}{f}+C''\frac{|A^-|}{\sqrt{f}}\\
	&+\frac{1}{f}\Big(2\sum_{i,j}|\bar{R}_{ij}(\nu_1)|^2+4\sum_{i,j,p}\langle\bar{R}_{ij}(\nu_1),\mathring{h}_{ip}A^-_{jp}-\mathring{h}_{jp}A^-_{ip}\rangle \Big).
	\end{align*}
Recall $\Big(\partial_t-\Delta\Big)f$ is non negative at each point in space time.
 
Now,
	\begin{align*}
	\frac{1}{f}\Big(\partial_t-\Delta\Big)|A^-|^2-\frac{|A^-|^2}{f^2}\Big(\partial_t-\Delta\Big)f&= \Big(\partial_t-\Delta\Big)\frac{|A^-|^2}{f}-2\Big\langle\nabla\frac{|A^-|^2}{f},\nabla\log f\Big\rangle\\
	&\le-\delta\frac{|A^-|^2}{f^2}\Big(\partial_t-\Delta\Big)f+C'\frac{|A^-|^2}{f}+C''\frac{|A^-|}{\sqrt{f}}\\
	&+\frac{1}{f}\Big(2\sum_{i,j}|\bar{R}_{ij}(\nu_1)|^2+4\sum_{i,j,p}\langle\bar{R}_{ij}(\nu_1),\mathring{h}_{ip}A^-_{jp}-\mathring{h}_{jp}A^-_{ip}\rangle \Big).
	\end{align*}
                                                                                                                                                                                                     
We are now ready to prove the main theorem of this paper.
\begin{theorem}
Let $F: \mathcal{M}^n\times[0, T) \rightarrow \mathcal{N}^{n+m}$ be a smooth solution to mean curvature flow so that
$F_0(p)=F(p, 0)$ is compact and quadratically pinched.
Then $\forall \e>0, \exists H_0 >0$, such that if $f \geq H_0$, then
	\begin{align*}
	\left|A^-\right|^2 \leq \e f+C_{\e}
	\end{align*}
$\forall t \in[0, T)$ where $C_\e=C_{\e}(n, m)$.
\end{theorem}
\begin{proof}
Since $\mathcal{M}$ is quadratically bounded, there exist constants $ C,D$ such that
	\begin{align*}
	|A^-|^2 \leq Cf +D.
	\end{align*}
Therefore, the above estimate holds for all $\e\geq \frac{c_n}{\delta}$. Indeed, from the pinching $|A|^2\leq c_n|H|^2-d_n$, we can make a little bit more space so that
	\begin{align*}
	|A^-|^2+|A^+|^2=|A|^2\le (c_n-\delta)|H|^2-C_\delta-d_n
	\end{align*}
and therefore,
	\begin{align*}
	\delta|H|^2\le c_n|H|^2-|A|^2-d_n.
	\end{align*}
But since $|A^-|^2\le|A|^2\le c_n|H|^2$, we have $\frac{\delta |A^-|^2}{c_n}\le\frac{\delta|A|^2}{c_n}\le\delta|H|^2$, so
	\begin{align*}
	\frac{\delta}{c_n} |A^-|^2 \le c_n |H|^2-d_n -|A|^2= f,
	\end{align*}
which means that $|A^-|^2\le\frac{c_n}{\delta}f\le\e f+C_\e$. Hence, let $\e_0$ denote the infimum of such $\e$ for which the estimate is true and suppose $\e_0>0$. We will prove the theorem by contradiction. Hence, let us assume that the conclusions of the theorem are not true that is there exists a family of mean curvature flow $\mathcal{M}_t^k$ with points $(p_k,t_k)$ such that
	\begin{align}\label{eqn_limitepsilon0}
	\lim_{k\rightarrow \infty} \frac{|A_k^-(p_k,t_k)|^2}{f_k(p_k,t_k)}= \e_0
	\end{align}
with $\e_0>0$ and $ f_k(p_k,t_k)\rightarrow \infty$.
We perform a parabolic rescaling of $ \mathcal{M}_t^k $ in such a way that $f_k$ at $(p_k,t_k)$ becomes $1$. If we consider the exponential map $\exp_{\bar{p}}\colon T_{\bar{p}}\mathcal{N} \cong \mathbb{R}^{n+m}\to \mathcal{N}^{n+m}$ and $\gamma$ a geodesic, then for a vector $v\in T_{\bar{p}}\mathcal{N}$, then
	\begin{align*}
	\exp_{\bar{p}}(v)=\gamma_{\bar{p},\frac{v}{|v|}} (|v|), \ \ \gamma '(0)=\frac{v}{|v|} \ \ \text{ and} \ \ \gamma(0)=\bar{p}=F_k(p_k,t_k).
	\end{align*}
That is, if $F_k$ is the parameterisation of the original flow $ \mathcal{M}_t^k $, we let $ \hat r_k = \frac{1}{f_k(p_k,t_k)}$, and we denote the rescaled flow by $ \overline{\mathcal{M}}_t^k $ and we define its parameterisation by
	\begin{align*}
	\overline F_k (p,\tau) = \exp^{-1}_{F_k(p_k,t_k)} \circ F_k (p,\hat{r}^2_k \tau+t_k).
	\end{align*}
In the Riemannian case, when we change the metric after dilation, we do not need to multiply the immersion by the same constant as we would do in the Euclidean space. When we rescale the background space, following the example of the dilation of a sphere, we see that
	\begin{align*}
	\bar{g}_{ij}=\frac{1}{\hat{r}^2_k}g_{ij} \ \ \text{ and} \ \ \overline{K}=\hat{r}^2_k K,
	\end{align*}
where $K$ is the sectional curvature of $\mathcal{N}$. In the same way,
	\begin{align*}
	|\overline{A}|^2=\hat{r}^2_k |A|^2 \ \ \text{ and} \ \ |\overline{H}|^2=\hat{r}^2_k |H|^2.
	\end{align*}
Since $d_n$ depends on $n$ and the sectional curvature $K$, the new $\bar{d}_n$ depends on n and $\overline{K}$. Hence,
	\begin{align*}
	\bar{d}_n=\hat{r}^2_k d_n.
	\end{align*}
For $\hat{r}_k\to 0$, the background Riemannian manifold will converge to its tangent plane in a pointed $C^{d,\gamma}$ H\"older topology \cite{Petersen2016}. Therefore, we can work on the manifold $\mathcal{N}$ as we would work in a Euclidean space. For simplicity, we choose for every flow a local co-ordinate system centred at $ p_k$. In these co-ordinates we can write $0$ instead of $ p_k$. The parabolic neighbourhoods $\mathcal P^k ( p_k, t_k, \hat r_k L, \hat r_k^2 \theta)$ in the original flow becomes $ \overline{\mathcal P}^k(0,0,L,\theta)$. By construction, each rescaled flow satisfies
	\begin{align} \label{eqn_H1}
	\overline F_k (0,0) = 0, \quad \overline f_k (0,0) = 1.
	\end{align}
Indeed,
	\begin{align*}
	\overline F_k(0,0)=\exp_{F_k(0,0)}^{-1}\circ F_k(0,\hat r_k^2 \cdot 0)=0
	\end{align*}
and
	\begin{align*}
	\overline f_k (p,\tau)&=-|\overline A_k (p,\tau)|^2+c_n|\overline{H}_k (p,\tau)|^2-\bar{d}_n\\
	&=\hat r_k^2\Big(-|A_k(p,\hat r_k^2 \tau+t_k)|^2+c_n|H_k(p,\hat r_k^2 \tau+t_k)|^2-d_n\Big)\\
	&=\hat r_k^2 f_k(p,\hat r_k^2 \tau+t_k)
	\end{align*}
and so
	\begin{align*}
	\overline f_k (0,0)=\hat r_k^2 f_k(0,0)=1,
	\end{align*}
since $\hat r_k (0,0)=\frac{1}{f_k (0,0)}=1$ from the change of coordinates. The gradient estimates give us uniform bounds (depending only on the pinching constant) on $ |A_k|$ and its derivatives up to any order on a neighbourhood of the form $\overline{\mathcal P }^k( 0 ,0,d,d)$ for a suitable $ d > 0$. From Theorem \eqref{thm_gradient}, we obtain gradient estimates on the second fundamental form in $ C^\infty $ on $ \overline F_k$. Hence we can apply Arzela-Ascoli (via the Langer-Breuning compactness theorem \cite{Breuning2015} and \cite{Langer1985}) and conclude there exists a subsequence converging in $ C^\infty $ to some limit flow which we denote by $ \widetilde{\mathcal{M}}_\tau^\infty$. We analyse the limit flow $ \widetilde{\mathcal{M}}_\tau^\infty$. Note we have for the Weingarten map
	\begin{align*}
		[\overline A_k^-]_i^j ( p , \tau) = \hat r_k [A_k^-]_i^j ( p , \hat r_k^2 \tau+t_k),
	\end{align*}
so that
	\begin{align*}
	\frac{|\overline A_k^-(p, \tau) |^2}{\overline{f}_k(p,\tau)}&= \frac{|A_k^- ( p, \hat r_k^2 \tau+t_k) |^2}{f_k(p, \hat r_k^2 \tau+t_k)}.
	\end{align*}
From \eqref{eqn_limitepsilon0} and \eqref{eqn_H1}, we see
	\begin{align*}
	\frac{| \widetilde A^-(0,0)|^2}{\widetilde f(0,0)}= \e_0, \quad \widetilde f(0,0) = 1.
	\end{align*}
We claim
	\begin{align*}
	\frac{|\widetilde A^- ( p , \tau) |^2}{\widetilde f (p, \tau)}=\lim_{k\rightarrow \infty}\frac{|\overline A_k^-(p,\tau)|^2}{\overline f_k (p,\tau)} \leq \e \quad \forall \e>\e_0.
	\end{align*}
 Since $\widetilde{f}(0,0)=1$, it follows that $|\widetilde{f}| \geq \frac{1}{2}$ in $ \widetilde{\mathcal{P}}^\infty(0,0,r,r)$ for some $r < d^\#$.
This is true since any point $( p, \tau) \in \widetilde{\mathcal{M}}^\infty_\tau$ is the limit of points $(p_{j_k},t_{j_k}) \in \overline{\mathcal{M}}^k$ and for every $ \e > \e_0 $ if we let $ \eta = \eta(\e,c_n)< d^\#$ then for large $k$, $\mathcal{M}^k$ is defined in
	\begin{align*}
	 \mathcal P^k\left(p_{j_k},t_{j_k}, \frac{1}{f_k(p_{j_k},t_{j_k})}\eta,\left(\frac{1}{f_k(p_{j_k},t_{j_k})}\right)^2 \eta\right)
	\end{align*}
which implies
	\begin{align*}
	\frac{|\overline A_k^- ( p_{j_k} , t_{j_k}) |^2}{\overline f_k (p_{j_k}, t_{j_k})} \leq \e \quad \forall \e>\e_0.
	\end{align*}
Hence the flow $\widetilde{\mathcal{M} }_t^\infty \subset \mathbb R^{n+m}$ has a space-time maximum $\e_0$ for $\frac{|\widetilde A^- ( p , \tau) |^2}{\widetilde f (p, \tau)}$ at $ (0,0)$. The evolution equation for $ \frac{|{A}^-|^2}{{f}}$ is given by 
	\begin{align*}
	\Big(\partial_t-\Delta\Big)\frac{| A^-|^2}{ f}&\le2\Big\langle\nabla\frac{| A^-|^2}{ f},\nabla\log  f\Big\rangle-\delta\frac{| A^-|^2}{ f^2}\Big(\partial_t-\Delta\Big) f+C'\frac{| A^-|^2}{ f}+C''\frac{| A^-|}{\sqrt{ f}}\\
	&+\frac{1}{ f}\Big(2\sum_{i,j}|\bar{R}_{ij}(\nu_1)|^2+4\sum_{i,j,p}\langle\bar{R}_{ij}(\nu_1),\mathring{h}_{ip}A^-_{jp}-\mathring{h}_{jp} A^-_{ip}\rangle \Big).
	\end{align*}
But in the limit our background space is Euclidean, therefore the background curvature tensor is identically zero. So the evolution equation becomes
	\begin{align*}
	\Big(\partial_t-\Delta\Big)\frac{|\widetilde A^-|^2}{\widetilde f}&\le2\Big\langle\nabla\frac{|\widetilde A^-|^2}{\widetilde f},\nabla\log \widetilde f\Big\rangle-\delta\frac{|\widetilde A^-|^2}{\widetilde f^2}\Big(\partial_t-\Delta\Big)\widetilde f.
	\end{align*}
Hence, since $\frac{|\widetilde A^-|^2}{\widetilde f}$ attains a maximum $\e_0$ at $(0,0)$ by the strong maximum principle then $ \frac{|\widetilde A^-|^2}{\widetilde f} \equiv \text{ constant}$.
             Therefore, there exists this constant $\mathcal{C}$ depending up to $d_n,K_1,K_2,L$ and $n$, such that
	\begin{align*}
	\mathcal{C}=\frac{|\widetilde A^-|^2}{\widetilde f}.
	\end{align*}
Putting this into the evolution equation we have
	\begin{align*}
	0\le-\delta\frac{\mathcal{C}}{\widetilde f}\Big(\partial_t-\Delta\Big)\widetilde f\le 0,
	\end{align*}
which means we get $\mathcal{C}=0$ and therefore, $|\widetilde A^-|=0$. This implies
	\begin{align*}
	\frac{|\widetilde A^-|^2}{\widetilde{f}}=0\implies \e_0=0,
	\end{align*}
which is a contradiction. Hence, we obtain
	\begin{align*}
	\lim_{k\rightarrow \infty} \frac{|\widetilde A_k^-(p_k,t_k)|}{\widetilde f_k(p_k,t_k)} = 0.
	\end{align*}
\end{proof}
\section{Cylindrical Estimates}
In this section, we present estimates that demonstrate an improvement in curvature as we approach a singularity. These estimates play a critical role in the analysis of high curvature regions in geometric flows. In particular, in the high codimension setting, we establish the quadratic pinching ratio $\frac{|A|^2}{|H|^2}$ approaches the ratio of the standard cylinder, which is $\frac{1}{n-1}$.
\begin{theorem}[\cite{HuSi09}]\label{thm_cylindrical}
Let $F: \mathcal{M}^n\times[0, T) \rightarrow \mathcal{N}^{n+m}$ be a smooth solution to mean curvature flow so that
$F_0(p)=F(p, 0)$ is compact and quadratically pinched with constant $c_n=\frac{1}{n-2}$.
Then $\forall \e>0, \exists H_1 >0$, such that if $f \geq H_1$, then
	\begin{align*}
	|A|^2- \frac{1}{n-1}|H|^2\leq \e f+C_{\e}
	\end{align*}
$\forall t \in[0, T)$ where $C_\e=C_{\e}(n, m)$.
\end{theorem}
\begin{proof}
Since $\mathcal{M}$ is quadratically bounded, there exist constants $ C,D$ such that
	\begin{align*}
	|A|^2- \frac{1}{n-1}|H|^2 \leq Cf +D.
	\end{align*}
Hence, let $\e_0$ denote the infimum of such $\e$ for which the estimate is true and suppose $\e_0>0$. We will prove the theorem by contradiction. Hence, let us assume that the conclusions of the theorem are not true that is there exists a family of mean curvature flow $\mathcal{M}_t^k$ with points $(p_k,t_k)$ such that
	\begin{align}\label{eqn_limitepsilon0versiontwo}
	\lim_{k\rightarrow \infty} \frac{\left (|A(p_k,t_k)|^2- \frac{1}{n-1}|H(p_k,t_k)|^2\right )}{f_k(p_k,t_k)}= \e_0
	\end{align}
with $\e_0>0$ and $ f_k(p_k,t_k)\rightarrow \infty$.	

We perform a parabolic rescaling of $ \mathcal{M}_t^k $ in such a way that $f_k$ at $(p_k,t_k)$ becomes $1$. If we consider the exponential map $\exp_{\bar{p}}\colon T_{\bar{p}}\mathcal{N} \cong \mathbb{R}^{n+m}\to \mathcal{N}^{n+m}$ and $\gamma$ a geodesic, then for a vector $v\in T_{\bar{p}}\mathcal{N}$, then
	\begin{align*}
	\exp_{\bar{p}}(v)=\gamma_{\bar{p},\frac{v}{|v|}} (|v|), \ \ \gamma '(0)=\frac{v}{|v|} \ \ \text{ and} \ \ \gamma(0)=\bar{p}=F_k(p_k,t_k).
	\end{align*}
That is, if $F_k$ is the parameterisation of the original flow $ \mathcal{M}_t^k $, we let $ \hat r_k = \frac{1}{f_k(p_k,t_k)}$, and we denote the rescaled flow by $ \overline{\mathcal{M}}_t^k $ and we define its parameterisation by
	\begin{align*}
	\overline F_k (p,\tau) = \exp^{-1}_{F_k(p_k,t_k)} \circ F_k (p,\hat{r}^2_k \tau+t_k).
	\end{align*}
In the Riemannian case, when we change the metric after dilation, we do not need to multiply the immersion by the same constant as we would do in the Euclidean space. When we rescale the background space, following the example of the dilation of a sphere, we see that
	\begin{align*}
	\bar{g}_{ij}=\frac{1}{\hat{r}^2_k}g_{ij} \ \ \text{ and} \ \ \overline{K}=\hat{r}^2_k K,
	\end{align*}
where $K$ is the sectional curvature of $\mathcal{N}$. In the same way,
	\begin{align*}
	|\overline{A}|^2=\hat{r}^2_k |A|^2 \ \ \text{ and} \ \ |\overline{H}|^2=\hat{r}^2_k |H|^2.
	\end{align*}
Since $d_n$ depends on $n$ and the sectional curvature $K$, the new $\bar{d}_n$ depends on n and $\overline{K}$. Hence,
	\begin{align*}
	\bar{d}_n=\hat{r}^2_k d_n.
	\end{align*}
For $\hat{r}_k\to 0$, the background Riemannian manifold will converge to its tangent plane in a pointed $C^{d,\gamma}$ H\"older topology \cite{Petersen2016}. Therefore, we can work on the manifold $\mathcal{N}$ as we would work in a Euclidean space. For simplicity, we choose for every flow a local co-ordinate system centred at $ p_k$. In these co-ordinates we can write $0$ instead of $ p_k$. The parabolic neighbourhoods $\mathcal P^k ( p_k, t_k, \hat r_k L, \hat r_k^2 \theta)$ in the original flow becomes $ \overline{\mathcal P}^k(0,0,L,\theta)$. By construction, each rescaled flow satisfies
	\begin{align} \label{eqn_H1version2}
	\overline F_k (0,0) = 0, \quad \overline f_k (0,0) = 1.
	\end{align}
Indeed,
	\begin{align*}
	\overline F_k(0,0)=\exp_{F_k(0,0)}^{-1}\circ F_k(0,\hat r_k^2 \cdot 0)=0
	\end{align*}
and
	\begin{align*}
	\overline f_k (p,\tau)&=-|\overline A_k (p,\tau)|^2+c_n|\overline{H}_k (p,\tau)|^2-\bar{d}_n\\
	&=\hat r_k^2\Big(-|A_k(p,\hat r_k^2 \tau+t_k)|^2+c_n|H_k(p,\hat r_k^2 \tau+t_k)|^2-d_n\Big)\\
	&=\hat r_k^2 f_k(p,\hat r_k^2 \tau+t_k)
	\end{align*}
and so
	\begin{align*}
	\overline f_k (0,0)=\hat r_k^2 f_k(0,0)=1,
	\end{align*}
since $\hat r_k (0,0)=\frac{1}{f_k (0,0)}=1$ from the change of coordinates. The gradient estimates give us uniform bounds (depending only on the pinching constant) on $ |A_k|$ and its derivatives up to any order on a neighbourhood of the form $\overline{\mathcal P }^k( 0 ,0,d,d)$ for a suitable $ d > 0$. From Theorem \eqref{thm_gradient}, we obtain gradient estimates on the second fundamental form in $ C^\infty $ on $ \overline F_k$. Hence we can apply Arzela-Ascoli (via the Langer-Breuning compactness theorem \cite{Breuning2015} and \cite{Langer1985}) and conclude there exists a subsequence converging in $ C^\infty $ to some limit flow which we denote by $ \widetilde{\mathcal{M}}_\tau^\infty$. We analyse the limit flow $ \widetilde{\mathcal{M}}_\tau^\infty$. Note we have for the Weingarten map
	\begin{align*}
		[\overline A_k^-]_i^j ( p , \tau) = \hat r_k [A_k^-]_i^j ( p , \hat r_k^2 \tau+t_k),
	\end{align*}
so that
so that
	\begin{align*}
	\frac{|\overline A_k(p, \tau) |^2-\frac{1}{n-1}|\overline H_k(p,\tau)|^2}{\overline{f}_k(p,\tau)}&= \frac{|A_k ( p, \hat r_k^2 \tau+t_k) |^2-\frac{1}{n-1}|H_k ( p, \hat r_k^2 \tau+t_k) |^2}{f_k(p, \hat r_k^2 \tau+t_k)}.
	\end{align*}
From \eqref{eqn_limitepsilon0} and \eqref{eqn_H1}, we see
	\begin{align*}
	\frac{| \widetilde A(0,0)|^2-\frac{1}{n-1}|\widetilde H(0,0)|^2}{\widetilde f(0,0)}= \e_0, \quad \widetilde f(0,0) = 1.
	\end{align*}
We claim
	\begin{align*}
	\frac{|\widetilde A( p , \tau) |^2-\frac{1}{n-1}|\widetilde H ( p , \tau) |^2}{\widetilde f (p, \tau)}=\lim_{k\rightarrow \infty}\frac{|\overline A_k(p,\tau)|^2-\frac{1}{n-1}|\overline H_k(p,\tau)|^2}{\overline f_k (p,\tau)} \leq \e \quad \forall \e>\e_0.
	\end{align*}
 Since $\widetilde{f}(0,0)=1$, it follows that $|\widetilde{f}| \geq \frac{1}{2}$ in $ \widetilde{\mathcal{P}}^\infty(0,0,r,r)$ for some $r < d^\#$.
This is true since any point $( p, \tau) \in \widetilde{\mathcal{M}}^\infty_\tau$ is the limit of points $(p_{j_k},t_{j_k}) \in \overline{\mathcal{M}}^k$ and for every $ \e > \e_0 $ if we let $ \eta = \eta(\e,c_n)< d^\#$ then for large $k$, $\mathcal{M}^k$ is defined in
	\begin{align*}
	 \mathcal P^k\left(p_{j_k},t_{j_k}, \frac{1}{f_k(p_{j_k},t_{j_k})}\eta,\left(\frac{1}{f_k(p_{j_k},t_{j_k})}\right)^2 \eta\right)
	\end{align*}
which implies
	\begin{align*}
	\frac{|\overline A_k ( p_{j_k} , t_{j_k}) |^2-\frac{1}{n-1}|\overline H_k( p_{j_k} , t_{j_k}) |^2}{\overline f_k (p_{j_k}, t_{j_k})} \leq \e \quad \forall \e>\e_0.
	\end{align*}
Hence the flow $\widetilde{\mathcal{M} }_t^\infty \subset \mathbb R^{n+m}$ has a space-time maximum $\e_0$ for $\frac{|\widetilde A( p , \tau) |^2-\frac{1}{n-1}|\widetilde H ( p , \tau) |^2}{\widetilde f (p, \tau)}$ at $ (0,0)$ which implies that the flow $\overline{\mathbb M }_t^\infty $ has a space-time maximum $\frac{1}{n-1}+\e_0$ for $\frac{|\overline A ( p , \tau) |^2}{| \overline H (p, \tau) |^2}$ at $ (0,0)$. Since the evolution equation for $ \frac{|A|^2}{|H|^2}$ is given by
	\begin{align*}
	\Big(\partial_t -\Delta\Big) \frac{|A|^2}{|H|^2}&= \frac{2}{|H|^2}\left\la \nabla |H|^2 , \nabla \left( \frac{|A|^2}{|H|^2}\right) \right\ra\\
	&-\frac{2}{|H|^2} \left( |\nabla A|^2-\frac{|A|^2}{|H|^2}|\nabla H|^2 \right) \\
	&+\frac{2}{|H|^2}\left( R_1-\frac{|A|^2}{|H|^2} R_2\right)
	\end{align*}
We have
	\begin{align*}
	|\nabla H|^2 \leq \frac{3}{n+2}|\nabla A|^2, \quad \frac{|A|^2}{|H|^2}\leq c_n
	\end{align*}
which gives
	\begin{align*}
	-\frac{2}{|H|^2} \left( |\nabla A|^2-\frac{|A|^2}{|H|^2}|\nabla H|^2 \right) \leq 0.
	\end{align*}
Furthermore, if $\frac{|A|^2}{|H|^2}=c < c_n$ then
	\begin{align*}
	 R_1-\frac{|A|^2}{|H|^2} R_2&= R_1-c R_2\\
	&\leq \frac{2}{n}\frac{1}{c-\nicefrac{1}{n}}| A_-|^2 \mathcal Q+\left(6-\frac{2}{n (c-\nicefrac{1}{n})} \right) |\circo A_1|^2 | \circo A_-|^2+\left(3-\frac{2}{n (c-\nicefrac{1}{n})} \right)|\circo A_-|^4\\
	&\leq 0.
	\end{align*}
Hence the strong maximum principle applies to the evolution equation of $\frac{|A|^2}{|H|^2}$ and shows $\frac{|A|^2}{|H|^2}$ is constant. The evolution equation then shows $ |\nabla A|^2= 0$, that is the second fundamental form is parallel and that $|A_-|^2 = |\circo A_-|^2=0$, that is the submanifold is codimension one. Finally this shows locally $ \mathbb M = \mathbb S^{n-k}\times \mathbb R^k$, \cite{Lawson1969}. As $\frac{|A|^2}{|H|^2}< c_n\leq \frac{1}{n-2}$ we can only have
	\begin{align*}
	\mathbb S^n, \mathbb S^{n-1}\times \mathbb R
	\end{align*}
which gives $\frac{|A|^2}{|H|^2}= \frac{1}{n}, \frac{1}{n-1}\neq \frac{1}{n-1}+\e_0, \e>0$ which gives a contradiction.
	
\end{proof}
\section{Singularity Models of Pinched Solutions of Mean Curvature Flow in Higher Codimension}
In this section, we derive a corollary from Theorem \ref{Th1}, which provides information about the blow up models at the first singular time. Specifically, we show that these models can be classified up to homothety.
\begin{corollary}[{\cite[Corollary 1.4]{Naff6} }]\label{resultnaff} Let $n \geq 5$ and $N>n$. Let $c_n=\frac{1}{n-2}$ if $n \geq 8$ and $c_n=\frac{3(n+1)}{2 n(n+2)}$ if $n=5,6$, or 7 . Consider a closed, n-dimensional solution to the mean curvature flow in $\mathbb{R}^N$ initially satisfying $|H|>0$ and $|A|^2<c_n|H|^2$. At the first singular time, the only possible blow-up limits are codimension one shrinking round spheres, shrinking round cylinders and translating bowl solitons.
\end{corollary}
According to Theorem \ref{Th1} and Theorem \ref{thm_cylindrical}, for $F: \mathcal{M}^n\times[0, T) \rightarrow \mathcal{N}^{n+m}$ be a smooth solution to mean curvature flow so that $F_0(p)=F(p, 0)$ is compact and quadratically pinched with $c_n=\frac{3(n+1)}{2 n(n+2)}$ if $n=5,6$, or $7$, then $\forall \e>0, \exists H_0, H_1 >0$, such that if $f \geq \max\{H_0,H_1\}$, then
	\begin{align*}
	\left|A^-\right|^2 \leq \e f+C_{\e} \quad \text{and} \quad  |A|^2- \frac{1}{n-1}|H|^2\leq \e f+C_{\e},
	\end{align*}
$\forall t \in[0, T)$ where $C_\e=C_{\e}(n, m)$. At the first singular time, the only possible blow-up limits are codimension one shrinking round spheres, shrinking round cylinders, and translating bowl solitons. Therefore, we can classify these blowup limits as follows:
\begin{corollary}[{\cite[Corollary 4.7]{HuSi99a}}]
Let $n \geq 5$. Let $c_n=\frac{1}{n-2}$ if $n \geq 8$ and $c_n=\frac{3(n+1)}{2 n(n+2)}$ if $n=5,6$, or $7$. Suppose $F_t\colon\mathcal{M}^n \rightarrow \mathcal{N}^{n+m},m\ge 2$ is a smooth solution of the mean curvature flow, compact and with positive mean curvature on the maximal time interval $[0, T)$.
\begin{enumerate}
\item If the singularity for $t \rightarrow T$ is of type I, the only possible limiting flows under the rescaling procedure in $(4.4)$ in \cite{HS} are the homothetically shrinking solutions associated with $\mathbb{S}^2, \mathbb{R} \times \mathbb{S}^1$ and $\mathbb{R} \times \Gamma$, where $\Gamma$ is one of the selfsimilar immersed curves introduced by Mullins (see also Abresch--Langer \cite{Abresch1986}).
\item If the singularity is of type II, then from Theorem \ref{Th1}, the only possible blow-up limits at the first singular time are codimension one shrinking round spheres, shrinking round cylinders, and translating bowl solitons.
\end{enumerate}
\end{corollary}
\section{The case of Constant Curvature}
In this chapter, we prove Theorem \ref{Th1} in the case of constant curvature. In the case of constant negative curvature, the proof is more straightforward and more quantitative so we give a direct proof of the statement. 
\subsection{Evolution equations}
\label{sec:evolution_equations}
We start by stating the evolution equations for the length and the squared length of the second fundamental form and the mean curvature vector in the case of constant curvature. We denote $\bar{K}$ to be the sectional curvature and with respect to local orthonormal frames $\{e_i\}$ and $\{\nu_\alpha\}$ for the tangent and normal bundles,
	\begin{align*}
	(\partial_t - \Delta) A_{ij}&=\sum_{p,q,\beta} A_{ij}^{\beta}A_{pq}^{\beta} A_{pq} +\sum_{p,q,\beta}A_{iq}^{\beta}A_{qp}^{\beta} A_{pj}+\sum_{p,q,\beta}A_{jq}^{\beta}A_{qp}^{\beta} A_{pi} - 2 \sum_{p,q,\beta}A_{ip}^{\beta}A_{jq}^{\beta}A_{pq}\\
	&+ 2 \bar K H g_{ij} - n \bar K A_{ij},\\
	(\partial_t - \Delta) H&=\sum_{p,q,\beta} H^{\beta}A_{pq}^{\beta} A_{pq} + n\bar K H.
	\end{align*}
From these equations we can compute
	\begin{align*}
	(\partial_t -\Delta) |A|^2&= - 2 |\nabla A|^2 + 2 | \langle A, A\rangle |^2 + 2 |R^{\perp}|^2 + 4 \bar K |H|^2 - 2 n \bar K |A|^2, \\
	(\partial_t - \Delta) |H|^2&= - 2 |\nabla H|^2 + 2 |\langle A, H\rangle |^2 + 2 n \bar K |H|^2.
	\end{align*}
Here we show the preservation of pinching for high codimension submanifolds of hyperbolic space. To prove the codimension estimate we need good estimates for the reaction terms in this equation. These are proven following Andrews--Baker. 

We assume throughout $H \not = 0$. We first observe
\begin{align*}
\sum_{i,j,p,q}|\langle A_{ij}, A_{pq} \rangle |^2 &= |h|^4 + 2 \sum_{i,j}|h_{ij} A^-_{ij}|^2 + |A^-|^4\\
&= |\circo h|^4 + \frac{2}{n} |\circo h|^2 |H|^2 + \frac{1}{n^2} |H|^4 + 2\sum_{i,j} |h_{ij} A^-_{ij}|^2 + \sum_{i,j,p,q}|\langle A^-_{ij}, A^-_{pq} \rangle |^2
\end{align*}
and recall the identity 
\begin{align*}
|R^\perp|^2 = 2\sum_{i,j,p} |h_{ip}  A^-_{pj} - h_{jp} A^-_{pi}|^2 +\sum_{i,j,p} |A^-_{ip} \otimes A^-_{pj} - A^-_{jp} \otimes A^-_{pi}|^2
\end{align*}
in order to express
\begin{align*}
R_1 &= |\circo h|^4 + \frac{2}{n} |\circo h|^2 |H|^2 + \frac{1}{n^2}|H|^4 + 2\sum_{i,j} |h_{ij} A^-_{ij}|^2 + \sum_{i,j,p,q}|\langle A^-_{ij}, A^-_{pq} \rangle |^2\\
&+2\sum_{i,j,p} |h_{ip}  A^-_{pj} - h_{jp} A^-_{pi}|^2 + \sum_{i,j,p}|A^-_{ip} \otimes A^-_{pj} - A^-_{jp} \otimes A^-_{pi}|^2.
\end{align*}
Andrews--Baker establish the estimate
\begin{align*}
\sum_{i,j}|h_{ij} A^-_{ij}|^2 + \sum_{i,j,p}|h_{ip}  A^-_{pj} - h_{jp} A^-_{pi}|^2 &\leq 2 |\circo h|^2 |A^-|^2,
\end{align*}
and also observe that
\begin{align*}
\sum_{i,j,p,q}|\langle A^-_{ij}, A^-_{pq} \rangle |^2 +\sum_{i,j,p} |A^-_{ip} \otimes A^-_{pj} - A^-_{jp} \otimes A^-_{pi}|^2 &\leq \frac{3}{2} |A^-|^4
\end{align*}
by setting $B^\alpha = A^{-,\alpha}$ for $\alpha\geq 2$ in the following result \cite[Theorem 1]{Li1992}:
\begin{theorem}
\label{thm:algebraic_est}
Let $\{B^\alpha\}$ be a finite set of symmetric $(n\times n)$-matrices. Then we have 
\[\sum_{i,j,\alpha,\beta} (B^\alpha_{ij}B^\beta_{ij})^2 + \sum_{i,j,p,\alpha,\beta} |B^\alpha_{ip}B^\beta_{pj} - B^\alpha_{jp} B^\alpha_{pi}|^2 \leq \frac{3}{2} \Big(\sum_{\alpha} |B^\alpha|^2 \Big)^2.\]
\end{theorem}
Putting these estimates together we obtain the inequality
\begin{align*}
R_1 &\leq |\circo h|^4 + \frac{2}{n} |\circo h|^2 |H|^2 + \frac{1}{n^2}|H|^4 + 4 |\circo h|^2 |A^-|^2 + \frac{3}{2} |A^-|^4.
\end{align*}
We may also expand
\begin{align*}
R_2 &=  \sum_{i,j}|\langle A_{ij}, H\rangle |^2 = |h|^2 |H|^2 = |\circo h|^2 |H|^2 + \frac{1}{n} |H|^4, 
\end{align*}
hence 
\begin{align*}
2R_1 - 2c_n R_2 &\leq 2|\circo h|^4 + 8 |\circo h|^2 |A^-|^2 + 3 |A^-|^4- 2(c_n - 2/n) |\circo h|^2 |H|^2 - \frac{2}{n}(c_n - 1/n)|H|^4.
\end{align*}

Now we express
\[\mathcal Q = |\circo A|^2 - (c_n-1/n) |H|^2 - d_n \bar K\]
and rearrange to obtain
\[|H|^2 = \frac{1}{c_n - 1/n} (|\circo h|^2 + |A^-|^2 - \mathcal Q -d_n \bar K).\]
Substituting this back in gives
\begin{align*}
2R_1 - 2c_n R_2 &\leq 2|\circo h|^4 + 8 |\circo h|^2 |A^-|^2 + 3 |A^-|^4- 2(c_n - 1/n) |\circo h|^2 |H|^2 \\
&- \frac{2}{n} |A^-|^2 |H|^2 + \frac{2}{n}(\mathcal Q+ d_n \bar K) |H|^2\\
&=6 |\circo h|^2 |A^-|^2 + 3 |A^-|^4 - \frac{2}{n} |A^-|^2 |H|^2\\
& + \frac{2}{n}(\mathcal Q+ d_n \bar K) |H|^2 + 2(\mathcal Q+d_n \bar K) |\circo h|^2\\
&=\bigg(6 - \frac{2/n}{c_n - 1/n}\bigg) |\circo h|^2 |A^-|^2 + \bigg(3 - \frac{2/n}{c_n - 1/n} \bigg) |A^-|^4 \\
& + \frac{2}{n}(\mathcal Q+ d_n \bar K) |H|^2 + 2(\mathcal Q+d_n \bar K) |\circo h|^2 + \frac{2/n}{c_n-1/n} (\mathcal Q+d_n \bar K)|A^-|^2.
\end{align*}

The terms on the last line can be written as 
\begin{align*}
\frac{2}{n}(\mathcal Q+ d_n \bar K) |H|^2 &= \frac{2/n}{c_n - 1/n} \mathcal Q (|\circo h|^2 + |A^-|^2  - \mathcal Q) - \frac{4/n}{c_n - 1/n} d_n \bar K \mathcal Q\\
&+ \frac{2/n}{c_n-1/n} d_n \bar K (|\circo h|^2 + |A^-|^2) - \frac{2/n}{c_n - 1/n} d_n^2 \bar K^2,
\end{align*}
hence
\begin{align*}
\frac{2}{n}&(\mathcal Q+ d_n \bar K) |H|^2 + 2(\mathcal Q+d_n \bar K) |\circo h|^2 + \frac{2/n}{c_n-1/n} (\mathcal Q+d_n \bar K)|A^-|^2\\
& = \frac{2/n}{c_n-1/n} \mathcal Q(2 |A^-|^2  - \mathcal Q) - \frac{4/n}{c_n - 1/n} d_n \bar K \mathcal Q + \bigg( \frac{2/n}{c_n - 1/n} + 2\bigg) |\circo h|^2 \mathcal Q\\
& +  \bigg(\frac{2/n}{c_n - 1/n} + 2\bigg) d_n \bar K |\circo h|^2 + \frac{4/n}{c_n - 1/n} d_n \bar K |A^-|^2 - \frac{2/n}{c_n - 1/n} d_n^2 \bar K^2 
\end{align*}
and we have 
\begin{align*}
2R_1 - 2c_n R_2 &\leq \bigg(6 - \frac{2/n}{c_n - 1/n}\bigg) |\circo h|^2 |A^-|^2 + \bigg(3 - \frac{2/n}{c_n - 1/n} \bigg) |A^-|^4 \\
&+  \bigg(\frac{2/n}{c_n - 1/n} + 2\bigg) d_n \bar K |\circo h|^2 + \frac{4/n}{c_n - 1/n} d_n \bar K |A^-|^2 - \frac{2/n}{c_n - 1/n} d_n^2 \bar K^2 \\
& + \frac{2/n}{c_n-1/n} \mathcal Q(2 |A^-|^2  - \mathcal Q) - \frac{4/n}{c_n - 1/n} d_n \bar K \mathcal Q + \bigg( \frac{2/n}{c_n - 1/n} + 2\bigg) |\circo h|^2 \mathcal Q.
\end{align*}
Next we compute 
\begin{align*}
-2n\bar K |\circo A|^2 - 2n\bar K (c_n - 1/n) |H|^2 &= - 4n \bar K |\circo h|^2 - 4n \bar K |A^-|^2 + 2n d_n \bar K^2 + 2n\bar K \mathcal Q, 
\end{align*}
and so obtain the following estimate for the zeroth-order terms in the evolution of $\mathcal Q$: 
\begin{align}
\label{eqn_ReactionEstimage}
2R_1 - 2&c_n R_2 -2n\bar K |\circo A|^2 - 2n\bar K (c_n - 1/n) |H|^2 \notag\\
&\leq \bigg(6 - \frac{2/n}{c_n - 1/n}\bigg) |\circo h|^2 |A^-|^2 + \bigg(3 - \frac{2/n}{c_n - 1/n} \bigg) |A^-|^4 \notag \\
& + 2   \bigg(\frac{d_n/n}{c_n - 1/n} + d_n -2n \bigg) \bar K|\circo h|^2 + 4\bigg(\frac{d_n /n}{c_n - 1/n} -n\bigg) \bar K |A^-|^2 \notag \\
&+2\bigg(n -\frac{d_n/n}{c_n - 1/n} \bigg)d_n \bar K^2 + 2\bigg(1 + \frac{1/n}{c_n - 1/n}\bigg)|\circo h|^2 Q \notag \\
&+  \frac{2/n}{c_n-1/n} \mathcal Q(2 |A^-|^2 - \mathcal Q)+ 2\bigg(n - \frac{2d_n/n}{c_n-1/n}\bigg) \bar K \mathcal Q .
\end{align}
Suppose $\bar K <0$. In this case, if $d_n >0$ then the condition $\mathcal Q \leq 0$ implies $|H|^2 >0$. As above, for $c_n \leq \frac{4}{3n}$ we have
\begin{align*} 
2R_1 - 2&c_n R_2 -2n\bar K |\circo A|^2 - 2n\bar K (c_n - 1/n) |H|^2 \notag\\
&\leq  2   \bigg(\frac{d_n/n}{c_n - 1/n} + d_n -2n \bigg) \bar K|\circo h|^2 + 4\bigg(\frac{d_n /n}{c_n - 1/n} -n\bigg) \bar K |A^-|^2 \notag \\
&+2\bigg(n -\frac{d_n/n}{c_n - 1/n} \bigg)d_n \bar K^2 + 2\bigg(1 + \frac{1/n}{c_n - 1/n}\bigg)|\circo h|^2 Q \notag \\
&+  \frac{2/n}{c_n-1/n} \mathcal Q(2 |A^-|^2 - \mathcal Q)+ 2\bigg(n - \frac{2d_n/n}{c_n-1/n}\bigg) \bar K \mathcal Q.
\end{align*} 
The first term on the left is nonpositive for $d_n \geq 2n - 2/c_n$, and this is also sufficient to ensure 
\[\bigg(\frac{d_n /n}{c_n - 1/n} -n\bigg)\geq 2/c_n -n \geq n,\]
so we have 
\begin{align*} 
2R_1 - 2&c_n R_2 -2n\bar K |\circo A|^2 - 2n\bar K (c_n- 1/n) |H|^2 \notag\\
&\leq 4(2/c_n - n) \bar K |A^-|^2 \notag +2(n -2/c_n)d_n \bar K^2 + 2\bigg(1 + \frac{1/n}{c_n - 1/n}\bigg)|\circo h|^2 Q \notag \\
&+  \frac{2/n}{c_n-1/n} \mathcal Q(2 |A^-|^2 - \mathcal Q)+ 2\bigg(n - \frac{2d_n/n}{c_n-1/n}\bigg) \bar K \mathcal Q.
\end{align*}
All of the terms on the right are either nonpositive or carry a factor $\mathcal Q$, so we see that $\mathcal Q\leq 0$ is preserved for 
\[c_n \leq \min\left\{\frac{4}{3n}, \frac{3}{n+2}\right\}, \qquad d_n \geq 2n - 2/c_n.\]

Observe that for our allowed range of constants $c_n$ and $d_n$,
\[2\bigg(n - \frac{2d_n/n}{c_n-1/n}\bigg) \bar K \geq 0,\]
so when $\mathcal Q\leq 0$ we can further estimate 
\begin{align*} 
2R_1 - 2&c_n R_2 -2n\bar K |\circo A|^2 - 2n\bar K (c_n - 1/n) |H|^2 \notag\\
&\leq  -\frac{2/n}{c_n-1/n} \mathcal Q^2.
\end{align*}
Hence 
\[(\partial_t - \Delta )\mathcal Q \leq -\frac{2/n}{c_n-1/n} \mathcal Q^2,\]
which forces $\mathcal Q$ to blow up in finite time.
\subsection{The evolution of $h$}
From the equations for $A$ and $H$, we have that the projection $\langle A, H\rangle $ satisfies
	\begin{align*}
	(\partial_t - \Delta) A_{ij}^{\alpha} H^{\alpha}&= -2\sum_{p,\alpha}\nabla_p A_{ij}^{\alpha} \nabla_p H^{\alpha} + 2 \sum_{p,q,\alpha,\beta}H^{\alpha} A_{ij}^{\beta}A_{pq}^{\beta} A_{pq}^{\alpha} \\
	&+\sum_{p,q,\alpha,\beta}H^{\alpha}(A_{iq}^{\beta}A_{qp}^{\beta} A_{pj}^{\alpha}+A_{jq}^{\beta}A_{qp}^{\beta} A_{pi}^{\alpha} - 2 A_{ip}^{\beta}A_{jq}^{\beta}A_{pq}^{\alpha})\\
	&+ 2 \bar K |H|^2 g_{ij}.
	\end{align*}
The first of the reaction terms can be split into a hypersurface and a codimension component, as follows:
	\begin{align*}
	2 \sum_{p,q,\alpha,\beta}H^{\alpha} A_{ij}^{\beta}A_{pq}^{\beta} A_{pq}^{\alpha}&= 2 |H| |h|^2 h_{ij} + 2\sum_{p,q,\alpha\ge 2} |H| A_{ij}^{\alpha} A_{pq}^{\alpha} h_{pq}.
	\end{align*}
Similarly, the remaining reaction terms can be written as
	\begin{align*}
	\sum_{p,q,\alpha,\beta}H^{\alpha}(A_{iq}^{\beta}A_{qp}^{\beta} A_{pj}^{\alpha}&+A_{jq}^{\beta}A_{qp}^{\beta} A_{pi}^{\alpha} - 2 A_{ip}^{\beta}A_{jq}^{\beta}A_{pq}^{\alpha})= \sum_{p,q,\alpha\ge 2}|H| A_{iq}^{\alpha} A_{qp}^{\alpha} h_{pj}\\
	& +\sum_{p,q,\alpha\ge 2} |H| A_{jq}^{\alpha} A_{qp}^{\alpha} h_{pi} - 2\sum_{p,q,\alpha\ge 2}|H| A_{ip}^{\alpha} A_{jq}^{\alpha} h_{pq}.
	\end{align*}
Therefore,
	\begin{align*}
	(\partial_t - \Delta) A_{ij}^{\alpha} H^{\alpha}&= -2\sum_{p,\alpha}\nabla_p A_{ij}^{\alpha} \nabla_p H^{\alpha} +2|H| |h|^2 h_{ij} + 2\sum_{p,q,\alpha\ge 2} |H| h_{pq} ( A_{ij}^{\alpha} A_{pq}^{\alpha} - A_{ip}^{\alpha} A_{jq}^{\alpha})\\
	&+ \sum_{p,q,\alpha\ge 2}|H| A_{iq}^{\alpha} A_{qp}^{\alpha} h_{pj} +\sum_{p,q,\alpha\ge 2} |H| A_{jq}^{\alpha} A_{qp}^{\alpha} h_{pi} + 2 \bar K |H|^2 g_{ij}.
	\end{align*}
For a positive function $f$, we have
	\begin{align*}
	(\partial_t-\Delta) \sqrt{f}&=\frac{1}{4 f^{3/2} } |\nabla f|^2 + \frac{1}{2 \sqrt{f} } (\partial_t - \Delta) f,
	\end{align*}
hence the quantity $\sqrt{f} = |H|$ satisfies
	\begin{align*}
	(\partial_t - \Delta) |H|&=\frac{1}{4 |H|^3} |\nabla |H|^2|^2 + \frac{1}{2 |H|}(- 2 |\nabla H|^2 + 2 |\langle A, H\rangle |^2 + 2n\bar K |H|^2)\\
	&= \frac{1}{|H|^3} \langle H, \nabla_i H\rangle \langle H, \nabla_i H\rangle - \frac{|\nabla H|^2}{|H|} + \frac{|\langle A, H\rangle |^2 }{|H|} + n \bar K |H|.
	\end{align*}
Inserting the identities
\[\frac{|\langle A, H\rangle |^2 }{|H|} = |h|^2 |H|\]
and
	\begin{align*}
- \frac{|\nabla H|^2}{|H|}+ \frac{1}{|H|^3}\langle H, \nabla_i H\rangle \langle H, \nabla_i H\rangle&= - \frac{|\nabla |H||^2}{|H|} - |H| |\nabla \nu_1|^2 + \frac{1}{|H|} \langle \nu_1, \nabla_i H\rangle \langle \nu_1, \nabla_i H\rangle\\
&= - |H| |\nabla \nu_1|^2,
	\end{align*}
we obtain
	\begin{align}
	\label{eq:H_evol}
(\partial_t - \Delta) |H|&= |h|^2 |H| + n \bar K |H| - |H||\nabla \nu_1|^2.
	\end{align}
For a tensor $B_{ij}$ divided by a positive scalar function $f$, there holds
	\begin{align*}
	(\partial_t - \Delta) \frac{B_{ij}}{f} = \frac{1}{f} (\nabla_t - \Delta) B_{ij} - \frac{B_{ij}}{f^2} (\partial_t - \Delta) f + \frac{2}{f} \bigg\langle \nabla \frac{B_{ij}}{f}, \nabla f \bigg\rangle.
	\end{align*}
Therefore, dividing $\langle A_{ij}, H\rangle$ by $|H|$, we obtain
	\begin{align*}
	(\partial_t - \Delta) h_{ij}&= |h|^2 h_{ij} + 2\sum_{p,q,\alpha\ge 2} h_{pq} ( A_{ij}^{\alpha} A_{pq}^{\alpha} - A_{ip}^{\alpha} A_{jq}^{\alpha}) +\sum_{p,q,\alpha\ge 2}A_{iq}^{\alpha} A_{qp}^{\alpha} h_{pj} \\
	&+ \sum_{p,q,\alpha\ge 2}A_{jq}^{\alpha} A_{qp}^{\alpha} h_{pi}+ 2 \bar K |H| g_{ij} - n \bar K h_{ij} -2|H|^{-1} \langle \nabla A_{ij} ,\nabla H \rangle + h_{ij} |\nabla \nu_1|^2 \\
	&+2 |H|^{-1} \langle \nabla h_{ij}, \nabla |H| \rangle.
	\end{align*}
We simplify the gradient terms by decomposing
	\begin{align*}
	- 2\langle \nabla A_{ij} , \nabla H\rangle&= - 2\langle \nabla h_{ij} \nu_1 + h_{ij} \nabla \nu_1 + \nabla A_{ij}^- , \nabla |H| \nu_1 + 2|H| \nabla \nu_1 \rangle\\
	&= - 2\langle \nabla h_{ij} , \nabla |H|\rangle - 2|H| h_{ij} |\nabla \nu_1|^2 - 2\langle \nabla A_{ij}^- , \nabla |H| \nu_1 \rangle \\
	&- 2|H| \langle \nabla A_{ij}^- , \nabla \nu_1 \rangle,
	\end{align*}
and so obtain
	\begin{align*}
	(\partial_t - \Delta) h_{ij}&= |h|^2 h_{ij} + 2\sum_{p,q,\alpha\ge 2} h_{pq} ( A_{ij}^{\alpha} A_{pq}^{\alpha} - A_{ip}^{\alpha} A_{jq}^{\alpha}) +\sum_{p,q,\alpha\ge 2} A_{iq}^{\alpha} A_{qp}^{\alpha} h_{pj}\\
	& +\sum_{p,q,\alpha\ge 2} A_{jq}^{\alpha} A_{qp}^{\alpha} h_{pi} + 2 \bar K |H| g_{ij} - n \bar K h_{ij} - h_{ij} |\nabla \nu_1|^2 - 2|H|^{-1}\langle \nabla A_{ij}^- , \nabla |H| \nu_1 \rangle\\
	& - 2 \langle \nabla A_{ij}^- , \nabla \nu_1 \rangle.
	\end{align*}
Next, we compute
	\begin{align*}
	(\partial_t - \Delta) |h|^2&=2\sum_{i,j} h_{ij}(\nabla_t - \Delta) h_{ij} - 2 |\nabla h|^2 \\
	&= 2|h|^4 + 4\sum_{i,j} |h_{ij}A_{ij}^- |^2 - 4\sum_{i,j,p,q,\alpha\ge 2} h_{ij} h_{pq}A_{ip}^{\alpha} A_{jq}^{\alpha} + 4\sum_{i,j,p,q,\alpha\ge 2} h_{ij} h_{pj} A_{iq}^{\alpha} A_{qp}^{\alpha}\\
	& + 4 \bar K |H|^2- 2n \bar K |h|^2- 2|\nabla h|^2 - 2 |h|^2 |\nabla \nu_1|^2 - 4\sum_{i,j}|H|^{-1} h_{ij} \langle \nabla A_{ij}^- , \nabla |H| \nu_1 \rangle \\
	&- 4\sum_{i,j}h_{ij} \langle \nabla A_{ij}^- , \nabla \nu_1 \rangle,
	\end{align*}
and, following Naff, rewrite
	\begin{align*}
	4 \sum_{i,j,p,q,\alpha\ge 2}h_{ij} h_{pj} A_{iq}^{\alpha} A_{qp}^{\alpha} - 4\sum_{i,j,p,q,\alpha\ge 2} h_{ij} h_{pq}A_{ip}^{\alpha} A_{jq}^{\alpha}&= 2\sum_{i,j,p,q}\langle h_{ij} A_{iq}^- - h_{iq}A_{ij}^- , h_{pj} A_{pq}^- - h_{pq} A_{pj}^- \rangle\\
	&= 2\sum_{i,j,p} |h_{ip} A_{pj}^- - h_{jp}A_{pi}^- |^2.
	\end{align*}
Hence,
	\begin{align*}
	(\partial_t - \Delta) |h|^2&= 2|h|^4 + 4\sum_{i,j} |h_{ij}A_{ij}^- |^2+ 2\sum_{i,j,p} |h_{ip} A_{pj}^- - h_{jp}A_{pi}^- |^2 + 4 \bar K |H|^2 - 2n \bar K |h|^2\\
	&- 2|\nabla h|^2 - 2 |h|^2 |\nabla \nu_1|^2 - 4\sum_{i,j}|H|^{-1} h_{ij} \langle \nabla A_{ij}^- , \nabla |H| \nu_1 \rangle - 4\sum_{i,j}h_{ij} \langle \nabla A_{ij}^- , \nabla \nu_1 \rangle,
	\end{align*}
and since $|A^-|^2 = |A|^2 - |h|^2$,
	\begin{align*}
	(\partial_t- \Delta) |A^-|^2 &= 2 | \langle A, A\rangle |^2 -2|h|^4 - 4\sum_{i,j} |h_{ij}A_{ij}^- |^2+ 2 |R^{\perp}|^2 -2\sum_{i,j,p} |h_{ip} A_{pj}^- - h_{jp}A_{pi}^- |^2 \\
	&- 2n \bar K |A^-|^2- 2 |\nabla A|^2+ 2|\nabla h|^2 + 2 |h|^2 |\nabla \nu_1|^2 \\
	&+ 4\sum_{i,j}|H|^{-1} h_{ij} \langle \nabla A_{ij}^- , \nabla |H| \nu_1 \rangle+ 4\sum_{i,j}h_{ij} \langle \nabla A_{ij}^- , \nabla \nu_1 \rangle.
	\end{align*}
The reaction terms can be simplified by observing
\[2 | \langle A, A\rangle |^2 -2|h|^4 - 4\sum_{i,j} |h_{ij}A_{ij}^- |^2 = 2|\langle A^-, A^-\rangle |^2,\]
and (recalling the decomposition of $R^\perp$ carried out above)
\[2 |R^{\perp}|^2 -2\sum_{i,j,p} |h_{ip} A_{pj}^- - h_{jp}A_{pi}^- |^2= 2 \sum_{i,j,p}|h_{ip} A_{pj}^- - h_{jp} A_{pi}^- |^2 + 2\sum_{i,j,p}|A_{ip}^- \otimes A_{pj}^- - A_{jp}^- \otimes A_{pi}^- |^2,\]
hence
	\begin{align*}
	(\partial_t - \Delta) |A^-|^2&= 2|\langle A^-, A^-\rangle |^2 + 2\sum_{i,j,p} |h_{ip} A_{pj}^- - h_{jp} A_{pi}^- |^2 + 2\sum_{i,j,p}|A_{ip}^- \otimes A_{pj}^- - A_{jp}^- \otimes A_{pi}^- |^2\\
	& - 2n \bar K |A^-|^2- 2 |\nabla A|^2+ 2|\nabla h|^2 + 2 |h|^2 |\nabla \nu_1|^2 \\
	&+ 4\sum_{i,j}|H|^{-1} h_{ij} \langle \nabla A_{ij}^- , \nabla |H| \nu_1 \rangle+ 4\sum_{i,j}h_{ij} \langle \nabla A_{ij}^- , \nabla \nu_1 \rangle.
	\end{align*}
Since $\nabla A = \nabla h \nu_1 + h \nabla \nu_1 + \nabla A^-$, we compute
	\begin{align*}
	 2 |\nabla A|^2 = 2 |\nabla h|^2 + 2 |h|^2 |\nabla \nu_1|^2 + 2 |\nabla A^-|^2 + 4\sum_{i,j} h_{ij} \langle \nabla A_{ij}^- ,\nabla \nu_1\rangle + 4\sum_{i,j} \langle \nabla A_{ij}^- , \nabla h_{ij}\nu_1\rangle ,
	\end{align*}
and so obtain
	\begin{align*}
	(\partial_t - \Delta) |A^-|^2&= 2|\langle A^-, A^-\rangle |^2 + 2 \sum_{i,j,p}|h_{ip} A_{pj}^- - h_{jp} A_{pi}^- |^2 + 2\sum_{i,j,p}|A_{ip}^- \otimes A_{pj}^- - A_{jp}^- \otimes A_{pi}^- |^2\\
	& - 2n \bar K |A^-|^2- 2 |\nabla A^-|^2- 4\sum_{i,j} \langle \nabla A_{ij}^- , \nabla h_{ij}\nu_1\rangle+ 4\sum_{i,j}|H|^{-1} h_{ij} \langle \nabla A_{ij}^- , \nabla |H| \nu_1 \rangle.
	\end{align*}
Differentiating $\langle A_{ij}^- , \nu_1\rangle = 0$, we see the last two gradient terms may be expressed as
	\begin{align*}
	&- 4 \sum_{i,j}\langle \nabla A_{ij}^- ,\nabla h_{ij}\nu_1\rangle+ 4\sum_{i,j}|H|^{-1} h_{ij} \langle \nabla A_{ij}^- , \nabla |H| \nu_1 \rangle\\
	&= -\sum_{i,j,k} (4 \nabla_k h_{ij} - 4|H|^{-1} h_{ij} \nabla_k |H|)\langle \nabla_k A_{ij}^- , \nu_1 \rangle \\
	&= \sum_{i,j,k}(4 \nabla_k h_{ij} - 4|H|^{-1} h_{ij} \nabla_k |H|)\langle A_{ij}^- , \nabla_k \nu_1 \rangle ,
	\end{align*}
and consequently,
	\begin{align*}
	(\partial_t - \Delta) |A^-|^2&= 2\sum_{i,j,p,q}|\langle A_{ij}^- , A_{pq}^- \rangle |^2 + 2\sum_{i,j,p} |h_{ip} A_{pj}^- - h_{jp} A_{pi}^- |^2 + 2\sum_{i,j,p}|A_{ip}^- \otimes A_{pj}^- - A_{jp}^- \otimes A_{pi}^- |^2 \\
	&- 2n \bar K |A^-|^2- 2 |\nabla A^-|^2+\sum_{i,j,k} (4 \nabla_k h_{ij} - 4|H|^{-1} h_{ij} \nabla_k |H|)\langle A_{ij}^- , \nabla_k \nu_1 \rangle.
	\end{align*}
Since $f=c_n|H|^2-|A|^2-d_n$ and
	\begin{align*}
	\Big(\partial_t-\Delta\Big)f=2(|\nabla^\bot A|^2-c_n|\nabla^\bot H|^2)+2\Big(c_n\sum_{i,j}|\langle A_{ij},H\rangle|^2-\sum_{i,j,p,q}|\langle A_{ij},A_{pq}\rangle|^2-\sum_{i,j}|R_{ij}^\bot |^2\Big),
	\end{align*}
we have
	\begin{align*}
	&\Big(\partial_t-\Delta\Big)\frac{|A^-|^2}{f}=\frac{1}{f}\Big(\partial_t-\Delta\Big)|A^-|^2-|A^-|^2\frac{1}{f^2}\Big(\partial_t-\Delta\Big)f+2\Big\langle\nabla\frac{|A^-|^2}{f},\nabla \log f\Big\rangle\\
	&=\frac{1}{f}\Big(2\sum_{i,j,p,q}|\langle A_{ij}^- ,A_{pq}^- \rangle|^2+2\sum_{i,j,p}|h_{ip} A_{pj}^- -h_{jp}A_{ip}^- |^2+2\sum_{i,j,p} |A_{ip}^- \otimes A_{jp}^- -A_{jp}^- \otimes A_{ip}^- |^2-2n\bar{K}|A^-|^2\Big)\\
	&+\frac{1}{f}\Big(-2|\nabla^\bot A^-|^2+4\sum_{i,j,k}(\nabla_k h_{ij}-|H|^{-1}h_{ij} \nabla_k |H|)\langle A_{ij}^- ,\nabla^\bot_k \nu_1\rangle\Big)\\
	&-|A^-|^2\frac{1}{f^2}\Big(2(|\nabla^\bot A|^2-c_n|\nabla^\bot H|^2)\Big)\\
	&-|A^-|^2\frac{1}{f^2}\Big(2\Big(c_n\sum_{i,j}|\langle A_{ij},H\rangle|^2-\sum_{i,j,p,q}|\langle A_{ij},A_{pq}\rangle|^2-\sum_{i,j}|R_{ij}^\bot |^2\Big)\Big)\\
	&+2\Big\langle\nabla\frac{|A^-|^2}{f},\nabla\log f\Big\rangle.
	\end{align*}
According to Section $5$, we get
	\begin{align*}	
	\Big(\partial_t-\Delta\Big)\frac{|A^-|^2}{f}\le2\Big\langle\nabla\frac{|A^-|^2}{f},\nabla\log f\Big\rangle-2n\bar{K}\frac{|A^-|^2}{f}-\delta\frac{|A^-|^2}{f^2}\Big(\partial_t-\Delta\Big)f.
	\end{align*}
Note this shows
	\begin{align*}
	\Big(\partial_t-\Delta\Big)\frac{|A^-|^2}{f}\le2\Big\langle\nabla\frac{|A^-|^2}{f},\nabla\log f\Big\rangle-2n\bar{K}\frac{|A^-|^2}{f}.
	\end{align*}
If we have an equation of the form
	\begin{align*}	
		\partial_t u = \Delta u + \langle f ,\nabla u \rangle + Cu
	\end{align*}
by considering $ U = e^{-Ct}u$, we get
	\begin{align*}
		\partial_t U&= e^{-Ct}\partial_t u -Ce^{-Ct}u \\
		&= e^{-Ct} ( \Delta u + \langle f ,\nabla u \rangle + Cu) - C e^{-Ct}u.
	\end{align*}
Hence, we get
	\begin{align*}
	\partial_t U&= \Delta U + \langle f, \nabla U \rangle.
	\end{align*}
By the maximum principle we find $e^{-Ct}\max_{x\in \mathcal M}U(x,t) \leq \max_{x\in\mathcal M}U(x,0)$. Applying this to the above we get
	\begin{align*}
	\frac{|A^-|^2}{f} \leq e^{Ct} C(\mathcal M_0).
	\end{align*}
If we assume $ \bar K \leq 0$, then
	\begin{align*}
	\Big(\partial_t - \Delta\Big) |H|^2&= - 2 |\nabla H|^2 + 2 |\langle A, H\rangle |^2 + 2 n \bar K |H|^2\geq 2 |\langle A, H\rangle |^2 \geq \frac{2}{n}|H|^4.
	\end{align*}
The maximum principle shows
	\begin{align*}
	|H(x,t)|^2 \geq \frac{1}{\frac{1}{\max_{x\in \mathcal{M}_0} |H(x,0)|^2}- \frac{2}{n}t }.
	\end{align*}
Hence, $T_{\max} \leq\frac{n}{2\max_{x\in \mathcal{M}_0} |H(x,0)|^2}$ and thus, we can take
	\begin{align*}
	\frac{|A^-|^2}{f} \leq C(\mathcal M_0).
	\end{align*}
Recall $\Big(\partial_t-\Delta\Big)f$ is non negative at each point in space and time. Let $\sigma=\delta$. We compute
	\begin{align*}
	\Big(\partial_t-\Delta\Big)f^{1-\sigma}&=(1-\sigma)f^{-\sigma}\Big(\partial_t-\Delta\Big)f+\sigma(1-\sigma)f^{-1-\sigma}|\nabla f|^2\\
	&\ge (1-\sigma)f^{-\sigma}\Big(\partial_t-\Delta\Big)f.
	\end{align*}
Then,
	\begin{align*}
	\Big(\partial_t-\Delta\Big)\frac{|A^-|^2}{f^{1-\sigma}}&=\frac{1}{f^{1-\sigma}}\Big(\partial_t-\Delta\Big)|A^-|^2-|A^-|^2\frac{1}{f^{2-2\sigma}}\Big(\partial_t-\Delta\Big)f^{1-\sigma}+2\Big\langle\nabla\frac{|A^-|^2}{f^{1-\sigma}},\nabla\log f^{1-\sigma}\Big\rangle\\
	&\le\frac{1}{f^{1-\sigma}}\Big(\partial_t-\Delta\Big)|A^-|^2-|A^-|^2\frac{1}{f^{2-2\sigma}}(1-\sigma)f^{-\sigma}\Big(\partial_t-\Delta\Big)f\\
	&+2\Big\langle\nabla\frac{|A^-|^2}{f^{1-\sigma}},\nabla\log f^{1-\sigma}\Big\rangle\\
	&=f^\sigma\Big(\frac{1}{f}\Big(\partial_t-\Delta\Big)|A^-|^2-\frac{|A^-|^2}{f^2}\Big(\partial_t-\Delta\Big)f\Big)+\sigma\frac{|A^-|^2}{f^2}f^\sigma\Big(\partial_t-\Delta\Big)f\\
	&+2\Big\langle\nabla\frac{|A^-|^2}{f^{1-\sigma}},\nabla\log f^{1-\sigma}\Big\rangle.
	\end{align*}
Now,
	\begin{align*}
	\frac{1}{f}\Big(\partial_t-\Delta\Big)|A^-|^2-\frac{|A^-|^2}{f^2}\Big(\partial_t-\Delta\Big)f&=\Big(\partial_t-\Delta\Big)\frac{|A^-|^2}{f}-2\Big\langle\nabla\frac{|A^-|^2}{f},\nabla\log f\Big\rangle\\
	&\le-2n\bar{K}\frac{|A^-|^2}{f}-\delta\frac{|A^-|^2}{f^2}\Big(\partial_t-\Delta\Big)f.
	\end{align*}
Therefore,
	\begin{align*}
	\Big(\partial_t-\Delta\Big)\frac{|A^-|^2}{f^{1-\sigma}}&\le -2n\bar{K}\frac{|A^-|^2}{f^{1-\sigma}}-\delta\frac{|A^-|^2}{f^2}f^\sigma\Big(\partial_t-\Delta\Big)f+\sigma\frac{|A^-|^2}{f^2}f^\sigma\Big(\partial_t-\Delta\Big)f\\
	&+2\Big\langle\nabla\frac{|A^-|^2}{f^{1-\sigma}},\nabla\log f^{1-\sigma}\Big\rangle\\
	&=-2n\bar{K}\frac{|A^-|^2}{f^{1-\sigma}}+2\Big\langle\nabla\frac{|A^-|^2}{f^{1-\sigma}},\nabla\log f^{1-\sigma}\Big\rangle.
	\end{align*}
As before, considering $U=e^{-Ct}u$, we get
	\begin{align*}
	\partial_t U&= e^{-Ct}\partial_t u -Ce^{-Ct}u \\
	&= e^{-Ct} ( \Delta u + \langle f ,\nabla u \rangle + Cu) - C e^{-Ct}u.
	\end{align*}
Hence, we get
	\begin{align*}
	\partial_t U&= \Delta U + \langle f, \nabla U \rangle.
	\end{align*}
By the maximum principle we find $e^{-Ct}\max_{x\in \mathcal M}U(x,t) \leq \max_{x\in\mathcal M}U(x,0)$. Applying this to the above we get
	\begin{align*}
	\frac{|A^-|^2}{f^{1-\sigma}} \leq e^{Ct} C(\mathcal M_0).
	\end{align*}
If we assume $ \bar K \leq 0$, then
	\begin{align*}
	(\partial_t - \Delta) |H|^2&= - 2 |\nabla H|^2 + 2 |\langle A, H\rangle |^2 + 2 n \bar K |H|^2\geq 2 |\langle A, H\rangle |^2 \geq \frac{2}{n}|H|^4.
	\end{align*}
The maximum principle shows
	\begin{align*}
	|H(x,t)|^2 \geq \frac{1}{\frac{1}{\max_{x\in \mathcal{M}_0} |H(x,0)|^2}- \frac{2}{n}t }.
	\end{align*}
Hence, $T_{\max} \leq\frac{n}{2\max_{x\in \mathcal{M}_0} |H(x,0)|^2}$ and we can take
	\begin{align*}
	\frac{|A^-|^2}{f^{1-\sigma}} \leq C(\mathcal M_0),
	\end{align*}
which means
	\begin{align*}
	|A^-|^2\le Cf^{1-\sigma},
	\end{align*}
$\forall t\in[0,T)$, Since, $f=c_n|H|^2-|A|^2-d_n\le c_n|H|^2-d_n< c_n|H|^2$, for $d_n>d>0$, this implies
	\begin{align*}
	|A^-|^2< C|H|^{2-2\sigma},
	\end{align*}
which completes the proof.

\end{document}